\documentclass{amsart}
\usepackage[utf8]{inputenc}

\DeclareMathAlphabet{\mathpzc}{OT1}{pzc}{m}{it}

\usepackage{amssymb,amsmath,mathtools,latexsym, amscd, epsfig}
\usepackage{amsthm}
\usepackage{setspace}
\usepackage{graphicx}
\usepackage{mathrsfs}
\usepackage{mdwlist}
\usepackage{color}
\usepackage{bbm,soul}
\usepackage{url}
\usepackage[colorlinks=true]{hyperref}

\usepackage[all,cmtip]{xy}
\usepackage[shortlabels]{enumitem}
\usepackage{todonotes}

\newtheorem{theorem}{Theorem}
\newtheorem{lemma}[theorem]{Lemma}
\newtheorem{definition}[theorem]{Definition}
\newtheorem{proposition}[theorem]{Proposition}

\newtheorem{corollary}[theorem]{Corollary}
\theoremstyle{remark}
\newtheorem{remark}[theorem]{Remark}
\newtheorem{example}[theorem]{Example}

\theoremstyle{definition}
\newtheorem{defin}[theorem]{Definition}

\newtheorem*{claim*}{Claim}

\numberwithin{theorem}{section}
\numberwithin{equation}{section}

\newtheoremstyle{customlabel}{}{}{\itshape}{}{\bfseries}{.}{.5em}{#1 \thmnote{#3}}
\theoremstyle{customlabel}

\newtheoremstyle{named}{}{}{\itshape}{}{\bfseries}{.}{.5em}{\thmnote{#3}}
\theoremstyle{named}

\newcommand{\C}{\mathbb{C}}
\newcommand{\R}{\mathbb{R}}
\newcommand{\Q}{\mathbb{Q}}
\newcommand{\Z}{\mathbb{Z}}
\newcommand{\N}{\mathbb{N}}
\newcommand{\A}{\mathbb{A}}
\newcommand{\Mat}{\mathrm{Mat}}

\newcommand{\Lie}{\mathrm{Lie}}

\newcommand{\G}{\mathbf{G}} % alg. group G
\renewcommand{\H}{\mathbf{H}} % alg. group H
\newcommand{\GL}{\mathrm{GL}}

\newcommand{\SL}{\mathrm{SL}}
\newcommand{\SO}{\mathrm{SO}}
\newcommand{\Spin}{\mathrm{Spin}}
\newcommand{\Aut}{\mathrm{Aut}}

\newcommand{\ord}{\mathrm{ord}}
\newcommand{\lev}{\mathrm{lev}}

\newcommand{\id}{\mathrm{id}}
\newcommand{\Res}{\operatorname{Res}} % restriction of scalars
\newcommand{\disc}{\operatorname{disc}}
\renewcommand{\mod}{\, \mathrm{mod}\,} % mod without weird spacing
\newcommand{\de}{\,\mathrm{d}} % differential
\newcommand{\vol}{\mathrm{vol}} % volume
\newcommand{\height}{\mathrm{ht}} % height
\newcommand{\minht}{\mathrm{minht}}
\newcommand{\content}{\mathsf{c}}
\newcommand{\Ad}{\mathrm{Ad}}
\newcommand{\ad}{\mathrm{ad}}
\newcommand{\red}{\mathrm{red}}
\newcommand{\metric}{\mathrm{d}}
\newcommand{\norm}[1]{\|#1\|}

\newcommand{\Hcal}{\mathcal{H}}

\newcommand{\cfrak}{\mathfrak{c}}

\newcommand{\gfrak}{\mathfrak{g}}
\newcommand{\hfrak}{\mathfrak{h}}

\newcommand{\lfrak}{\mathfrak{l}}
\newcommand{\mfrak}{\mathfrak{m}}

\newcommand{\rfrak}{\mathfrak{r}}
\newcommand{\sfrak}{\mathfrak{s}}
\newcommand{\tfrak}{\mathfrak{t}}

\newcommand{\Gfrak}{\mathfrak{G}}
\newcommand{\Hfrak}{\mathfrak{H}}

\newcommand{\Abf}{\mathbf{A}}

\newcommand{\Cbf}{\mathbf{C}}

\newcommand{\Ebf}{\mathbf{E}}

\newcommand{\Gbf}{\mathbf{G}}
\newcommand{\Hbf}{\mathbf{H}}

\newcommand{\Lbf}{\mathbf{L}}
\newcommand{\Mbf}{\mathbf{M}}
\newcommand{\Nbf}{\mathbf{N}}

\newcommand{\Pbf}{\mathbf{P}}

\newcommand{\Rbf}{\mathbf{R}}

\newcommand{\Tbf}{\mathbf{T}}

\newcommand{\Vbf}{\mathbf{V}}
\newcommand{\Wbf}{\mathbf{W}}

\newcommand{\Zbf}{\mathbf{Z}}

\def\vpz{\mathpzc{v}}
\def\upz{\mathpzc{u}}

\def\wpz{\mathpzc{w}}
\def\xpz{\mathpzc{x}}

\def\zpz{\mathpzc{z}}

\newcommand\data{\mathcal D}
\newcommand\adele{\mathbb A}

\newcommand{\order}{\mathfrak{o}}
\newcommand{\places}{\Sigma}
\newcommand{\Nr}{\mathrm{Nr}}

\newcommand{\jmap}{{\rm j}}
\def\plw{{\mathcal{S}}}

\newcommand{\mcpl}{{\rm mcpl}}

\newcommand{\gen}{\mathrm{gen}}
\newcommand{\spingenus}{\mathrm{spn}}

\newcommand\rquot[2]{
  \mathchoice
  {% \displaystyle
    \text{\raise0.5ex\hbox{$#1$}\big/\lower0.5ex\hbox{$#2$}}%
  }
  {% \textstyle
    #1\,/\,#2
  }
  {% \scriptstyle
    #1\,/\,#2
  }
  {% \scriptscriptstyle
    #1\,/\,#2
  }
}

\newcommand\lquot[2]{
  \mathchoice
  {% \displaystyle
    \text{\lower0.5ex\hbox{$#1$}\big\backslash\raise0.5ex\hbox{$#2$}}%
  }
  {% \textstyle
    #1\,\backslash\,#2
  }
  {% \scriptstyle
    #1\,\backslash\,#2
  }
  {% \scriptscriptstyle
    #1\,\backslash\,#2
  }
}

\newcommand\lrquot[3]{
  \mathchoice
  {% \displaystyle
    \text{\lower0.5ex\hbox{$#1$}\big\backslash\raise0.5ex\hbox{$#2$}\big/
      \lower0.5ex\hbox{$#3$}}%
  }
  {% \textstyle
    #1\,\backslash\,#2\,/\,#3
  }
  {% \scriptstyle
    #1\,\backslash\,#2\,/\,#3
  }
  {% \scriptscriptstyle
    #1\,\backslash\,#2\,/\,#3
  }
}

\newcounter{consta}
\renewcommand{\theconsta}{{A_{\arabic{consta}}}}
\newcommand{\consta}{\refstepcounter{consta}{\color{red}\theconsta}}

\newcounter{constk}
\renewcommand{\theconstk}{{\kappa_{\arabic{constk}}}}
\newcommand{\constk}{\refstepcounter{constk}{\color{red}\theconstk}}

\newcounter{constc}
\renewcommand{\theconstc}{{c_{\arabic{constc}}}}
\newcommand{\constc}{\refstepcounter{constc}{\color{red}\theconstc}}

\newcounter{constE}

\newcounter{constd}

\newcommand{\be}{\begin{equation}}
\newcommand{\ee}{\end{equation}}
\DeclareMathOperator{\diff}{d}

\def\dG{{\dim(\Gbf)}}

\def\hcal{\mathcal{H}}

\newcommand{\extranil}{\mathcal{N}}
\newcommand{\tempered}{\mathsf{T}}
\newcommand{\cpl}{\mathrm{cpl}}

\newcommand{\gp}{p}
\newcommand\gep{\ell}
\newcommand{\gplace}{w}
\newcommand\nth{m}

\newcommand{\smalllattice}{\mathfrak{M}}
\newcommand{\biglattice}{\mathfrak{L}}

\newcommand{\diagsubgrp}{{A}}

\newcommand{\Ganisotropic}{}

\begin{document}

\title[Effective Equidistribution of semisimple adelic periods]{Effective Equidistribution of semisimple adelic periods and representations of quadratic forms}

\author{M.~Einsiedler}
	\address[M. E.]{Departement Mathematik, ETH Z\"urich, R\"amistrasse 101, 8092 Z\"urich,
	Switzerland}
	\email{manfred.einsiedler@math.ethz.ch}

\author{E.~Lindenstrauss}
\address[E.L.]{Institute of Advanced Study, 1 Einstein Drive, Princeton, NJ 08540, USA\newline
\emph{and}\newline
The Einstein Institute of Mathematics, Edmund J. Safra Campus, Givat Ram,
The Hebrew University of Jerusalem, Jerusalem 91904, Israel
}
\email{elonl@ias.edu}

\author{A.~Mohammadi}
\address[A.M.]{Department of Mathematics, University of California, San Diego, CA 92093, USA}
\email{ammohammadi@ucsd.edu}
\thanks{}

\author{A.~Wieser}
\address[A.W.]{Institute of Advanced Study, 1 Einstein Drive, Princeton, NJ 08540, USA}
\email{awieser@ias.edu}

\thanks{M.E.\ acknowledges support by SNF grants 178958 and 10003145.
E.L.\ and A.W.\ acknowledge the support of ERC 2020 grant HomDyn (grant no.~833423).
A.M.\ acknowledges support by NSF grants DMS-2055122 and 2350028. 
A.W.\ was also supported by SNF grant 217944.
}

\begin{abstract}
We prove an effective equidistribution theorem for semisimple closed orbits on compact adelic quotients. The obtained error depends polynomially on the minimal complexity of intermediate orbits and the complexity of the ambient space.
The proof uses dynamical arguments, property $(\tau)$, Prasad's volume formula, an effective closing lemma, and a novel effective generation result for subgroups. The latter in turn relies on an effective version of Greenberg's theorem.

We apply the above to the problem of establishing a local-global principle for representations of integral quadratic forms, improving the codimension assumptions and providing effective bounds in a theorem of Ellenberg and Venkatesh.
\end{abstract}

\maketitle
\setcounter{tocdepth}{1}
\tableofcontents

\section{Introduction}\label{sec:intro}

There is a very fruitful interplay between homogeneous dynamics and number theory.  The origins of this interplay are quite classical. It can be traced back to Minkowski's pioneering work on the Geometry of Numbers, which allowed geometry and implicitly dynamics to be brought to bear on questions in number theory and Diophantine approximations. This connection was further explored e.g.\ in  pioneering work by Artin as well as the remarkable results of Linnik and his school (see e.g.\ \cite{Linnik-book}). 

The results obtained by Dani--Margulis and Ratner on unipotent flows have been particularly influential, stemming from conjectures of Raghunathan and Dani \cite{Dani-Horo2}, which were also implicitly touched upon in the work of Cassels and Swinnerton-Dyer~\cite{Cassels-Swinnerton-Dyer}. 
These conjectures were fully resolved by Ratner in \cite{Ratner-Acta, Ratner-measure, Ratner-topological}, with important special cases proved using a different approach by Dani and Margulis \cite{Margulis-Oppenheim, DM-Oppenheim, DM-MathAnn}. To pass from the measure classification to understanding individual orbits Ratner used important nondivergence estimates of Margulis and Dani \cite{Marg-Nondiv, Dani-divergent, DM-Nondiv}. Particularly powerful has been the combination of Ratner's measure classification results \cite{Ratner-measure} with significant additional techniques known as the Linearization method  that came out of the above mentioned works of Dani and Margulis \cite{DM-Linearization}; cf.\ also Shah's related work \cite{Shah-MathAnn}.

In this paper, our focus is on equidistribution of \emph{periodic orbits of semisimple groups}. Using Ratner's work and the Linearization method, Mozes and Shah showed in \cite{Mozes-Shah} that a sequences of such orbits in a quotient space $\Gamma\backslash G$ (say with volume tending to infinity) either has a subsequence that becomes equidistributed in a homogeneous subspace of $ \Gamma\backslash G$ which contains them (up to small shifts), or escapes to the cusp; unless the sequence become equidistributed in $ \Gamma\backslash G$ this can be explained by algebraic reasons. Our main result in this paper is a fully effective and quantitative and `fully adelic' version of this result when $ \Gamma\backslash G$ is compact. Let $\G$ be a semisimple group defined over a number field $F$. Then the space we work in\footnote{Strictly speaking, this is only true when $\G$ is simply connected; more generally we look at a slight variant of this space.} is $\G(F)\backslash \G(\A)$ with $\A$ the ring of adeles of $F$, and by fully adelic we mean an equidistribution result for a sequence of (translated) orbits corresponding to semisimple  $F$-subgroups of $\G$ that can be quantified solely in terms of global properties of the subgroups and not on how they split at given places.
A non effective adelic equidistribution result was given by Gorodnik and Oh in \cite{GorOh-RationalPts} but it is not fully adelic in our sense as it requires the $F$-subgroups to have no compact factors over a fixed place of $F$. Gorodnik and Oh's work relies on $S$-arithmetic generalizations of Ratner's measure classification theorem by Ratner \cite{Ratner-p-adic} and by Margulis and Tomanov \cite{MargulisTomanov, Margulis:1996p86}; cf.\ also Tomanov's paper~\cite{Tomanov:2000p98}.

Perhaps surprisingly, the powerful dynamical techniques of unipotent flows, whose origins stem from the study of statistical properties of typical trajectories in dynamical systems, have been very successful in studying integer points on arithmetic varieties. 
An early notable examples of this was the work of Eskin, Mozes and Shah \cite{Eskin-Mozes-Shah-1996} on counting the number of integer matrices of norm $<T$ with a given characteristic polynomial.
Gorodnik and Oh in \cite{GorOh-RationalPts} in their paper used $S$-arithmetic unipotent flows to count rational point of bounded height as the height increases.
In this paper, we need dynamics already to show the \emph{existence} of at least one integer point on the relevant arithmetical variety, though our techniques also find the asymptotic number of points. That one can establish existence of integer points on `sufficiently large' varieties using dynamics was shown by Ellenberg and Venkatesh in \cite{EllenbergVenkatesh-localglobal} who established a local-to-global principle for representation of a quadratic form in $m$ variables by a form in $n$ variables, $n \geq m+3$. We strengthen their work by making it effective (hence one can establish the existence of such representation in explicit cases). Moreover by using our fully adelic equidistribution result we can drop certain splitting conditions Ellenberg and Venkatesh had to impose when $n<m+5$. 
The local-to-global principle of \cite{EllenbergVenkatesh-localglobal} as well as the stronger results we prove in this paper rely on the Hasse--Minkowski theorem, a classical local-to-global result that shows that while we do not know apriori that our variety has an integer point, at least we have a rational point on it.

An effective equidistribution result for periodic orbits of a fixed semisimple group~$H$ in quotients of a real group $G$ by a congruence lattice was proved by Margulis, Venkatesh, and one of us (M.E.) in \cite{EMV}. In that work, the semisimple subgroup $H$ was assumed to have finite centralizer in the bigger group $G$, hence the periodic orbits are isolated and do not come in continuous families.
Margulis, Venkatesh, M.E.\ and A.M.\ in \cite{EMMV} proved a fully adelic equidistribution results for maximal semisimple $F$-subgroups $\H$ of an $F$-group $\G$. The fact that the result obtained is not just effective (and quantitative) but also fully adelic allowed the authors of \cite{EMMV} to prove new equidistribution results that were not known even qualitatively earlier. Note that the assumption that $\H$ is maximal is more restrictive than the finite centralizer assumption of \cite{EMV}.

Unfortunately, the finite centralizer assumption is a significant one, and in particular made it impossible to use these tools to provide an effective local-to-global result extending the work of Ellenberg and Venkatesh from \cite{EllenbergVenkatesh-localglobal}.
The finite centralizer assumption was recently removed by A.W. for quotients of real algebraic groups in \cite{AW-realsemisimple}. There are difficulties in extending \cite{AW-realsemisimple} already to the $S$-arithmetic context due to the presence of small subgroups in $p$-adic groups; the fully adelic version we give here requires additional ideas and careful tracking of the dependence of all steps of the proof on the arithmetic complexity of all groups involved.

Finally, we mention that recently there have been significant advances in understanding quantitatively the distribution properties of arbitrary unipotent trajectories, we mention in particular \cite{LMW22,Yang-SL3,LMWY25}. At present, such results are known only in groups of rank at most $2$, but in principle one could envision that modification of that approach could prove e.g.\ results like \cite{EMV, AW-realsemisimple} (and of course can be used to study general orbits, not just periodic orbits). However, these results are inherently not fully adelic in the sense discussed above.

\subsection{Equidistribution of adelic periods}\label{sec:intro equid}

In the following~$F$ will\index{F@$F$, the number field} 
always denote a number field,~$\adele$ will denote\index{a@$\adele, \adele_f$, adeles and finite adeles over~$F$}
the ring of adeles over~$F$, and~$\G$ will be a connected semisimple algebraic $F$-group. We also consider as given a homomorphism with finite central kernel $\rho:\G\to\SL_N$. \index{r@$\rho:\G\rightarrow\SL_N$, a fixed embedding over~$F$}

\begin{quote}
Throughout this paper, we will assume $\G$ is $F$-anisotropic.
\end{quote}\Ganisotropic
We define the compact (ambient) homogeneous space\index{X@$X=\G(F)\backslash\G(\adele)$, the ambient space} 
\begin{align*}
X= [\rho(\G(\A))] = \SL_N(F)\rho(\G(\A)) \subset \lquot{\SL_N(F)}{\SL_N(\A)}.
\end{align*}
Let $\mu_X$ denote the $\rho(\G(\A))$-invariant probability measure on $X$.

We wish to study closed orbits, also known as adelic periods, inside $X$ arising from data $\data=(\mathbf{H},\iota,g)$\index{data@$\data=(\mathbf{H},\iota,g)$, triple defining a ASH} 
consisting of 
\begin{itemize}
\item[(1)] an $F$-algebraic group\index{H@$\mathbf{H}$, the algebraic group giving rise to the ASH~$Y$} $\mathbf{H}$
such that~$\H(F)\backslash \H(\adele)$ has finite volume, 
\item[(2)] an algebraic homomorphism $\iota: \mathbf{H} \rightarrow \SL_N$ defined over $F$ with finite kernel, and
\item[(3)] an element $g \in \SL_N(\A)$. 
\end{itemize}
We call $\mathcal{D}= (\mathbf{H},\iota,g)$ consistent with $(\G,\rho)$ if $\iota(\Hbf) \subset \rho(\G)$ and $g \in \rho(\G(\A))$.

To any data $\data=(\mathbf{H},\iota,g)$, we may associate the \emph{algebraic homogeneous set} 
\[
 Y_{\data} := [\iota(\mathbf{H}(\adele))g] = \SL_N(F)\iota(\mathbf{H}(\adele))g
\]
and the {\em algebraic homogeneous measure} $\mu_{\data}=\mu_{Y_\data}$\index{mu@$\mu_\data, Y_\data$, 
algebraic homogeneous measure and homogeneous set} given by the push-forward, under the map $x \mapsto \iota(x) g,$ 
of the normalized Haar measure on $\mathbf{H}(F) \backslash \mathbf{H}(\adele)$.
We say that $\data$, $Y_\data$, or $\mu_\data$ is simple, semisimple, simply connected, etc.\
according to whether the algebraic group~$\H$ is so.
We note that $\mu_\data$ is invariant under $H=g^{-1}\iota(\H(\A))g$ and supported on the single orbit $Y_\data=\SL_N(F)g H$. 
If $\data$ is semisimple, simply connected, and consistent with $(\G,\rho)$, $Y_\data$ is contained in $X$.

We now define the \emph{complexity} associated to an orbit (slightly adapted from \cite[App.~B]{EMMV}, sometimes also called a discriminant or height),
which will become the basis for
the error rate in the desired equidistribution result.
Let $\data = (\Hbf,\iota,g)$ be data for an algebraic homogeneous measure and let
\begin{align*}
\vpz_\Hbf \in \bigwedge^{\dim(\Hbf)} \iota(\Lie(\H))(F)\subseteq\bigwedge^{\dim(\Hbf)} \mathfrak{gl}_N(F)
\end{align*}
be any non-zero vector.
For every place $v$ of $F$ we let $F_v$ be the completion of $F$ at $v$ and let $\norm{\cdot}_v$ denote the norm on $\bigwedge^{\dim(\Hbf)} \mathfrak{gl}_N(F_v)$ induced by the maximum norm, with respect to the elementary matrices, on $\mathfrak{gl}_N(F_v)$.
Then the \emph{complexity} of $Y_\data$
is defined by
\begin{align}\label{eq:defcomplexity}
\cpl(Y_\data) = \prod_v \norm{g_v^{-1}.\vpz_\Hbf}_v,
\end{align}
where the product is over all places of $F$.
Note that the definition is independent of the choice of vector $\vpz_\Hbf$ and so only depends on $\data$.
In fact, it only depends on the orbit $Y_\data$.

If an algebraic homogeneous set $Y_\data\subset X$ is contained in another such proper subset of intermediate dimension and small complexity,
then $\mu_\data$ cannot be close to $\mu_X$. Therefore we are also concerned with the complexity
of intermediate orbits
and define the \emph{min-complexity} of $Y_\data$ by\Ganisotropic
\[
\mcpl(Y_\data)=\min\bigl\{\cpl(\SL_N(F){\bf M}(\A)g): \iota(\H)\subseteq {\bf M}\subsetneq \rho(\G) \text{ semisimple}\bigr\}.
\]

Recall that a $C^1$-function $f$ on $X$ is a function invariant under a compact open subgroup $K<\G(\A_{f})$ (via $\rho$)
and $C^1$ on the ($\rho(\G(F\otimes \R))$-invariant) manifold $X/\rho(K)$,
where $\A_{f}<\A$ is the ring of finite adeles over $F$. 
If $L\geq 1$ is minimal so that $\rho(K)<\SL_N(\A_{f})$ can be defined using a congruence condition modulo $L$, then we say that $f$ \emph{has level $L$}.
Finally we fix an inner product on $\mathfrak{gl}_N(F)\otimes\R$
and define the $C^1$-norm $\|f\|_{C^1(X)}$ of a smooth function $f$ as the maximum of the sup norms of the function and its partial derivatives in directions corresponding to an orthonormal basis of $\Lie(\G)(F)\otimes\R$.

In this paper we prove the following effective equidistribution result regarding algebraic homogeneous measures: 

\begin{theorem}\label{thm:main equi}
Assume that $\G$ and $\H$ are semisimple simply connected $F$-groups that are also $F$-anisotropic. Let $X=[\rho(\G(\A))]$
and $Y_\data\subset X$ be defined by data $\data=(\H,\iota,g)$ consistent with $(\G,\rho)$. 
Then there exists $\consta\label{A:main}\geq 1$ depending only on $N$ and $[F:\Q]$ so that
\[
\biggl|\int_{Y_\data} f \diff\!\mu_{\data}- \int_X f\diff\!\mu_X\biggr|\ll \frac{\cpl(X)^{\ref{A:main}}}{\mcpl(Y_\data)^{1/\ref{A:main}}} 
\|f\|_{C^1(X)}L
\]
where $f$ is $C^1$, $L$ is the level of $f$, and the implied constant depends only on~$N$, $[F:\Q]$, and polynomially on $|\disc(F)|$.
\end{theorem}  

The above theorem describes equidistribution in the ambient space.
The following theorem describes equidistribution to an intermediate object at polynomial rate in the complexity of the orbit $Y_\data$. Nominally it is a generalization of Theorem~\ref{thm:main equi}, though we deduce it from that theorem.

\begin{theorem}\label{thm:asinEMV}
Assume $\G$ and $\H$ are semisimple connected $F$-anisotropic and that $\H$ is simply connected. 
Let $Y_\data$ be defined by data $\data=(\H,\iota,g)$ consistent with $(\G,\rho)$. 
There exists $\consta\label{a:main2}>0$ depending only on $N$ and $[F:\Q]$ so that the following holds.
For any $B \geq \ref{a:main2}|\disc(F)|^{\ref{a:main2}}\cpl(X)^{\ref{a:main2}}$ there exists a semisimple simply connected data $\data' = (\Hbf',\iota',g)$ consistent with $(\G,\rho)$ and with $\cpl(Y_{\data'}) \leq B$ such that $\iota'(\Hbf') \supset \iota(\Hbf)$ and
\[
\biggl|\int_{Y_\data} f \diff\!\mu_{\data}- \int_{Y_{\data'}}f \de \mu_{\data'}\biggr|
\leq B^{-1/\ref{a:main2}} \|f\|_{C^1(X)}L
\]
for all $C^1$-functions $f$ of level $L\geq 1$. 
\end{theorem}

As mentioned earlier, Gorodnik and Oh \cite{GorOh-RationalPts} have considered sequences of semisimple simply connected adelic periods $[\iota_j(\mathbf{H}_j(\adele))g_j]$ and classified their limit measures \emph{provided} there exists a place $v$ of $F$ for which all of the semisimple $F_v$-groups $\mathbf{H}_j(F_v)$ have no compact factors.
Removing this type of splitting condition already requires proving an effective theorem as the size of the `least splitting place' can increase along the sequence of orbits (the size of the least splitting place was addressed in \cite{EMMV}).

\subsection{Representations of quadratic forms}\label{sec:intro-qf}

A \emph{quadratic lattice} over $\Z$ is a pair $(\Z^n,Q)$ where $Q: \Z^n \to \Z$ is a quadratic form.
A \emph{representation} of a quadratic lattice $(\Z^m,Q)$ by another quadratic lattice $(\Z^n,Q_{0})$ is a linear isometry
\begin{align*}
    \iota: (\Z^m,Q) \to (\Z^n,Q_{0}).
\end{align*}
We say that $\iota$ is \emph{primitive} if $\iota(\Z^m)$ is a primitive lattice in $\Z^n$ (that is, $\iota(\Z^m) = (\iota(\Z^m)\otimes\Q)\cap \Z^n$) and that $(\Z^m,Q)$ is primitively representable by $(\Z^n,Q_{0})$ if a primitive representation exists.
A necessary criterion for primitive representability is that there exist local primitive representations $(\Z_{\gep}^m,Q) \to (\Z_{\gep}^n,Q_{0})$ for any prime $\gep$ as well as a representation of real quadratic spaces $(\R^m,Q) \to (\R^n,Q_{0})$.
In that case, we will call $(\Z^m,Q)$ \emph{locally primitively representable} by $(\Z^n,Q_{0})$.

The (integral) \emph{local-to-global principle} for primitive representations of quadratic lattices asks whether local primitive representability implies (global) primitive representability.
This is an analogy to the Hasse-Minkowski theorem by which local representability of \emph{rational} forms implies global representability.
When $n \geq m+3$ and the quadratic form $Q_0$ on the rank $n$ lattice is indefinite, the local-to-global principle is a consequence of strong approximation for spin groups established by Eichler \cite{Eicher-localglobal}.

When $Q_{0}$ is positive definite, the problem has proven to be significantly harder. 
In the 70's, Hsia, Kitaoka, and Kneser \cite{HsiaKitaokaKneser} established the local-to-global principle for $n \geq 3m+3$ under the (necessary) additional assumption that the minimum $\min(Q) = \min_{0 \neq x \in \Z^m}Q(x)$ is sufficiently large in terms of $Q_{0}$.
J\"ochner and Kitaoka \cite{JochnerKitaoka} improved the assumptions of \cite{HsiaKitaokaKneser} to $n \geq 2m+3$ under some weak conditions on $Q$.

In a later breakthrough, Ellenberg and Venkatesh \cite{EllenbergVenkatesh-localglobal} utilized methods from homogeneous dynamics to show the local-to-global principle for $n \geq m+5$ (and under a Linnik-type splitting condition for $n \geq m+3$) and $\min(Q)$ sufficiently large.
Their result is ineffective as it relies on measure classification results from unipotent dynamics \cite{Ratner-p-adic,MargulisTomanov}.
To start off, they utilize (as do we) the Hasse principle to produce a primitive representation by a form in the spin genus of $Q$.

Here, we prove an effective version of the local-to-global principle for dimensions $n \geq m+3$, also avoiding a Linnik-type splitting condition.

\begin{theorem}\label{thm:local global}
Let $m,n$ be positive integers with $n \geq m+3$.
Then there exist constants $C,A >0$ depending only on $n$ with the following property.
    
Let $(\Z^m,Q),(\Z^n,Q_{0})$ be positive definite quadratic lattices.
Suppose that $(\Z^m,Q)$ is locally primitively representable by $(\Z^n,Q_{0})$ and that
\begin{align*}
\min_{x \in \Z^m\setminus\{0\}} Q(x) \geq C \disc(Q_{0})^A.
\end{align*}
Then $(\Z^m,Q)$ is primitively representable by $(\Z^n,Q_{0})$.
\end{theorem}

For $m=1$ the above theorem encapsulates a local-to-global principle for primitive representations of integers, which is already contained in the literature.
Given codimension at least $3$, this has been established by Kloosterman (for diagonal forms) and by Tartakovskii (see for instance \cite[Ch.~11]{Iwaniec-topics}).
For primitive representations of integers, a local-to-global principle exists in codimension $2$, but is more intricate to formulate --- see the work of Duke and Schulze-Pillot \cite{DukeSchulzePillot} relying on Duke \cite{Duke-hyperbolic} and Iwaniec \cite{Iwaniec-halfintegral}.
Correspondingly, we suspect a version of Theorem~\ref{thm:local global} to hold in codimension $2$ for $m>1$ (for instance, for quadratic lattices primitively representable by the spin genus).

Existing literature (e.g.~the work of Hsia, Kitaoka, and Kneser \cite{HsiaKitaokaKneser}) typically focusses on local-to-global principles for (not necessarily primitive) representations.
Here, the question is whether local representability implies global representability for quadratic lattices with sufficiently large minimum.
This version of the local-to-global principle can be seen to fail in general when the codimension is $3$.
For instance, the form $x_1^2+x_2^2+25x_3^2 + 25x_4^2$ locally represents any positive number, yet does not represent numbers of the form $3\cdot 2^k$ for $k \geq 0$.
For general examples of this form see \cite[p.~144]{Kitaoka-modularformsII}.
Work of Kitaoka \cite{Kitaoka-primitive1,Kitaoka-primitive2,Kitaoka-primitive3} has aimed to compare the two variants of the local-to-global principle (primitive or not necessarily primitive).
In conjuction with the work of Ellenberg and Venkatesh \cite{EllenbergVenkatesh-localglobal} and Theorem~\ref{thm:local global} this implies the local-to-global principle for representations given dimensions $(m,n)$ 
\begin{itemize}
    \item when $m=2$ and $n\geq 6$,
    \item when $m=3,4,5$ and $n\geq 2m+1$, and finally
    \item when $m\geq 6$ and  $n\geq 2m$. 
\end{itemize}
In particular, the counter-examples of Kitaoka, \cite{EllenbergVenkatesh-localglobal} and Theorem \ref{thm:local global} completely resolve the local-to-global principle for representations of binary and ternary forms.
For binary forms, this was already known following work of J\"ochner \cite{Jochner-2in6}.
To the knowledge of the authors it remains unclear in which codimensions the local-to-global principle for representations should be suspected to hold in general (e.g.~does it hold for $(m,n) = (4,8)$?).
We also note that the current work may be extended to representations of bounded imprimitivity --- see Schulze-Pillot \cite{SchulzePillot09}.

We refer to the comprehensive surveys of Schulze-Pillot \cite{SchulzePillot-survey1,SchulzePillot-survey2} for a discussion of known results towards the local-to-global principle that we have omitted here.
In particular, there is a variety of results proving local-to-global principles under assumptions on successive minima of $Q$ --- see for instance \cite{ChanEstesJochner}.

\subsection*{Acknowledgments} The origin of this work is in the joint works by two of us (M.E.\ and A.M.) with Gregory Margulis and Akshay Venkatesh \cite{EMV, EMMV}. Their insight plays an important role in this paper and we thank them for many helpful and illuminating discussions. 
We are grateful to Rainer Schulze-Pillot for his input clarifying to us some issues in the theory of representations of quadratic forms. We would also like to thank Peter Sarnak for his interest and encouragement. E.L.\ would like to thank Michael Larsen for helpful conversations about algebraic groups. We thank the Forschungsinstitut f\"ur Mathematik at ETH Z\"urich, the Hebrew University of Jerusalem, the Institute for Advanced Study, and the Mathematische Forschungsinstitut Oberwolfach for stimulating environments and for supporting mutual visits.

\section{Representations of quadratic forms}

We will now phrase a strengthening of Theorem~\ref{thm:local global} by working over an arbitrary number field and obtaining asymptotics for representation numbers.

Let $F$ be a number field, let $\mathcal{O}_F$ be its ring of integers, and let $(V,Q)$ be a quadratic space over $F$.
A lattice (over $\mathcal{O}_F$) is a finitely generated $\mathcal{O}_F$-module $\biglattice \subset V$ spanning $V$ over $F$.
The rank of $\biglattice$ is $\mathrm{rk}(\biglattice) = \dim(V)$.
A lattice $\biglattice \subset V$ is quadratic if $Q(\biglattice) \subset \mathcal{O}_F$.
(Sometimes, one calls the pair $(\biglattice,Q)$ a quadratic lattice avoiding the ambient space $V$.)
 
In much of the discussion to follow the quadratic form $Q$ or the ambient vector space will sometimes be implicit (though they are a crucial part of the data).
The discriminant $\mathfrak{d}\biglattice$ is the ideal of $\mathcal{O}_F$ generated by the discriminants of the quadratic form on all free submodules of $\biglattice$ of full rank.

The notions from \S\ref{sec:intro-qf} carry over to the current, more general setup without much effort.
For instance, a \emph{representation} $\iota\colon \smalllattice \to \biglattice$ of quadratic lattices is an $\mathcal{O}_F$-linear isometric embedding and it is \emph{primitive} if $\iota(\smalllattice) = \biglattice \cap (\iota(\smalllattice) \otimes F)$.
The terms \emph{primitively representable} and \emph{locally primitively representable} are defined similarly to \S\ref{sec:intro-qf}.

In this setting, strong approximation for spin groups \cite{Eicher-localglobal} also implies the local-to-global principle unless $F/\Q$ is totally real and the quadratic form on the `larger' lattice is totally definite.
Here, a quadratic space $(V,Q)$ is totally definite if the quadratic form $Q\colon V \otimes F_v \to F_v$ is definite for every real place $v$ and we call a quadratic lattice totally definite if its ambient quadratic space is.

Assume in the following that $F/\Q$ is totally real and that $\smalllattice,\biglattice$ are totally definite quadratic lattices.
Let
\begin{align*}
\mathcal{R}(\smalllattice,\biglattice) = \{\iota:\smalllattice \to \biglattice\text{ primitive representation}\}
\end{align*}
be the set of primitive representations of $\smalllattice$ by $\biglattice$ and set $r(\smalllattice,\biglattice) = \#\mathcal{R}(\smalllattice,\biglattice)$.
Let $\Aut(\biglattice)$ be the (finite) group of $\mathcal{O}_F$-linear invertible isometries of the quadratic lattice $\biglattice$.
Note that $\Aut(\biglattice)$ acts on the set $\mathcal{R}(\smalllattice,\biglattice)$ of primitive representations.

Recall that the genus $\gen(\biglattice)$ of $\biglattice$ consists of all lattices in $\biglattice \otimes F$ locally isometric to $\biglattice$.
It is a classical result that the genus consists of finitely many equivalence classes (i.e.~(global) isometry classes) of quadratic lattices; let $\biglattice_1= \biglattice,\biglattice_2,\ldots, \biglattice_g$ be a set of representatives.
Note that the numbers $r(\smalllattice,\biglattice_i)$ and $\#\Aut(\biglattice_i)$ depend only on the equivalence class of $\biglattice_i$.

Define
\begin{align}
\omega_{\biglattice} &= \sum_{i} \#\Aut(\biglattice_i)^{-1}\label{eq:Siegelav-genus}
\end{align}
and
\begin{align}
r(\smalllattice,\gen(\biglattice)) &= \frac{1}{\omega_{\biglattice}} \sum_{i} \frac{r(\smalllattice,\biglattice_i)}{ \#\Aut(\biglattice_i)} \label{eq:Siegelav}.
\end{align}
When $\smalllattice$ is locally primitively representable by $\biglattice$, then it is primitively representable by some element of the genus and hence $r(\smalllattice,\gen(\biglattice))$ is non-zero.
We remark that Siegel's mass formula \cite{Siegel-I,Siegel-II,Siegel-III}  writes the averaged representation number in \eqref{eq:Siegelav} in terms of a local product of representation numbers modulo prime powers as well as archimedean factors involving only the discriminants.

We prove an effective asymptotic for representation numbers.

\begin{theorem}\label{thm:local global2}
Let $n \geq 1$ and $d \geq 1$ be integers.
Then there exists $A>1$ depending on $n,d$ with the following property.

Let $F$ be a totally real number field of degree $d$.
Let $\biglattice$ be a totally definite quadratic lattice of rank $n$ over $\mathcal{O}_F$.
Then for any quadratic space $(W,q)$ of dimension at most $n-3$ and any quadratic lattice $\smalllattice \subset W$ locally primitively representable by $\biglattice$ we have
\begin{align*}
\Big|\frac{r(\smalllattice,\biglattice)}{r(\smalllattice,\gen(\biglattice))} -1 \Big|
\ll_{n,F} \Big(\min_{x \in \smalllattice\setminus\{0\}} |\Nr_{\Q}^F(q(x))| \Big)^{-1/A} |\Nr_{\Q}^F(\mathfrak{d}\biglattice)|^{A}.
\end{align*}
\end{theorem}

Using volume estimates contained in this article and in \cite{EMMV} one can show (see Lemma~\ref{lem:spinrepvsrep} below) that there exists $\delta>0$ depending only on $n$ and $d$ such that 
\begin{align}\label{eq:genrepnumber-lowerbound}
r(\smalllattice,\gen(\biglattice)) \gg_{F,\biglattice} |\Nr_{\Q}^F(\mathfrak{d}\smalllattice)|^\delta.
\end{align}
In particular, Theorem~\ref{thm:local global2} provides an amount of primitive representation polynomial in the discriminant of $\smalllattice$ as soon as one assumes the minimum of $\smalllattice$ to be sufficiently large.
Estimating local densities in Siegel's theorem likely yields much more precise lower bounds in \eqref{eq:genrepnumber-lowerbound}; see also \cite[Lemma 2.1]{SchulzePillot-survey1}.

\begin{remark}[Variants of Theorems~\ref{thm:local global} and \ref{thm:local global2}]
Theorem~\ref{thm:local global2} can be extended in various ways.
For instance, given local primitive representations and $M \in \N$ one may prove asymptotics for the number of primitive representations congruent to these local representations mod $M$ in analogy to e.g.~\cite{JochnerKitaoka}.
Also, the positions of the primitive representations as points on the appropriate Grassmannian variety can be shown to be asymptotically random (see e.g.~the discussion in \cite[\S1.2]{AKA_MUSSO_WIESER_2023}).
\end{remark}

This section is structured as follows: In \S\ref{sec:qf-effequiorbits} we establish an effective equidistribution result for adelic periods of spin groups as a corollary of the theorems in the introduction.
In \S\ref{sec:qf-reformulation}, we recall various facts regarding spin genera and primitive representations and how to view these objects in a setting relevant to this article.
In \S\ref{sec:qf-proofthm}, we prove Theorems~\ref{thm:local global} and \ref{thm:local global2}.

\subsection{Adelic periods of spin stabilizer groups}\label{sec:qf-effequiorbits}

We first deduce an effective equidistribution result for adelic periods of spin groups as a consequence of Theorem~\ref{thm:main equi}.
Let $F$ be a totally real number field with discriminant $\disc(F)$ and ring of integers $\mathcal{O}_F$.
To avoid some of the technicalities, we assume in the following that $Q$ is a totally definite quadratic form on $F^n$,
\begin{align*}
Q(x_1,\ldots,x_n) = \sum_{i\leq j} m_{ij}x_ix_j,
\end{align*}
with coefficients $m_{ij} \in \mathcal{O}_F$.
We set $h_v = \max_{i\leq j} |m_{ij}|_v$ for each place $v$ of $F$ and $h_\infty = \prod_{v \mid \infty} h_v$ as well as $h = \prod_v h_v \leq h_\infty$.

We let $\G = \Spin_Q$ and write $\varrho: \G \to \G'= \SO_Q$ for the standard representation.
For any $F$-subspace $W \subset V$ we write
\begin{align*}
\Hbf_W = \{g \in \G: \varrho(g)w = w \text{ for any }w \in W\}
\end{align*}
for the pointwise stabilizer group of the subspace $W$ under the action of $\G$.

The following corollary captures the input of our dynamical theorems into local-to-global principles discussed in this section.

\begin{corollary}\label{cor:qf-dynthm}
There exist $\consta\label{a:effspinperiods}>0$ depending only on $n$ and $[F:\Q]$ with the following property.

Let $W \subset V$ be an $F$-subspace of codimension at least $3$ and set
\begin{align*}
X = [\varrho(\G(\A))],\quad Y_W = [\varrho(\H_W(\A))].
\end{align*}
Then for any $f \in C^1(X)$ of level $L$
\begin{align*}
\Big| \int_{Y_W}f -\int_X f\Big| \ll \big(\min_{w \in W \cap \mathcal{O}_F^n\setminus\{0\}} \big|\Nr^F_\Q(Q(w))\big| \big)^{-1/\ref{a:effspinperiods}} |\disc(F)|^{\ref{a:effspinperiods}} h_\infty^{\ref{a:effspinperiods}} \norm{f}_{C^1(X)} L.
\end{align*}
\end{corollary}

By Theorem~\ref{thm:main equi} the proof of the corollary will boil down to the following estimates:
\begin{align}\label{eq:qf-ambientvol}
\cpl(X) &\ll h^\star |\disc(F)|^\star,\\
\mcpl(Y_W) &\gg h_\infty^{-\star} |\disc(F)|^{-\star} \Big(\min_{w \in W \cap \mathcal{O}_F^n \setminus\{0\}} \big|\Nr^F_{\Q}(Q(w))\big|\Big)^{\star} \label{eq:qf-mincpl}.
\end{align}
The latter will require statements regarding intermediate groups between $\H_W$ and $\G$ already present in \cite{EllenbergVenkatesh-localglobal}.
For any subgroup $\Mbf < \G$ recall the choice of vector $\vpz_{\Mbf} \in \wedge^{\dim(\Mbf)}\gfrak$ corresponding to the Lie algebra of $\Mbf$ and define $\height(\Mbf) = \prod_v \norm{\vpz_{\Mbf}}_v$.
Throughout we will use the following version of Siegel's lemma over $F$ due to Bombieri and Vaaler.

\begin{theorem}[{\cite[Thm.~9]{BombieriVaaler}}]\label{thm:SiegeloverF}
For $k<\ell$ let $A \in \Mat_{k\ell}(\mathcal{O}_F)$ be of full rank. Then there exist $v_1,\ldots,v_{\ell-k} \in \mathcal{O}_F^\ell$ linearly independent over $F$ with $Av_i = 0$ and
\begin{align*}
\prod_i \height(v_i) 
\ll |\disc(F)|^{\frac{\ell-k}{2[F:\Q]}} \height(a_1\wedge \ldots \wedge a_k)
\end{align*}
where $a_1,\ldots,a_k$ denote the rows of $A$.
\end{theorem}

\begin{proof}[Proof of \eqref{eq:qf-ambientvol}]
We need to estimate the height of the Lie algebra of $\varrho(\Spin_Q) = \SO_Q$ as a point in the appropriate projective space for $\mathfrak{sl}_n$.
Note that $\Lie(\SO_Q)$ is given by
\begin{align*}
\Lie(\SO_Q) = \{X \in \Mat_n: X M_Q + M_Q X^t = 0\}
\end{align*}
where we write $M_Q \in \frac{1}{2}\Mat_{n}(\mathcal{O}_F)$ for the representation matrix.
We may apply Theorem~\ref{thm:SiegeloverF} to find an $F$-basis for $\Lie(\SO_Q)(F)$ consisting of integral vectors $\vpz_1,\ldots,\vpz_d$ with height $\ll |\disc(F)|^\star h^\star$.
In particular, 
\begin{align*}
\cpl(X) = \height(\SO_Q) = \height(\vpz_1\wedge \ldots \wedge \vpz_d) \ll |\disc(F)|^\star h^\star
\end{align*}
as claimed.
\end{proof}

We turn to proving \eqref{eq:qf-mincpl}.
As mentioned earlier, we will require some understanding of intermediate groups between $\Hbf_W$ and $\G$.

\begin{lemma}\label{eq:intermediategroups-qf}
Let $\Lbf \lneq \G $ be a connected semisimple $F$-subgroup so that $\Hbf_W$ is contained in $\Lbf$ but not in any proper normal $F$-subgroup of $\Lbf$.
Then $\Lbf = \Hbf_{W'}$ for a non-trivial $F$-subspace $W' \subset W$.
\end{lemma}

In \cite{EllenbergVenkatesh-localglobal}, Ellenberg and Venkatesh invoke classification results of Guralnick and Saxl \cite{GuralnickSaxl} on groups generated by reflections to obtain a similar statement. An earlier version of the same article on the arXiv contains an elementary argument in codimension at least $7$.
The elementary argument we present here is already present in work of M.E.~ and Wirth \cite[Prop.~3.1]{EinsiedlerWirth}, a variant of which is also contained in work of Lee and Oh \cite[Cor.~3.8]{LeeOh}.

\begin{proof}
We first prove a version of the lemma for Lie algebras over $\R$:
Suppose that $\lfrak < \mathfrak{so}_n(\R)$ is a proper Lie subalgebra so that
\begin{align*}
\hfrak = \Big\{\begin{pmatrix}
\ast_{k,k} & 0_{k,n-k}\\ 0_{n-k,k} & 0_{n-k,n-k}
\end{pmatrix} \Big\} \subset\lfrak
\end{align*}
and $\hfrak$ is not contained in any proper factor of $\lfrak$.
We claim that there exists $k \leq k' < n$ and $g \in \{1_k\} \times \SO_{n-k}(\R)$ such that
\begin{align}\label{eq:intermediatealgR}
\Ad(g)\lfrak = \Big\{\begin{pmatrix}
\ast_{k',k'} & 0_{k',n-k'}\\ 0_{n-k',k'} & 0_{n-k',n-k'}
\end{pmatrix} \Big\}.
\end{align}
To keep the notation simple, we write $\SO_k(\R)$ for the subgroup of $\SO_n(\R)$ with Lie algebra $\hfrak$ and $\SO_{n-k}(\R)$ for the group with Lie algebra the centralizer of $\hfrak$.
The adjoint representation of $\hfrak$ on $\mathfrak{so}_{n}(\R)$ can be decomposed as $\mathfrak{so}_{k}(\R) \oplus \mathfrak{so}_{n-k}(\R) \oplus \rfrak_k$ where
\begin{align*}
\rfrak_k = \Big\{\begin{pmatrix}
0_{k,k} & A\\ -A^t & 0_{n-k,n-k}
\end{pmatrix}: A \in \Mat_{k,n-k}(\R) \Big\}.
\end{align*}
Note that $\rfrak_k \simeq \R^k \otimes \R^{n-k}$ is invariant under $\SO_k(\R) \times \SO_{n-k}(\R)$ where the action on $\R^k \otimes \R^{n-k}$ is the natural one.
Since $\lfrak \cap \rfrak_k$ is invariant under $\SO_k(\R)$,
there exists a subspace $U \subset \R^{n-k}$ so that $\lfrak \cap \rfrak_k = \R^k \otimes U$ under the identification $\rfrak_k \simeq \R^k \otimes \R^{n-k}$.
Here, we used that the standard representation of $\SO_k(\R)$ is irreducible.
As $\SO_{n-k}(\R)$ acts transitively on subspaces of $\R^{n-k}$, there exists $k' \geq k$ and $g \in \SO_{n-k}(\R)$ so that $\Ad(g)\lfrak \cap \rfrak_k \simeq \R^k \otimes (\R^{k'-k} \times \{0_{n-k'}\})$.
In particular, $\Ad(g)\lfrak$ contains the subspace
\begin{align*}
\Big\{\begin{pmatrix}
\ast_{k,k} & A\\ -A^t & 0_{k'-k,k'-k}
\end{pmatrix}: A \in \Mat_{k,k'-k}(\R)\Big\}
\end{align*}
viewed as embedded in the top left corner in the above coordinates on $\mathfrak{so}_n(\R)$.
This implies that $\Ad(g)\lfrak$ in fact contains $\mathfrak{so}_{k'}(\R)$, again embedded in the top left corner.

To see that $\Ad(g)\lfrak = \mathfrak{so}_{k'}(\R)$, notice that the adjoint representation of $\mathfrak{so}_n(\R)$ restricted to $\mathfrak{so}_{k'}(\R)$ may be decomposed in precisely the same manner as for $\hfrak$.
Moreover, $\Ad(g)\lfrak$ cannot intersect the complement $\rfrak_{k'}$ to $\mathfrak{so}_{k'}(\R)\oplus \mathfrak{so}_{n-k'}(\R)$. Indeed, any intersection would yield a vector with non-zero entry in the last $n-k'$ entries of the top row (by irreducilibity of the standard representation of $\mathfrak{so}_{k'}(\R)$) contradicting the choice of $k'$.
This proves $\lfrak \subset \mathfrak{so}_{k'}(\R) \oplus \mathfrak{so}_{n-k'}(\R)$ and thus \eqref{eq:intermediatealgR} since $\hfrak$ is not contained in any proper factor of $\lfrak$.

We turn to proving the lemma.
Recall that $Q$ was assumed to be totally definite. For any real place $w$ of $F$, \eqref{eq:intermediatealgR} shows that $\lfrak_w = \Lie(\Lbf)(F_w)$ is the Lie algebra of the stabilizer of some $F_w$-subspace of $W \otimes F_w$.
Since $\lfrak$ is defined over $F$, this shows $\Lie(\Lbf) = \Lie(\Hbf_{W'})$ for an $F$-subspace $W' \subset W$.
This proves the lemma.
\end{proof}

\begin{proof}[Proof of \eqref{eq:qf-mincpl}]
Let $\Mbf< \G$ be a proper semisimple $F$-subgroup of $\G$ with $\H_W \subset \Mbf \subsetneq \G$.
By e.g.~\cite[Lemma 8.6]{AW-realsemisimple}, any normal $F$-subgroup of $\Mbf$ has height controlled polynomially by the height of $\Mbf$.
We may hence, for simplicity, replace $\Mbf$ with the minimal normal $F$-subgroup of $\Mbf$ containing $\Hbf_W$.

By Lemma~\ref{eq:intermediategroups-qf}, there exists a non-trivial $F$-subspace $W' \subset W$ with $\Mbf = \Hbf_{W'}$.
By a version of Minkowski's second theorem --- see \cite[Thm.~8]{BombieriVaaler} --- there exists an $F$-basis $\vpz_1,\ldots,\vpz_d$ of $\Lie(\H_{W'})$ consisting of integral vectors with
\begin{align*}
\prod_i \height(\vpz_i) \ll |\disc(F)|^\star \height(\H_{W'}).
\end{align*}
Also, observe that $W'$ is the unique subspace pointwise stabilized by $\H_{W'}$ and so
\begin{align*}
W' = \{w \in F^n: \mathrm{D}\varrho(\vpz_i) w = 0 \text{ for all } i\}.
\end{align*}
By Theorem~\ref{thm:SiegeloverF} we have
\begin{align*}
\height(\H_{W'}) \gg\height(W')^\star |\disc(F)|^{-\star}.
\end{align*}
By Minkowski's first theorem (a sufficient version of which is implied by Theorem~\ref{thm:SiegeloverF}), there exists a vector $w \in W'$ with $\height(w) \ll |\disc(F)|^\star \height(\Hbf_{W'})^\star$.
Replacing $w$ by a multiple and using Minkowski's bound, we may assume that
\begin{align*}
[W \cap \mathcal{O}_F^n:\mathcal{O}_F w] \ll |\disc(F)|^{\frac{1}{2}}.
\end{align*}
In particular, $\prod_{v \text{ finite}} \norm{w}_v \gg |\disc(F)|^{-\frac{1}{2}}$. 
We also have 
\begin{align*}
\prod_{v \mid \infty} \norm{w}_v \gg h_\infty^{-1} \prod_{v \mid \infty} |Q(w)|_v^{\frac{1}{2}} = h_\infty^{-1} \big|\Nr^F_\Q(Q(w))\big|^{\frac{1}{2}}.
\end{align*}
Putting all these estimates together shows 
\begin{align*}
\height(\Mbf) = \height(\H_{W'}) 
&\gg \height(w)^\star |\disc(F)|^{-\star}
\gg \prod_{v \mid \infty} \norm{w}_v^\star |\disc(F)|^{-\star}\\
&\gg h_\infty^{-\star} |\disc(F)|^{-\star} \big|\Nr^F_\Q(Q(w))\big|^\star.
\end{align*}
This proves the claim.
\end{proof}

\begin{proof}[Proof of Corollary~\ref{cor:qf-dynthm}]
In view of \eqref{eq:qf-ambientvol} and \eqref{eq:qf-mincpl} this is a direct consequence of Theorem~\ref{thm:main equi}.
\end{proof}

\subsection{Spin genera, primitive representations, and adelic orbits}\label{sec:qf-reformulation}

The goal of this subsection is to describe the sets at the beginning of the section (the genus, the set of primitive representations etc) in terms of adelic periods.
While this is mostly standard, we do so from first principles for the readers' convenience.

\subsubsection{Some notation}
As before, $F$ is a totally real number field with ring of integers $\mathcal{O}_F$. For any finite place $v$ of $F$ we write $\order_v$ for the ring of integers of $F_v$.
Let $(V,Q)$ be an $n$-dimensional totally definite quadratic space over $F$. 
We write $\G'=\SO_{Q}$ for the special orthogonal group of ${Q}$, $\G = \Spin_{Q}$ for the spin group of ${Q}$ (the simply connected cover of $\SO_{Q}$), and $\varrho: \G \to \G'$ for the standard representation.
For any subspace $U \subset V$ we let
\begin{align*}
\Hbf_U = \{g \in \G: \varrho(g)u = u \text{ for all }u \in U\}
\end{align*}
be the pointwise stabilizer group of the subspace $U$ under the action of $\G$.
Similarly, we define $\H_U'<\G'$ to be the pointwise stablizer group in $\G'$.
One may identify $\Hbf_U$ (resp.~$\H'_U$) with the spin group (resp.~special orthogonal group) of the restriction of the quadratic form $Q$ to the orthogonal complement of $U$.

\subsubsection{The genus and the spin genus}
Recall that the standard action of $\GL_n(\A_f)$ on $\mathcal{O}_F$-lattices $\biglattice\subset V$ is given by letting $g_\ast \biglattice$ be the unique lattice with completions $g_v (\biglattice \otimes \order_v)$ at every finite place $v$ of $F$.
Explicitly, 
\[
g_\ast \biglattice = \bigcap_{v \nmid \infty} \big((g_v \biglattice\otimes \order_v)  \cap V\big).
\]
For brevity, we set $\biglattice\otimes \order_v = \biglattice_v$.

Two quadratic lattices are equivalent (or isometric) if they are in the same $\G'(F)$-orbit.
The genus of a quadratic lattice $\biglattice \subset V$ is the orbit 
\begin{align*}
\gen(\biglattice) = \G'(\A_f)_\ast \biglattice.
\end{align*}
The genus consists of finitely many equivalence classes.
If $F=\Q$ and $Q(\biglattice) \subset \Z$, the genus of $\biglattice$ is naturally identified with the genus of the quadratic form $Q|_{\biglattice}$. Indeed, the quadratic form restricted to a lattice in $\SO_{Q}(\A_f)_\ast \biglattice$ is locally equivalent to the quadratic form on $\biglattice$. 

The spin genus of a quadratic lattice $\biglattice \subset V$ is
\begin{align}\label{eq:spingenus}
\spingenus(\biglattice) = \G'(F)\varrho(\G(\A_f))_\ast \biglattice
\end{align}
where we note that $\varrho(\G(\A_f))$ is normalized by $\G'(F)$.

\subsubsection{Primitive representations}
Suppose now that $\smalllattice$ is a quadratic lattice of rank $m\leq n$ in another quadratic space $(W,q)$.

\begin{lemma}\label{lem:spinrepresentable}
If $m \leq n-3$ and $\smalllattice$ is locally primitively representable by a lattice $\biglattice \subset V$, then $\smalllattice$ is primitively representable by some $\biglattice' \in \spingenus(\biglattice)$.
That is, there exists an $F$-subspace $U \subset V$ isometric to $W$ such that $\biglattice' \cap U$ is isometric to $\smalllattice$.
\end{lemma}

\begin{proof}
This is a well-known fact contained e.g.~in \cite[p.~2]{Hsia-spinorgenera} and \cite[Lemma 2]{SchulzePillot09}. Here, we replicate the proof of \cite[Lemma 2]{SchulzePillot09} in our notation for the readers' convenience.
By the Hasse-Minkowski theorem \cite[Thm.~66:3]{OMeara}, there exists a linear isometry $\iota: W \to V$. Set $U = \iota(W)$ and
\begin{align*}
\mathcal{S} = \{v \text{ finite place of }F: \iota(\smalllattice_v) \neq (U\otimes F_v) \cap \biglattice_v\}.
\end{align*}
Note that $\mathcal{S}$ is finite.
By assumption, $\smalllattice$ is locally primitively representable by $\biglattice$ and so there exists for any $v \in \mathcal{S}$ a linear isometry $\iota_v': \smalllattice_v \to \biglattice_v$ with $\iota_v'(\smalllattice_v)= (\iota_v'(\smalllattice_v) \otimes F_v) \cap \biglattice_v$.
By Witt's extension theorem (cf.~\cite[p.~21]{Cassels-book}, \cite[\S42F]{OMeara}), we may further find $g_v \in \G'(F_v)$ such that $g_v\iota_v'(w) = \iota(w)$ for any $w \in W\otimes F_v$.
Thus, for any $v \in \mathcal{S}$ we have $\iota(\smalllattice_v) = (U \otimes F_v) \cap g_v\biglattice_v$. If $\biglattice' \in \gen(\biglattice)$ is the lattice with $\biglattice_v' = \biglattice_v$ for all $v \not\in \mathcal{S}$ and $\biglattice_v' = g_v\biglattice_v$ for all $v \in \mathcal{S}$, we have shown that $\smalllattice$ is primitively representable by $\biglattice'$.

We adapt $g_v$, which is uniquely determined up to left-multiples with $\H'_U(F_v)$, so that $g_v \in \varrho(\G(F_v))$.
Since $m \leq n-3$, the spinor norm on both $\G'(F_v)$ and $\H'_U(F_v)$ is surjective (cf.~\cite[Thm.~91:6]{OMeara}) and we obtain exact sequences 
\begin{align*}
\G(F_v) &\to \G'(F_v)  \to F_v^\times/(F_v^\times)^2 \to 1,\\
\H_U(F_v) &\to \H'_U(F_v) \to F_v^\times/(F_v^\times)^2 \to 1.
\end{align*}
In particular, for every $v \in \mathcal{S}$ there exists $h_v \in \H'_U(F_v)$ with $h_v g_v \in \varrho(\G(F_v))$.
The lattice $\biglattice'$ with $\biglattice_v' = \biglattice_v$ for all $v \not\in \mathcal{S}$ and with $\biglattice_v' = h_vg_v\biglattice_v$ for all $v \in \mathcal{S}$ belongs to $\spingenus(\biglattice)$ and primitively represents $\smalllattice$. Thus, the lemma follows.
\end{proof}

From now on, we will always assume $m \leq n-3$.

In the following, we view for a given lattice $\biglattice\subset V$ primitive representations of $\smalllattice$ by the spin genus of $\biglattice$ as tuples $(\biglattice',\iota)$ where $\biglattice' \in \spingenus(\biglattice)$ and $\iota\colon W \to V$ is an isometry with $\iota^{-1}(\biglattice') = \smalllattice$.
The group $\G'(F)$ acts on the set of primitive representations by the spin genus via $\gamma.(\biglattice',\iota) = (\gamma \biglattice',\gamma \iota)$ (where $(\gamma \iota) (w) = \gamma \iota(w)$).
We write $[(\biglattice',\iota)] = \G'(F).(\biglattice',\iota)$ for the $\G'(F)$-equivalence class and define the following sets:
\begin{itemize}
    \item $\tilde{\mathcal{R}}(\smalllattice,\spingenus(\biglattice))$ is the set of $\G'(F)$-equivalence classes of primitive representations of $\smalllattice$ by the spin genus of $\biglattice$.
    \item $\mathcal{R}(\smalllattice,\biglattice)$ is the set of primitive representations of $\smalllattice$ by $\biglattice$.
    \item $\tilde{\mathcal{R}}(\smalllattice,\biglattice)$ is the set of primitive representations of $\smalllattice$ by $\biglattice$ up to the action of $\Aut(\biglattice)$ (where $\Aut(\biglattice)$ is the finite group of linear isometries of $\biglattice$).
\end{itemize}
In the above notation, whenever $\spingenus(\biglattice) = \bigsqcup_i \G'(F)_\ast \biglattice_i$ then
\begin{align*}
\tilde{\mathcal{R}}(\smalllattice,\spingenus(\biglattice)) \simeq
\bigsqcup_i \tilde{\mathcal{R}}(\smalllattice,\biglattice_i).
\end{align*}
By Lemma~\ref{lem:spinrepresentable}, if $\smalllattice$ is locally primitively representable by $\biglattice$, then $\tilde{\mathcal{R}}(\smalllattice,\spingenus(\biglattice))$ is non-empty.
There is an obvious forgetful map 
\begin{align*}
\pi: \tilde{\mathcal{R}}(\smalllattice,\spingenus(\biglattice)) \to \G'(F) \backslash \spingenus(\biglattice).
\end{align*}
In this phrasing, the conclusion of Theorem~\ref{thm:local global} corresponds to showing that $\tilde{\mathcal{R}}(\smalllattice,\biglattice_i)$ is non-empty for every $i$ or, equivalently, that $\pi$ is surjective.

\begin{lemma}\label{lem:primitive representations}
Let $(\biglattice,\iota)$ be a primitive representation of $\smalllattice$. 
The map
\begin{align}\label{eq:locallyisoemb}
\big\{\H'_{\iota W}(F)\text{-orbits in } \H'_{\iota W}(F)\varrho(\Hbf_{\iota W}(\A_f))_\ast \biglattice\big\} \to \tilde{\mathcal{R}}(\smalllattice,\spingenus(\biglattice))
\end{align}
given by $(\H'_{\iota W}(F)h)_\ast \biglattice \mapsto [(h_\ast \biglattice,\iota)]$ is injective.

Moreover, a class of primitive representations $[(\biglattice',\iota')]$ belongs to the image of \eqref{eq:locallyisoemb} if and only if there exist $\gamma \in \G'(F)$ and $g \in \G(\A_f)$ such that $\biglattice' = (\gamma \varrho(g))_\ast \biglattice$ and $\iota' = \gamma \varrho(g_v) \iota$ for every finite place $v$ of $F$.
\end{lemma}

By construction, the composition of \eqref{eq:locallyisoemb} with $\pi$ is given by the natural map
\begin{align}\label{eq:qf-inclusionHtoG}
\H'_{\iota W}(F) \backslash \H'_{\iota W}(F)\varrho(\Hbf_{\iota W}(\A_f))_\ast \biglattice 
\to \G'(F)\backslash\G'(F)\varrho(\G(\A_f))_\ast \biglattice.
\end{align}

The lemma yields an equivalence relation $\sim$ on $\tilde{\mathcal{R}}(\smalllattice,\spingenus(\biglattice))$.
(Following Kneser \cite{Kneser-Darstellungsmasse} one could call the classes under $\sim$ `spin genus classes of primitive representations'.)
We note that the claim for the image of \eqref{eq:locallyisoemb} will not be used for the proof of Theorem~\ref{thm:local global}.

\begin{proof}
For any $\biglattice' \in \H'_{\iota W}(F)\varrho(\Hbf_{\iota W}(\A_f))_\ast \biglattice$ we have $\biglattice' \cap \iota W = \iota(\smalllattice)$ and hence the map in \eqref{eq:locallyisoemb} is well-defined.
If $[(\biglattice_1,\iota)] = [(\biglattice_2,\iota)]$ then there exists $\gamma \in \G'(F)$ with $\gamma \biglattice_1 = \biglattice_2$ and $\gamma \iota = \iota$; the latter implies $\gamma \in \H'_{\iota W}(F)$ and so injectivity follows.

For the claim regarding the image, note that any image under \eqref{eq:locallyisoemb} has the desired property by construction.
Conversely, suppose there exist $g \in \G(\A_f)$ and $\gamma \in \G'(F)$ such that $\biglattice' = (\gamma \varrho(g))_\ast \biglattice$ and $\iota' = \gamma \varrho(g_v) \iota$ for every finite place $v$ of $F$.
By Witt's extension theorem (cf.~\cite[p.~21]{Cassels-book}, \cite[\S42F]{OMeara}), there is $\gamma' \in \G'(F)$ such that $\gamma' \iota' = \iota$. Replacing the representative of the class we may thus assume $\iota' = \iota$.
Thus, $\gamma \varrho(g) \in \H'_{\iota W}(\A_f)$. 
By \cite[101:8]{OMeara} using $m\leq n-3$, the spinor norms $\G'(F) \to F^\times/(F^\times)^2$ and $\H'_{\iota W}(F)\to F^\times/(F^\times)^2$ surject onto the set of totally positive elements.
We may thus write $\gamma = \gamma_1\varrho(\gamma_2)$ where $\gamma_1 \in \H'_{\iota W}(F)$ and $\gamma_2 \in \G(F)$. 
Replacing $g$ with $\gamma_2 g$ we can hence assume that $\gamma \in \H'_{\iota W}(F)$ and $g \in \H_{\iota W}(\A_f)$. 
This proves the lemma.
\end{proof}

\begin{remark}
The map in \eqref{eq:locallyisoemb} need not be surjective.
For an example, one may construct a non-unimodular quadratic lattice $\biglattice$ and two subspaces $W_1,W_2$ where $\biglattice \cap W_1$ and $\biglattice \cap W_2$ are isometric but $\biglattice \cap W_1^\perp$ and $\biglattice \cap W_2^\perp$ are not.
\end{remark}

Since \eqref{eq:qf-inclusionHtoG} factors through \eqref{eq:locallyisoemb}, the following corollary is immediate.

\begin{corollary}\label{cor:surjectivityspingenus}
If $(\biglattice,\iota)$ is a primitive representation of $\smalllattice$ such that \eqref{eq:qf-inclusionHtoG} is surjective, then $\smalllattice$ is primitively representable by any element of the spin genus of~$\biglattice$.
\end{corollary}

Corollary~\ref{cor:surjectivityspingenus} provides a clear path to proving Theorem~\ref{thm:local global}. Indeed, let $(\biglattice,\iota)$ be a primitive representation of $\smalllattice$ and set
\begin{align}
K_f &= \{g \in \G'(F)\varrho(\G(\A_f)): g_\ast \biglattice = \biglattice\},\nonumber\\ 
K &= \{g \in \G'(F)\varrho(\G(\A)): (g_f)_\ast \biglattice = \biglattice\} = \varrho(\G(F \otimes \R))K_f \label{eq:qf-maxcpt}
\end{align}
where we used $\varrho(\G(F \otimes \R)) = \G'(F \otimes \R)$ for the last equality; the latter holds since $\G'(F \otimes \R)$ is connected.

Then $\smalllattice$ is primitively representable by any lattice in the spin genus of~$\biglattice$ if the homogeneous set $[\varrho(\H_{\iota W}(\A))]$ intersects any of the finitely many (right-)$K$-orbit in $[\varrho(\G(\A))]$.
Thus, effective equidistribution (or rather effective density) of $[\varrho(\H_{\iota W}(\A))]$ can yield Theorem~\ref{thm:local global}.

\subsubsection{Determining the weights}\label{sec:weights}
Let $\biglattice \subset V$ be a quadratic lattice and suppose that there is a primitive representation $(\biglattice,\iota)$ of $\smalllattice$ by $\biglattice$.
Define $K$ as in \eqref{eq:qf-maxcpt}.
In the following we compute
\begin{enumerate}[a)]
    \item the measure of $K$-orbits on $[\varrho(\G(\A))]$ and
    \item the measure of $K \cap (\H'_{\iota W}(F)\varrho(\H_{\iota W}(\A)))$-orbits on $[\varrho(\H_{\iota W}(\A))]$
\end{enumerate}
for the Haar probability measures on the respective orbits.
Note that only a) (or rather a lower bound on the measures in a)) is needed to establish Theorem~\ref{thm:local global} as a corollary of Theorem~\ref{thm:main equi}.

We start with a).
Let $\nu_\G$ be the Haar measure on the group $  \G'(F)\varrho(\G(\A))$ which projects to the probability measure $\bar{\nu}_\G$ on the quotient $[\varrho(\G(\A))]$.
The measure of an orbit $[\varrho(g)K]$ is
\begin{align*}
\bar{\nu}_\G\big([\varrho(g)K]\big) = \frac{\nu_\G(K)}{\#\Aut(\biglattice')}
\end{align*}
where $\biglattice' = \varrho(g_f)_\ast \biglattice \in \spingenus(\biglattice)$, and recall that $\Aut(\biglattice') = \{g \in \G'(F): g_\ast \biglattice' = \biglattice'\}$.

Using the fact that $\bar{\nu}_\G$ is a probability measure we have $\nu_\G(K)^{-1} = \omega_{\spingenus(\biglattice)}$ where
\begin{align}\label{eq:omegaspin}
\omega_{\spingenus(\biglattice)} := \sum_{\G'(F)_\ast \biglattice' \subset \spingenus(\biglattice)} \#\Aut(\biglattice')^{-1}.
\end{align}
Thus,
\begin{align*}
\bar{\nu}_\G([\varrho(g)K]) = \frac{1}{\omega_{\spingenus(\biglattice)} \#\Aut(\biglattice')}.
\end{align*}

For b), set $K_\iota = K \cap (\H'_{\iota W}(F)\varrho(\H_{\iota W}(\A)))$ (which implicitly also depends on $\biglattice$).
Let $\nu_\iota$ be the Haar measure on $\H'_{\iota W}(F)\varrho(\H_{\iota W}(\A))$ that descends to the probability measure on $[\varrho(\H_{\iota W}(\A))]$.
In view of the map \eqref{eq:locallyisoemb} in Lemma~\ref{lem:primitive representations}, any $K_\iota$-orbit in $[\varrho(\H_{\iota W}(\A))]$ corresponds to a class $[(\biglattice',\iota)]$ of primitive representations in $\tilde{\mathcal{R}}(\smalllattice,\spingenus(\biglattice))$. 
The measure of this orbit is
\begin{align}\label{eq:qf-weightsHorbit}
\frac{\nu_\iota(K_\iota)}{\#(\Aut(\biglattice') \cap \H'_{\iota W}(F))}
= \frac{1}{\omega_{[(\biglattice,\iota)]}\cdot \#(\Aut(\biglattice') \cap \H'_{\iota W}(F))}
\end{align}
where $\omega_{[(\biglattice,\iota)]}$ is the normalizing constant 
\begin{align}\label{eq:qf-Htotvol}
\omega_{[(\biglattice,\iota)]}= \nu_\iota(K_\iota)^{-1} = \sum_{[(\biglattice',\iota)]\sim [(\biglattice,\iota)]} \frac{1}{\#(\Aut(\biglattice') \cap \H'_{\iota W}(F))}.
\end{align}
Here, we used the equivalence relation $\sim$ defined after Lemma~\ref{lem:primitive representations}.

Suppose now that $\biglattice_0\subset V$ is a quadratic lattice.
When $\{\biglattice_i\}$ is a set of representatives for $\G'(F)$-orbits in the spin genus of $\biglattice_0$, we have by unfolding
\begin{equation}\label{eq:avrepnumbers}
\begin{aligned}
\sum_{[(\biglattice,\iota)] \in \tilde{\mathcal{R}}(\smalllattice,\spingenus(\biglattice_0))/\sim} \omega_{[(\biglattice,\iota)]}
&= \sum_{[(\biglattice,\iota)]}\frac{1}{\#(\Aut(\biglattice) \cap \H'_{\iota W}(F))}\\
&= \sum_i \sum_{[(\biglattice_i,\iota)]} \frac{1}{\#(\Aut(\biglattice_i) \cap \H'_{\iota W}(F))} \\
&= \sum_i \sum_{[(\biglattice_i,\iota)]} \frac{\#(\Aut(\biglattice_i).(\biglattice_i,\iota))}{\#\Aut(\biglattice_i)} \\
&= \sum_i \frac{r(\smalllattice,\biglattice_i)}{\#\Aut(\biglattice_i)}.
\end{aligned}
\end{equation}
For future use, we set
\begin{align*}
r(\smalllattice,\spingenus(\biglattice_0)) = \frac{1}{\omega_{\spingenus(\biglattice_0)}}\sum_i \frac{r(\smalllattice,\biglattice_i)}{\#\Aut(\biglattice_i)}.
\end{align*}

\begin{lemma}\label{lem:spinrepvsrep}
When $m\leq n-3$ have
\begin{align}\label{eq:spinrepvsrep}
r(\smalllattice,\spingenus(\biglattice_0)) 
= r(\smalllattice,\gen(\biglattice_0)).
\end{align}
Moreover, if $\biglattice_0$ locally primitively represents $\smalllattice$ then
\begin{align*}
r(\smalllattice,\spingenus(\biglattice_0)) \gg_{F,\biglattice_0} |\Nr_{\Q}^F(\mathfrak{d}\smalllattice)|^\star.
\end{align*}
\end{lemma}

Recall that $r(\smalllattice,\gen(\biglattice_0))$ was defined in \eqref{eq:Siegelav} in a similarly to $r(\smalllattice,\spingenus(\biglattice_0))$ using the full genus.

\begin{proof}
The comparison of averaged numbers of representations by different spinor genera in a genus in \eqref{eq:spinrepvsrep} is due to Kneser \cite{Kneser-Darstellungsmasse} and Weil \cite[p.~473]{Weil-qf}.
These works also verify that $\omega_{\spingenus(\biglattice_0)}$ depends only on the genus of $\biglattice_0$.

Since we do not strictly speaking require the lower bound on $r(\smalllattice,\spingenus(\biglattice_0))$, we shall be brief in proving it.
It suffices to bound $\omega_{[(\biglattice,\iota)]}$ from below for any primitive representation $(\biglattice,\iota)$ by the spin genus of $\biglattice_0$.
Notice that by \eqref{eq:qf-Htotvol}
\begin{align*}
\omega_{[(\biglattice,\iota)]} = \nu_\iota\Bigl(K \cap (\H'_{\iota W}(F)\varrho(\H_{\iota W}(\A)))\Bigr)^{-1}
\end{align*}
where $K$ is as in \eqref{eq:qf-maxcpt} and $\nu_\iota$ is the Haar measure on the group $\H_{\iota W}'(F)\varrho(\H_{\iota W}(\A))$ which descends to the Haar probability measure on the quotient 
\begin{align*}
Y = [\varrho(\H_{\iota W}(\A))].
\end{align*}
In other words, $\omega_{[(\biglattice,\iota)]} = \widetilde{\vol}(Y)$ where the adapted volume $\widetilde{\vol}$ is defined in \cite[\S5.12]{EMMV}.
By \cite[\S5.12, App.~B]{EMMV}, we have $\widetilde{\vol}(Y) \gg_{F,\mathcal{L}_0} \cpl(Y)^\star$ where one uses reduction theory to realize $Y$ appropriately as a subset of $\SL_n(F)\backslash\SL_n(\A)$ (cf.~\S\ref{sec:reductiontheoryqf} below).
As in \S\ref{sec:qf-effequiorbits}, one may verify $\cpl(Y) \gg_{F,\biglattice_0} |\Nr_{\Q}^F(\mathfrak{d}\smalllattice)|^\star$ proving the lemma.
\end{proof}

\subsection{Proof of Theorems~\ref{thm:local global} and \ref{thm:local global2}}\label{sec:qf-proofthm}

The goal of this section is to prove Theorems~\ref{thm:local global} and \ref{thm:local global2}.
To make our arguments accessible, we will prove Theorem~\ref{thm:local global} first despite it being a direct corollary to Theorem~\ref{thm:local global2}.
The proof of Theorem~\ref{thm:local global2} will warrant a discussion of reduction theory for quadratic forms over the totally real number field $F$; we will recall these facts in \S\ref{sec:reductiontheoryqf}.

Throughout, $\biglattice \subset V$ and $\smalllattice \subset W$ are quadratic lattices and $\smalllattice$ is locally primitively representable by $\biglattice$.

\subsubsection{Proof of Theorem~\ref{thm:local global}}
Recall that $F=\Q$ is assumed in Theorem~\ref{thm:local global}. 
We fix $\biglattice_0 \in \spingenus(\biglattice)$ which primitively represents $\smalllattice$ (this is possible by Lemma~\ref{lem:spinrepresentable}).
Let $(\biglattice_0,\iota_0)$ be such a primitive representation.

By Minkowski's reduction theory of quadratic forms over $\Q$ (see e.g.~\cite[Ch.~12]{Cassels-book}), there exist linearly independent vectors $v_1,\ldots,v_n \in \biglattice_0$ spanning a sublattice of index $O_n(1)$ in $\biglattice_0$ such that
\begin{align*}
Q(v_1) \leq \ldots \leq Q(v_n)
\end{align*}
and $|(v_i,v_j)_Q| \leq Q(v_i)$ for $i < j$. In particular, $\disc(Q) \asymp \prod_i Q(v_i)$.
We use this basis to identify $V$ with $\Q^n$ and hence identify the isogeny $\varrho$ with a homomorphism $\varrho:\G = \Spin_Q \to \SL_n$.
Set
\begin{align*}
X = [\varrho(\G(\A))],\quad Y =  [\varrho(\H_{\iota_0W}(\A))].
\end{align*}
In this setting, Corollary~\ref{cor:qf-dynthm} yields that for any $f \in C^1(X)$ of level $L$
\begin{align}\label{eq:prooflocglob1}
\Big| \int_{Y}f -\int_X f\Big| \ll \big(\min_{w \in \smalllattice\setminus\{0\}} q(w)\big)^{-\star} \disc(Q|_{\biglattice_0})^\star\norm{f}_{C^1(X)} L^\star.
\end{align}
Note that $D := \disc(Q|_{\biglattice_0}) = \disc(Q|_{\biglattice})$ does not depend on the element of the genus of $\biglattice$.

Since $\biglattice_0 \in \spingenus(\biglattice)$ there exists $g_0 \in \G'(\Q)\varrho(\G(\A_f))$ with $(g_0)_\ast \biglattice_0 = \biglattice$.
As in \eqref{eq:qf-maxcpt}, let
\begin{align*}
K_f &= \{g \in \G'(\Q)\varrho(\G(\A_f)): g_\ast \biglattice_0 = \biglattice_0\},\
K = \varrho(\G(\R))K_f.
\end{align*}
By Lemma~\ref{lem:primitive representations}, the theorem follows if we can show that $Y$ intersects the open subset $[g_0K]$ of $X$ (see the comments after Corollary~\ref{cor:surjectivityspingenus}).
Let $f$ be its indicator function.
Note that $f$ is locally constant at the real place and invariant under $K_f$.
In particular, it is of level $O_n(1)$ recalling that the integral structure is defined using $v_1,\ldots v_n$.
By a) in \S\ref{sec:weights} we have
\begin{align*}
\int_X f = \frac{1}{\omega_{\spingenus(\biglattice)}\cdot \#\Aut(\biglattice)}
\end{align*}
where $\omega_{\spingenus(\biglattice)}$ is as in \eqref{eq:omegaspin}.
We have $\#\Aut(\biglattice) \ll 1$. 
In view of the above discussion of reduction theory of integral quadratic forms, $\omega_{\spingenus(\biglattice)} \ll D^\star$. 
Indeed, this follows from reduction theory as 
there are $\ll d^\star$ many forms over $\Z$ with coefficients bounded by $d$.
Overall, this shows that $\int_X f \gg D^{-\star}$.
Thus, if $\int_Y f = 0$ we have by \eqref{eq:prooflocglob1}
\begin{align*}
D^{-\star} \ll \int_X f \ll \big(\min_{w \in \smalllattice\setminus\{0\}} q(w)\big)^{-\star}D^\star
\end{align*}
and hence $\min_{w \in \smalllattice\setminus\{0\}} q(w) \ll D^{\star}$.
This proves the theorem.
\qed

\subsubsection{Reduction theory for quadratic forms over number fields}\label{sec:reductiontheoryqf}
The classical reduction theory over $\Q$ by Minkowksi \cite{Cassels-book} has been extended to quadratic forms over number fields by Humbert \cite{Humbert1,Humbert2} (see also \cite[\S9--11]{Koecher}).
Here, we briefly recall this theory in a slightly refined form (as our lattices are not free); the reader interested in the case $F=\Q$ may skip this subsection.

Let $\biglattice \subset V$ be a quadratic lattice.
We set
\begin{align*}
\min(\biglattice) = \min_{x \in \biglattice\setminus\{0\}} |\Nr^F_\Q(Q(x))|.
\end{align*}
Since $\biglattice$ is not necessarily free, the statements to follow are slightly more intricate than usual phrasings of reduction theory over $F$.
We refer to e.g.~\cite[Thm.~3.11]{ChanIcaza} for the following discussion.
By Minkowski's bound and the classification of finitely generated modules over Dedekind domains, there exists a free submodule $\Lambda \subset \biglattice$ such that
\begin{align*}
[\biglattice:\Lambda] \ll |\disc(F)|^{\star}.
\end{align*}
In fact, there exists $C_F>0$ (depending only on $F$) and such a free sublattice $\Lambda= \mathcal{O}_F \vpz_1 + \ldots + \mathcal{O}_F v_n$ so that $C_F\min(\biglattice) \geq \min(\Lambda)$, and the quadratic form $Q$ in the basis $v_1,\ldots,v_n$ is given by
\begin{align}\label{eq:Lagrangeexp}
a_1(x_1+b_{12}x_2+\ldots)^2 + a_2 (x_2 + b_{23}x_3+\ldots)^2 + \ldots + a_n x_n^2
\end{align}
for some $a_i,b_{ij} \in F$ with the following properties.
\begin{itemize}
    \item $|\Nr^F_{\Q}(a_1)| \leq C_F\min(\biglattice)$.
    \item For any $v \mid \infty$ and any $i<j$ we have $|b_{ij}|_v \leq C_F$.
    \item For any $i < j$ we have 
    \begin{align*}
        |\Nr^F_{\Q}(a_i)| \leq C_F |\Nr^F_{\Q}(a_j)|.
    \end{align*}
    \item For any $v,w\mid \infty$ and any $i$ we have $|a_i|_v \leq C_F |a_i|_w$.
\end{itemize}
In particular, applying these estimates in \eqref{eq:Lagrangeexp} the matrix representation $(m_{ij})$ of the quadratic form $Q|_\Lambda$ (in the basis provided above) satisfies for any place $v \mid \infty$
\begin{align}\label{eq:redthqf-boundcoeff}
|m_{ij}|_v \ll_F |\disc(Q|_\Lambda)|_v^\star \ll_F |\Nr^F_\Q(\mathfrak{d}\biglattice)|^\star.
\end{align}
Here, recall that $\mathfrak{d}\biglattice$ is the ideal of $\mathcal{O}_F$ generated by the discriminants of free sublattices of $\biglattice$ and that $[\biglattice:\Lambda] \ll |\disc(F)|^{\star}$.

We also record the following corollary of the above discussion.

\begin{corollary}\label{cor:boundgenus}
For any quadratic lattice $\biglattice \subset V$
\begin{align*}
\#(\G'(F)\backslash \spingenus(\biglattice)) \ll_F |\Nr^F_\Q(\mathfrak{d}\biglattice)|^\star.
\end{align*}
\end{corollary}

\begin{proof}
The number of quadratic forms over $\mathcal{O}_F$ in $n$ variables whose coefficients $(m_{ij})$ satisfy $|m_{ij}|_v \leq h$ for all $v \mid \infty$ is $\ll h^\star$.
Write $\G'(F)\backslash \spingenus(\biglattice) = \bigsqcup_i \G'(F)_\ast\biglattice_i$ and apply the reduction theory from above to find for any $i$ a free sublattice $\Lambda_i \subset \biglattice_i$ with the required properties.
The number of $\G'(F)$-inequivalent lattices $\Lambda_i$ is $\ll_F |\Nr^F_\Q(\mathfrak{d}\biglattice)|^\star$.
If $\Lambda_i = \gamma_\ast\Lambda_j$ for some $\gamma \in \G'(F)$, then 
\begin{align*}
[\biglattice_i: \biglattice_i \cap \gamma_\ast \biglattice_j],
[\gamma_\ast\biglattice_j: \biglattice_i \cap \gamma_\ast \biglattice_j] \ll_F 1.
\end{align*}
In particular, there are $\ll_F 1$ such $j$'s for any given $i$. This proves the corollary.
\end{proof}

\subsubsection{Proof of Theorem~\ref{thm:local global2}}
We proceed similarly to the above proof of Theorem~\ref{thm:local global}.
By Lemma~\ref{lem:spinrepresentable}, $\smalllattice$ is primitively representable by an element of the spin genus of $\biglattice$.
For the time being, fix a pair $(\biglattice_0,\iota_0)$ where $\biglattice_0 \in \spingenus(\biglattice)$ and $\iota_0$ is a primitive representation of $\smalllattice$.
Note that we will later have to vary this choice over all equivalence classes of the relation~$\sim$ introduced in Lemma~\ref{lem:primitive representations}.

By reduction theory as in \S\ref{sec:reductiontheoryqf}, let $\Lambda_0 \subset \biglattice_0$ be a free sublattice with a basis $v_1,\ldots,v_n$ satisfying the required coefficient bounds.
We identify $V$ with $F^n$ through this basis.
In particular, we view $\varrho$ as a homomorphism $\varrho: \G \to \SL_n$.
Set
\begin{align*}
X = [\varrho(\G(\A))],\quad
Y = [\varrho(\H_{\iota_0 W}(\A))].
\end{align*}
By Corollary~\ref{cor:qf-dynthm} we thus have for any $f \in C^1(X)$ of level $L$
\begin{align*}
\Big|\int_X f -\int_Y f\Big| \ll_F \min(\Lambda_0 \cap \iota_0(W))^{-\star} |\Nr^F_\Q(\mathfrak{d}\biglattice_0)|^\star \norm{f}_{C^1(X)} L.
\end{align*}
Notice that $\Nr^F_\Q(\mathfrak{d}\biglattice_0) = \Nr^F_\Q(\mathfrak{d}\biglattice)$ and
\begin{align*}
\min(\Lambda_0 \cap \iota_0(W)) 
\asymp_F \min(\biglattice_0 \cap \iota_0(W))
= \min(\smalllattice)
\end{align*}
so that
\begin{align}\label{eq:prooflocglob2}
\Big|\int_X f -\int_Y f\Big| 
\ll_F \min(\smalllattice)^{-\star} |\Nr^F_\Q(\mathfrak{d}\biglattice)|^\star \norm{f}_{C^1(X)} L.
\end{align}
As in \eqref{eq:qf-maxcpt} let
\begin{align*}
K_f &= \{g \in \G'(F)\varrho(\G(\A_f)): g_\ast \biglattice_0 = \biglattice_0\},\
K = \varrho(\G(\R))K_f.
\end{align*}
Since $\biglattice_0 \in \spingenus(\biglattice)$ there exists $g \in \G'(F)\varrho(\G(\A_f))$ with $g_\ast \biglattice_0 = \biglattice$.
We apply the estimate in \eqref{eq:prooflocglob2} to the characteristic function $f$ of $[gK]$.
The function $f$ is locally constant at all archimedean places and, since $[\biglattice_0:\Lambda_0] \ll_F 1$, the level $L$ of $f$ satisfies $L \ll_F 1$.
By a) in \S\ref{sec:weights} we have
\begin{align*}
\int_X f = \frac{1}{\omega_{\spingenus(\biglattice)} \#\Aut(\biglattice)}.
\end{align*}
Moreover, by \eqref{eq:qf-weightsHorbit}
\begin{align}\label{eq:qf-integraldiffHorbits}
\int_Y f = 
\frac{1}{\omega_{[(\biglattice_0,\iota_0)]}} \sum_{\substack{[(\biglattice',\iota_0)] \in \pi^{-1}(\G'(F).\biglattice)\\
[(\biglattice',\iota_0)] \sim [(\biglattice_0,\iota_0)]
}}
\frac{1}{\#(\Aut(\biglattice') \cap \H'_{\iota_0 W}(F))}
\end{align}
where, as was the case earlier,
\begin{align*}
\pi: \tilde{\mathcal{R}}(\smalllattice,\spingenus(\biglattice_0)) \to \G'(F) \backslash \spingenus(\biglattice_0).
\end{align*}
is the forgetful map.
Inserting these expressions into \eqref{eq:prooflocglob2} one obtains an asymptotic for primitive representations belonging to the equivalence class for $\sim$ of $(\biglattice_0,\iota_0)$.
Notice that the rate does not depend on the equivalence class and, hence, the same effective asymptotic holds for any average.

We take the convex combination of \eqref{eq:qf-integraldiffHorbits} 
using the weights $W^{-1}\omega_{[(\biglattice_0,\iota_0)]}$ for all equivalence classes $[(\biglattice_0,\iota_0)]$ of $\sim$. 
By \eqref{eq:avrepnumbers} we have
\begin{align*}
W=\sum_{[(\biglattice_0,\iota_0)] \in \tilde{\mathcal{R}}(\smalllattice,\spingenus(\biglattice))/\sim} \omega_{[(\biglattice_0,\iota_0)]}
= \omega_{\spingenus(\biglattice)}r(\smalllattice,\spingenus(\biglattice)).
\end{align*}
Moreover, by the same calculation in \eqref{eq:avrepnumbers} (for a fixed $\biglattice_i$) the weighted average of the right-hand side of \eqref{eq:qf-integraldiffHorbits} yields
\begin{align*}
\frac{r(\smalllattice,\biglattice)}{\omega_{\spingenus(\biglattice)}r(\smalllattice,\spingenus(\biglattice))\#\Aut(\biglattice)}.
\end{align*}
In summary, the above analysis shows with \eqref{eq:prooflocglob2}
\begin{align*}
\Big| \frac{1}{\omega_{\spingenus(\biglattice)} \#\Aut(\biglattice)} - \frac{r(\smalllattice,\biglattice)}{\omega_{\spingenus(\biglattice)}r(\smalllattice,\spingenus(\biglattice))\#\Aut(\biglattice)} \Big|
\ll_F \min(\smalllattice)^{-\star} |\Nr^F_\Q(\mathfrak{d}\biglattice)|^\star.
\end{align*}
Since $\#\Aut(\biglattice) \ll 1$ and, by Corollary~\ref{cor:boundgenus}, $\omega_{\spingenus(\biglattice)} \ll_F |\Nr^F_\Q(\mathfrak{d}\biglattice)|^\star$ we deduce 
\begin{align*}
\Big|1 - \frac{r(\smalllattice,\biglattice)}{r(\smalllattice,\spingenus(\biglattice))} \Big|
\ll_F \min(\smalllattice)^{-\star} |\Nr^F_\Q(\mathfrak{d}\biglattice)|^\star
\end{align*}
which proves the theorem since $r(\smalllattice,\spingenus(\biglattice)) = r(\smalllattice,\gen(\biglattice))$ (cf.~\eqref{eq:spinrepvsrep}).
\qed

\section{Notation and preliminaries}\label{sec: notation gp}

\subsection{Setup}\label{sec:setup}
We will prove Theorems~\ref{thm:main equi}--\ref{thm:asinEMV} first over $\Q$ and then deduce the version over general number fields. 
In particular, most of the article will assume $F=\Q$.
Let $\places$ and $\places_f$ denote the set of places, respectively finite places of $\Q$.

For $\nth\geq 0$ and every $\gep\in\places_{f}$, let 
\[
\SL_{N}(\Z_\gep)[k]=\mathrm{ker}\big(\SL_N(\Z_\gep) \rightarrow \SL_N(\Z_\gep/\gep^k \Z_\gep)\big)
\]
be the $k$-th congruence subgroup of $\SL_N(\Z_\gep)$.
For any closed subgroup $M<\SL_N(\Q_\gep)$ and any integer $k\geq 0$, we put $M[k]:= M\cap\SL_{N}(\Z_\gep)[k]$ and refer to this as the principal subgroup of level $k$ in $M$. 
When $k\geq 0$ is not an integer, we set $M[k] = M[\lceil k\rceil]$ for simplicity of notation.

Let $\G$ be simply connected semisimple $\Q$-group and let $\rho\colon \G \rightarrow \SL_N$ be an algebraic homomorphism defined over $\Q$\index{r@$\rho:\G\rightarrow\SL_N$, a fixed homomorphism over~$\Q$ with central kernel} with central kernel.
We define 
\[
K_\gep=\rho^{-1}(\SL_N(\Z_\gep)),
\]
and let $K_f=\prod_{\gep\in\places_{f}}K_\gep$. Also, set
\begin{equation} \label{eq:Kvm definition}
K_\gep[k] := \mathrm{ker}(K_\gep \rightarrow \SL_N(\Z_\gep/\gep^k \Z_\gep))
\end{equation}
for $k \geq 1$. 
It is convenient to write $K_\gep[0] := K_\gep$ and $G_\gep = \rho(\G(\Q_\gep))$.

Given $g \in \SL_N(\A)$ we write $[g] \in \SL_N(\Q)\backslash \SL_N(\A)$ for the corresponding coset and, similarly, we write $[B]$ for the set of cosets represented by a subset $B \subset~\SL_N(\A)$.

We consider data $\data = (\H,\iota,g)$ where $\H$ is a simply connected semisimple $\Q$-group, $\iota: \H \to \SL_N$ an algebraic homomorphism defined over $\Q$ with central kernel, and $g_\data \in \SL_N(\A)$.
Similarly to the introduction, we assume that the data $\data$ is consistent with the pair $(\G,\rho)$ i.e.~$\iota(\H) < \rho(\G)$ and $g_\data \in \rho(\G(\A))$.
We set
\begin{align*}
X = [\rho(\G(\A))],\quad
Y_{\data} = [\iota(\H(\A))g_{\data}].
\end{align*}
Then $Y_{\data}$ is invariant under the group 
 \[
 H_\data=g_{\data}^{-1}\iota(\H(\adele))g_{\data},
 \]
and for every $\gep\in\places$, we put 
\begin{align*}
H_\gep=g_{\data,\gep}^{-1}\iota(\H(\Q_\gep))g_{\data,\gep}.
\end{align*}
When there is no confusion, we denote $H_\data$ simply by $H$.

We also set up the corresponding notions at the level of the Lie algebra $\mathfrak{g}$
of~$\G$. \index{g@$\mathfrak{g},\mathfrak{g}_\gep$, Lie algebra and Lie algebra over
local field~$\Q_\gep$}
For any~$\gep\in\Sigma$, we let~$\mathfrak{g}_\gep$ be the Lie algebra of~$\G$ over~$\Q_\gep$.
For~$\gep\in\Sigma_f$, $\mathfrak{g}_\gep[0]$ denotes the preimage of the $\Z_\gep$-integral 
$N \times N$ matrices under the differential ${\rm D}\rho: \mathfrak{g} \rightarrow \mathfrak{sl}_N$.
More generally, we write $\gfrak_\gep[]$ for the preimage of the matrices all of whose entries have valuation at least $k$.
\index{g@$\mathfrak{g}_\gep[k]$, compact open subgroup
of level~$k$ of~$\mathfrak g_\gep$ for~$v\in\Sigma_f$}
For any subspace $V < \gfrak_\gep$ we put $V[k] = V \cap \gfrak_\gep[k]$.
For convenience, we will usually identify $\gfrak$ with its image under the isomorphism $\mathrm{D}\rho$ and similarly for its subalgebras.
We will also view $\gfrak$ as a linear subvariety of $\mathfrak{sl}_N$.

Throughout, $\red_\gep\colon\SL_N(\Z_\gep)\rightarrow\SL_N(\mathbb F_\gep)$ 
denotes the reduction map mod $\gep$; 
similarly we consider reduction mod $\gep$ for the Lie algebras, 
see~\cite[Ch.~3]{PlRa} for a discussion of reduction maps.\index{r@$\gep$, reduction maps}  

For $g \in \G(\Q_\gep)$, we write $\|g\|_\gep$, or simply $\|g\|$ if there is no confusion, for the largest absolute value of the matrix entries of 
$\rho(g)$ and $\rho(g)^{-1}$. For an element $g\in\G(\adele)$ we write 
\[
\|g\|=\max\{\|g_\gep\|_\gep:\gep\in\places\}.
\]
\index{g@$\|g\|$, norm of~$g\in\G(\Q_\gep)$ and $g\in\G(\adele)$.}
Notice that $\|g\|=\|g^{-1}\|$.
Moreover, for any $g\in\G(\adele)$ we define $\height(g)=\prod_{\gep\in\places}\|g_\gep\|$.
\index{g@$\height(g)$, height of~$g\in\G(\adele)$}

The content of a vector $\wpz \in (\bigwedge^k \mathfrak{sl}_N)(\A)$ (or $\wpz \in \A^N$) is
\begin{align}\label{eq:def-content}
\content(\wpz) = \prod_{\ell\in \Sigma} \norm{\wpz_\ell}_\ell.
\end{align}
Here, for $\gep$ a prime the norm $\norm{\cdot}_\ell$ is the largest absolute value of the entries of $\wpz$ (in an integral basis) and $\norm{\cdot}_\infty$ is the usual Euclidean norm.
Note that $\content(\alpha \wpz) = \content(\wpz)$ for any $\wpz$ and any $\alpha \in \Q^\times$.

For any $\Q$-subgroup $\Lbf < \SL_N$ we let 
\begin{align*}
\vpz_{\Lbf}\in \bigwedge^{\dim(\Lbf)}\Lie(\Lbf)(\Q) \subset \bigwedge^{\dim(\Lbf)}\mathfrak{sl}_N(\Q)
\end{align*}
be one of the two primitive integral vectors in the line $\bigwedge^{\dim(\Lbf)}\Lie(\Lbf)(\Q)$. 
The height of $\Lbf$, denoted by $\height(\Lbf)$, is the Euclidean norm of $\vpz_{\Lbf}$ where the Euclidean norm is induced from the usual Euclidean norm on $\mathfrak{sl}_N \subset \Mat_N$.
Since the height only depends on the Lie algebra, we write $\height(\Lbf) = \height(\rho(\Lbf))$ for any $\Q$-subgroup $\Lbf < \G$ (and in particular, for $\G$ itself).

More generally, and similarly to the introduction, we define the complexities
\begin{align*}
\cpl(X) = \content\big(\vpz_{\rho(\G)}\big)=\height(\G),\quad
\cpl(Y_{\data}) = \content\big(g_{\data}^{-1}.\vpz_{\iota(\H)}\big).
\end{align*}

\subsection{Volume of a homogeneous set}\label{sec:volumeoverQ}
We now recall a notion of volume of a homogeneous set from~\cite{EMMV}.
We will do so only in our setting i.e.~for adelic periods over $\Q$ (as this is sufficient for our purposes). Note that this discussion is subsumed by the discussion in Appendix~\ref{sec:volvscpl}.

Let $\Omega\subset\SL_N(\A)$ be an open neighborhood of the identity with compact closure.
Define the volume of $Y_\data=[\iota(\H(\A))g_\data]$ (and similarly for $X$) by 
\[
\vol(Y_\data)=\frac{m_H(Y_\data)}{m_H(H\cap\Omega)},
\]
where $m_H$ is a Haar measure on $H=g_\data^{-1}\iota(\H(\A))g_\data$.

If $\Omega'\subset \SL_N(\A)$ is another compact open neighborhood of the identity, then the volume $\vol'(\cdot)$ defined using $\Omega'$ satisfies $\vol'(\cdot)\asymp\vol(\cdot)$, see~\cite[\S2.3]{EMMV}. In the sequel, we will assume that
\[
\Omega=\prod_{v\in \Sigma}\Omega_v,
\]
where~$\Omega_\infty\subset\SL_N(\R)$ \index{Omega@$\Omega_0=\prod_{\gep\in\places}\Omega_\gep$, open set in~$\SL_N(\A)$ 
to normalize the volume of~$H$-orbits} is a fixed open neighborhood
of the identity in $\SL_N(\R)$ and~$\Omega_\gep=\SL_N(\Z_\gep)$ for all primes $\gep$. 

The following proposition is a special case of Proposition~\ref{prop: vol complexity} and is proved in Appendix~\ref{sec:volvscpl}.

\begin{proposition}\label{prop: vol complexity'}
Assume that $\G$ is $\Q$-anisotropic.
Then there exists a constant $\consta\label{a:cplvsvol}>0$ depending only on $N$ such that
\begin{align}\label{eq:volvscpl'}
\cpl(Y_\data)^{1/\ref{a:cplvsvol}} \ll \vol(Y_\data) \ll \cpl(X)^{\ref{a:cplvsvol}}\cpl(Y_\data)^{\ref{a:cplvsvol}}
\end{align}
where the implicit constants depend only on $N$.
\end{proposition}

In view of Proposition~\ref{prop: vol complexity'}, we may switch between $\vol(\cdot)$ and $\cpl(\cdot)$ at will in the proof of Theorems~\ref{thm:main equi}--\ref{thm:asinEMV}.

\subsection{Good primes}\label{sec:good place}
Let $\H=\H_1\cdots\H_k$ be a direct product decomposition of $\H$ 
into $\Q$-almost simple factors --- recall that $\H$ is simply connected. Let $F_j/\Q$ be a finite extension so that $\H_j={\rm Res}_{F_j/\Q}(\H_j')$ where $\H_j'$ is an absolutely almost simple $F_j$-group for all $1\leq j\leq k$. 
Then $[F_j:\Q]$ is bounded by $\dim\H$.  
Let $\mathcal H'_j$ be the quasi-split inner form of $\H'_j$ over $F_j$.
Let $L_j/F_j$ be the corresponding number field defined as in~\cite[\S0.2]{Pr-Volume}. 
That is, $L_j$ is the splitting field of~$\Hcal_j'$ except in the case where~$\Hcal_j'$ is a triality form of type~${}^6\mathsf{D}_4$ where it is a degree~$3$ subfield
of the degree~$6$ Galois splitting field with Galois group~$S_3$. 

For any prime $\gep$, let 
\[
K_\gep^* =  \iota^{-1}(g_\gep \SL_N(\Z_\gep) g_\gep^{-1}).
\] 
Then $K_\gep^*\subset \prod_{j=1}^k\prod_{v|\gep}K^*_{j,v}$ where $K^*_{j,v}$
is the projection of $K^*_\gep$ into $\H_j'(F_{j,v})$ and, in particular, it is a compact open subgroup of $\H_j'(F_{j,v})$.
We recall the following

\begin{proposition}[{\cite[\S5.11]{EMMV}}]\label{prop:splitting place}
For every $A\geq 1$, there exists a prime $\gp$ satisfying 
\[
A\leq p\ll_A \max\{(\log\vol (Y_\data))^{2}, (\log\vol (X))^{2}\}
\]
so that all of the following properties hold:
\begin{enumerate}  
	\item $\G$ is quasi split over~$\Q_\gp$ and splits over the maximal unramified extension~$\widehat{\Q_\gp}$, and~$K_\gp$
	is a hyperspecial subgroup of~$\G(\Q_\gp)$,
\item $L_j/\Q$ is unramified at $\gp$ for every $1\leq j\leq k$,
\item $\H_{j,v}'$ is quasi split over $F_{j,v}$ and splits
over~$\widehat{F_{j,v}}=\widehat{\Q_\gp}$, for every $1\leq j\leq k$ and every $v|\gp$,
\item $K_\gp^*=\prod_{j=1}^k\prod_{v|\gp}K_{j,v}^*$, and $K_{j,v}^*$ is hyperspecial for all $1\leq j\leq k$ and all $v|\gp$.
\end{enumerate}
\end{proposition}

We will refer to a prime $\gp\in\places_f$\index{p@$\gp$, starting with~\S\ref{ss;good-place} a good prime (or good place) $Y_\data$} satisfying properties in Proposition~\ref{prop:splitting place} as a {\em good} prime (for~$X$ and $Y_\data$).
For simplicity in notation, we write $\jmap_\gp\colon\H\to\rho(\G)$ for the homomorphism
defined by $\jmap_\gp(\cdot)=g_{\data,\gp}^{-1}\iota(\cdot)g_{\data,\gp}$ at a good prime~$\gp$\index{jp@$\jmap_\gp(\cdot)=g_{\data,\gp}^{-1}\iota(\cdot)g_{\data,\gp}$, homomorphism at good prime~$\gp$}.

Bruhat-Tits theory, see~\cite{Tits-Corvallis}, provides smooth group schemes $\Gfrak_\gp$ and $\Hfrak_\gp$ whose generic fibers are $\G$ and $\H$, respectively, and so that 
\[
K_\gp=\Gfrak_\gp(\Z_\gp)\quad\text{ and }\quad K_\gp^*=\Hfrak_\gp(\Z_\gp).
\]
In fact, $\Gfrak_\gp=\prod_i \Gfrak_{\gp,i}$ where for every $i$, the generic fiber of $\Gfrak_{\gp,i}$ is a $\Q_\gp$-simple factor of $\G(\Q_\gp)$ and the special fiber of  $\Gfrak_{\gp,i}$ is an $\mathbb F_\gp$-simple factor of the special fiber of $\Gfrak_{\gp}$. A similar statement holds for $\Hfrak_\gp$.

It follows from~\cite[\S6.2]{EMMV} that $\rho$ extends to a closed immersion from $\Gfrak_\gp$ to $(\SL_N)_{\Z_\gp}$, and that the map $\jmap_\gp$ extends to a closed immersion from $\Hfrak_\gp$ to $(\SL_N)_{\Z_\gp}$, respectively. 
In particular, the $\Z_\gp$-structure on $(\mathfrak{sl}_N)_{\Z_\gp}$ defined above agrees with the $\Z_\gp$-structure on $\Lie(\mathfrak G_\gp)$. 

Recall that $\red_\gp$ denotes the reduction map modulo $\gp$ and, for simplicity, we will write $\underline{\bullet}$ to denote $\red_\gp(\bullet)$ for any $\Z_\gp$-module, scheme, or morphism $\bullet$. Abusing the notation, given an $\Q_\gp$-subspace $V\subset\mathfrak{sl}_N(\Q_\gp)$, we denote $\underline{V[0]}$ simply by $\underline{V}$.  

\medskip

The following will play an important role in the sequel. 

\begin{lemma}\label{lem:principal SL2}
 There exists a closed immersion 
\[
\theta_{0,\gp}: \SL_2 \longrightarrow \Hfrak_\gp,
\] 
of $\Z_\gp$-group schemes so that $\theta_\gp:=\jmap_\gp\circ\theta_{0,\gp}$ satisfies the following properties.
\begin{enumerate}
\item The map $\theta_\gp:\SL_2\to \SL_N$ is a closed immersion. 
\item The projection of $\theta_\gp(\SL_2(\Q_\gp))$ into each $\Q_\gp$-almost simple factor of $H_\gp$ is nontrivial.
\item $\Ad\circ\theta_\gp: \SL_2\to \SL(\mathfrak{sl}_N)$ is a closed immersion of $\Z_\gp$-group schemes.
\item We have $\underline{\Ad\circ\theta_\gp}= \Ad\circ\underline{\theta_\gp}$
as representations of $\SL_2(\overline{\mathbb F_\gp})$ on $\mathfrak{sl}_N\otimes_{\mathbb F_\gp}\overline{\mathbb F_\gp}$. 
\end{enumerate}
\end{lemma}

\begin{proof}
For the existence of $\theta_{0,\gp}$ see~\cite[\S6.7]{EMMV} and the references therein. 
Then the first part follows from the fact that $\jmap_\gp$ is a closed immersion.

    The second part is proved in~\cite[\S6.7]{EMMV}. Indeed $\theta_{0,\gp}\colon \SL_2 \longrightarrow \Hfrak_\gp$ is constructed in loc.\ cit.\ precisely so that if we define $\theta_\gp$ as above, then part~(2) holds.

Since $\gp\gg_N1$, the adjoint representation $\Ad\colon \SL_N \to \SL(\mathfrak{sl}_N)$ is a closed immersion and the reduction mod $\gp$ is the adjoint representation of $(\SL_N)_{\overline{\mathbb{F}_\gp}}$.
Thus, Part~(3) follows as $\Ad \circ \theta_\gp$ is a composition of closed immersions. Taking the reduction mod $\gp$, we also conclude for Part~(4).   
\end{proof}

We will refer to $\theta_\gp(\SL_2(\Q_\gp))$ as the {\em principal} $\SL_2$ in the sequel (though we caution readers that $\theta_\gp(\SL_2(\Q_\gp))$ is merely isomorphic to a quotient of $\SL_2(\Q_\gp)$ by a finite normal subgroup). 
We define a one-parameter unipotent subgroup~$u:\Q_\gp\to\theta_\gp(\SL_2(\Q_\gp))$
by
\be\label{eq:ut in SL2}
 u(s):=\theta_\gp\left(\begin{pmatrix} 1&s\\0&1\end{pmatrix}\right);
\ee
let $U=\{u(s): s\in \Q_\gp\}$. 
We write 
\begin{align}\label{eq:direction z}
\zpz = \zpz^+ = \mathrm{D}\theta_\gp  \left(\begin{pmatrix}
0 &1 \\ 0 &0
\end{pmatrix}\right)
\in \gfrak_p[0]
\end{align}
for the derivative of $u(\cdot)$ at $0$ where we recall that $\gfrak_\gp$ is identified with its image under $\rho$.
We let $\zpz^- \in \gfrak_\gp[0]$ be the analogously defined direction using the lower nilpotent in $\mathfrak{sl}_2$.
For any $t\in \Q_\gp^\times$ set 
\be\label{eq:at in SL2}
 a(t):=\theta_\gp\left(\begin{pmatrix} t&0\\0&t^{-1}\end{pmatrix}\right),
\ee
and let $\diagsubgrp=\{a(t): t\in \Q_\gp^\times\}$ and $a:=a(\gp^{-1})$. 
Note that $U$ and $A$ are the $\Q_\gp$-points of $\Z_\gp$-schemes, and write, as before, $\{\underline{a}(t): t\in \mathbb F_\gp^\times\}$ and $\{\underline{u}(s): s\in \mathbb F_\gp\}$ for the reduction mod $\gp$ of the groups
$\{a(t): t\in \Z_\gp^\times\}$ and $\{u(s): s\in \Z_\gp\}$, respectively. 
Abusing the notation we write $\underline{A}=\{\underline{a}(t): t\in \mathbb F_\gp^\times\}$ and 
$\underline{U}=\{\underline{u}(s): s\in \mathbb F_\gp\}$.

\subsection{The adjoint representation of the principal $\SL_2$}\label{sec: adjoint rep}
A vector $\vpz \in \mathfrak{sl}_N(\Q_\gp)$ is said to be of weight $\lambda$ if $a(t).\vpz=t^\lambda\vpz$;
we will say $\vpz$ has pure weight if it has some weight. A non-zero vector will be called a highest weight vector if it has pure weight and is $U$ invariant. If a $\Q_\gp$-subspace $V\subset\mathfrak{sl}_N(\Q_\gp)$ is invariant under $A$, we let $V^{(\lambda)}$ denote the space of vectors in $V$ of weight $\lambda$, and let $V^{\rm hw}$ (resp.~$V^{\mathrm{hw},(\lambda)}$) denote the subspace of $U$-invariant vectors (resp,~the subspace of $U$-invariant vectors of weight $\lambda$). 

Similarly, a vector $\wpz \in \mathfrak{sl}_N(\mathbb{F}_\gp)$ is said to be of weight $\lambda$ if $\underline{a}(t).\wpz=t^\lambda\wpz$. 
We also define pure and highest weight vectors accordingly, albeit with $\underline{U}$ in place of $U$.  
If $W\subset\mathfrak{sl}_N(\mathbb{F}_\gp)$ is an $\mathbb F_\gp$-subspace which is invariant under $\underline{A}$, we let $W^{(\lambda)}$ denote the space of vectors in $W$ of weight $\lambda$. Define $W^{\rm hw}$ and $W^{\mathrm{hw},(\lambda)}$ similarly. 

\begin{lemma}\label{lem: reduction highest weight}
    Let $V\subset\mathfrak{sl}_N(\Q_\gp)$ be a $\theta_\gp(\SL_2(\Q_\gp))$-invariant subspace. Then 
    \[
    V^{\rm hw}[0]=\bigoplus_{\lambda} V^{\rm hw,(\lambda)}[0],
    \]
    and for each $\lambda$, there is a $\Z_\gp$-basis 
    $\{\vpz_{\lambda,1},\ldots,\vpz_{\lambda,d_\lambda}\}$ of $V^{\mathrm{hw},(\lambda)}[0]$ so that the following hold for every $\lambda$:
    \begin{enumerate}[label=\textnormal{(\theenumi)}]
        \item  $\bigl\{\underline{\vpz_{\lambda,1}},\ldots,\underline{\vpz_{\lambda,d_\lambda}}\bigr\}$ is a basis for $\underline{V^{\rm hw,(\lambda)}}$.
        \item For every $1\leq i\leq d_\lambda$, 
        \[
        \SL_2(\Q_\gp).\vpz_{\lambda,i}
        \]
        spans an $\lambda+1$-dimensional irreducible representation of $\SL_2(\Q_\gp)$, and 
        \[
        \SL_2(\mathbb F_\gp).\underline{\vpz_{\lambda,i}}
        \]
        spans an $\lambda+1$-dimensional irreducible representation of $\SL_2(\mathbb F_\gp)$. 
        \item $\underline{V^{\rm hw,(\lambda)}}=\underline V^{\rm hw,(\lambda)}$. In particular, $\underline{V^{\rm hw}}=\underline{V}^{\rm hw}$.
    \end{enumerate}
\end{lemma}

\begin{proof}
    First note that if $W\subset \mathfrak{sl}_N(\Q_\gp)$ is a $\Q_\gp$-subspace and $\wpz_1,\ldots,\wpz_d\in W[0]$ is a basis for $W[0]$ as $\Z_\gp$-module, then  $W[0]=\oplus_i \Z_\gp \wpz_i$ and $\{\underline{\wpz_1},\ldots,\underline{\wpz_d}\}$ is a basis for $\underline{W}$. 

    Using representation theory of $\SL_2$ over $\Q_\gp$ we conclude that
$V^{\rm hw}=\oplus_{\lambda\geq 0} V^{\rm hw,(\lambda)}$ and each non-zero vector $\vpz\in V^{\rm hw,(\lambda)}$ with $\lambda\geq 0$ is a highest weight vector. Moreover, $\theta_\gp(\SL_2(\Q_\gp)).\vpz$ spans an $\lambda+1$-dimensional irreducible representation of $\theta_\gp(\SL_2(\Q_\gp))$.

For each $\lambda$, fix an $\Z_\gp$-basis $\{\vpz_{\lambda,1},\ldots,\vpz_{\lambda,d_\lambda}\}$ for $V^{\rm hw,(\lambda)}[0]$. Then for every $i$, $\vpz_{\lambda,i}$ is a highest weight vector with weight $\lambda$. 
Since $\gp\gg 1$, we have that $\underline{\vpz_{\lambda,i}}$ is also a highest weight vector with weight $\lambda$.
Moreover, $\{\underline{\vpz_{\lambda,1}},\ldots,\underline{\vpz_{\lambda,d_\lambda}}\}$ is an $\mathbb F_\gp$-basis for $\underline{V^{\rm hw,(\lambda)}}$ (that is, part (1) of the lemma is verified) and the following holds  
\be\label{eq: basis for V-ht}
\bigl\{\underline{\vpz_{\lambda,i}}: \lambda\geq 0, 1\leq i\leq d_\lambda\bigr\}\quad\text{is a basis for $\underline{V^{\rm hw}}$.}
\ee
To verify \eqref{eq: basis for V-ht}, suppose $\sum_{\lambda,i}c_{\lambda,i}\underline{v_{\lambda,i}}=0$, and let $\upz_\lambda=\sum_ic_{\lambda,i}\underline{v_{\lambda,i}}$. 
Let $m$ be the largest weight where $\upz_{m}\neq 0$, then $t^m\upz_m=-\sum_{m'<m}t^{m'}\upz_{m'}$ for all $t\in \mathbb{F}_\gp$ by applying $\underline{a}(t)$. 
Since $\gp \gg 1$ and the above is a polynomial equation in $t$, this implies $\upz_\lambda=0$ for all $\lambda$.
As $\{\underline{\vpz_{\lambda,i}}:i\leq d_\lambda\}$ is an $\mathbb F_\gp$-basis for $\underline{V^{\rm hw,(\lambda)}}$,
\eqref{eq: basis for V-ht} holds.

Now since $\oplus_{\lambda} V^{\rm hw,(\lambda)}[0]\subset V^{\rm hw}[0]$, we conclude from~\eqref{eq: basis for V-ht} that 
\begin{align*}
V^{\rm hw}[0]=\oplus_{\lambda\geq 0} V^{\rm hw,(\lambda)}[0]   
\end{align*}
as it was claimed in the lemma. 

We now show that part~(2) also holds for these vectors. The claim regarding $\SL_2(\Q_\gp)$ was already discussed, thus, we work over the residue field $\mathbb F_\gp$.
Since $\gp\gg 1$ and the representation of $\SL_2(\mathbb F_\gp)$ on $\underline{\gfrak_\gp}$ is given by $\underline{\Ad\circ\theta_\gp}$, see Lemma~\ref{lem:principal SL2}, $\underline V$ is completely reducible and the representation appearing in $V$ are the standard highest weight representation. In particular, $\SL_2(\mathbb F_\gp).\underline{\vpz_{\lambda,i}}$ spans an $\lambda+1$-dimensional irreducible representation of $\underline{V}$. 

To see part~(3), first note that by part~(2), we have $\underline{V^{\rm hw,(\lambda)}}\subset \underline{V}^{\rm hw,(\lambda)}$. Moreover, by part~(1) we have $\dim(\underline{V^{\rm hw,(\lambda)}})=\dim(V^{\rm hw,(\lambda)})$. 
These, the above remark regarding subrepresentations of $\underline V$, and dimension count, imply
\begin{equation}\label{eq: dimension count}
\begin{aligned}
    \dim(V) = \dim(\underline V) &= \sum_{\lambda} (\lambda+1) \dim(\underline{V}^{\rm hw,(\lambda)}) \\
    &\geq \sum_{\lambda} (\lambda+1) \dim(V^{\rm hw,(\lambda)}) = \dim(V).
\end{aligned}
\end{equation}
Thus $\dim(\underline{V}^{\rm hw,(\lambda)}) = \dim (\underline{V^{\rm hw,(\lambda)}}) = \dim (V^{\rm hw,(\lambda)})$ for every $\lambda\geq 0$, which proves part~(3).
\end{proof}

\begin{lemma}\label{lem: reduction completely red}
Let $V\subset\mathfrak{sl}_N(\Q_\gp)$ be a $\theta_\gp(\SL_2(\Q_\gp))$-invariant subspace. Then there are $\theta_\gp(\SL_2(\Q_\gp))$-irreducible subspaces $W_1,\ldots, W_d\subset V$ so that 
\[
V[0]=\bigoplus_{i=1}^d W_i[0]
\]
and $\underline{V}=\bigoplus_{i=1}^d \underline{W_i}$ is decomposition of $\underline{V}$ into $\SL_2(\mathbb F_\gp)$-irreducible representations. 

In particular, if $V$ is irreducible, then so is $\underline{V}$.      
\end{lemma}

\begin{proof}
    By Lemma~\ref{lem: reduction highest weight}, $V^{\rm hw}[0]$ has a $\Z_\gp$-basis 
    $\{\vpz_i: 1\leq i\leq d\}$ so that $\{\underline{\vpz_i}: 1\leq i\leq d\}$ is an $\mathbb F_\gp$-basis for $\underline{V}^{\rm hw}$. For every $i$, let $W_i$ be the $\Q_\gp$-span of $\theta_\gp(\SL_2(\Q_\gp)).\vpz_i$. We claim the lemma holds with $W_1,\ldots, W_d$. 
    
    First note that $V=\oplus_i W_i$ and $\oplus W_i[0]\subset V[0]$. Let now $1\leq i\leq d$, and let $\mathcal W_i$ be the $\Z_\gp$-span of $\theta_\gp(\SL_2(\Z_\gp)).\vpz_i$. Then $\mathcal W_i\subset W_i[0]$, moreover, in view of Lemma~\ref{lem:principal SL2} and part~(2) of Lemma~\ref{lem: reduction highest weight}, $\underline{\mathcal W_i}$ is an irreducible representation of $\SL_2(\mathbb F_\gp)$ of dimension $\dim W_i$.  Put $\mathcal V=\oplus_i \mathcal W_i$; arguing as in~\eqref{eq: dimension count}, we conclude that $\underline{\mathcal V}=\underline{V}$. Since $\mathcal V\subset V[0]$, the lemma follows.   
\end{proof}

For any vector $\vpz \in V[0]$ in a subrepresentation $V \subset \mathfrak{sl}_N(\mathbb{Q}_\gp)$ of the principal $\SL_2$, we write $\vpz^{\mathrm{nt}}$ for the natural projection onto the direct sum of the non-trivial irreducible subrepresentations of $V$.
By Lemma~\ref{lem: reduction completely red}, we have $\vpz^{\mathrm{nt}} \in V[0]$.

We will also need the following lemma. 

\begin{lemma}\label{lem: weight spaces are integral}
Let the notation be as in Lemma~\ref{lem: reduction completely red}, and let $W_i\subset V$ be any of the irreducible representations given by loc.\ cit. Then there exists a $\Z_\gp$-basis $\{\wpz_{i,1},\ldots,\wpz_{i,d_i}\}$ of $W_i[0]$ so that $\wpz_{i,j}$ is a pure weight vector for all $j$. 

In particular, $\underline{V}^{(\lambda)} = \underline{V^{(\lambda)}}$ for any $\lambda$.
\end{lemma}

\begin{proof}
Let us first prove the first claim. For simplicity in the notation we will drop the index $i$ and denote $W_i$, $d_i$, etc.\ by $W$, $d$, etc. Let $\lambda_1,\ldots,\lambda_d$ denote the distinct weights of $W$. For every weight $\lambda_j$, let $\wpz_j\in W^{(\lambda_j)}[0]$ denote a primitive vector, i.e., $\Z_\gp.\wpz_j=W^{(\lambda_j)}[0]$. We claim $\{\wpz_1,\ldots,\wpz_d\}$ is the desired basis. To see this, note that for any $t\in\Z_\gp^\times$, we have 
    \[
    \underline{a}(\underline{t}) \underline{\wpz_j}=\underline{a(t)\wpz_j}=\underline{t^{\lambda_j}\wpz_j}=\underline{t}^{\lambda_j}\underline{\wpz_j}.
    \]
As $p$ is large,  $\{\underline{\wpz_1},\ldots,\underline{\wpz_d}\}$ are linearly independent.
    Since $\underline{W}$ is a $d$-dimensional irreducible representation of $\SL_2(\mathbb F_\gp)$, $\{\underline{\wpz_1},\ldots,\underline{\wpz_d}\}$ is a basis for $\underline{W}$. This completes the first assertion in the lemma.

    The second claim in the lemma follows from the above and Lemma~\ref{lem: reduction completely red}. 
\end{proof}

For $\vpz \in V[0]$ contained in a subrepresentation $V \subset \mathfrak{sl}_N(\mathbb{Q}_\gp)$ of the principal $\SL_2$ we write $\vpz^+$ (resp.~$\vpz^-$) for its projection onto the positive (resp.~negative) weight components.
Note that $\vpz^+,\vpz^-\in V[0]$ by Lemma~\ref{lem: weight spaces are integral}.

\subsection{Undistorted complements}\label{sec: undistorted}
We now record two corollaries of the above discussion which will be used in the sequel.  

\begin{lemma}\label{lem:nofactorscontainingHp}
Suppose that $\iota(\H)$ is not contained in the image of any proper $\Q$-factor of $\G$ under $\rho$.
Then the image of the closed immersion $\theta_{\gp}: \SL_2 \to \SL_N$ is not contained in the image of any proper factor of $\Gfrak_\gp$ under $\rho$.
Similarly, the image of the closed immersion $\underline{\theta_{\gp}}: (\SL_2)_{\overline{\mathbb{F}_\gp}} \to (\SL_N)_{\overline{\mathbb{F}_\gp}}$ is not contained in the image of any proper factor of $\underline{\Gfrak_\gp}$ under $\underline{\rho}$.
\end{lemma}

\begin{proof}
Recall that $G_\gp=\rho(\Gfrak(\Q_\gp))$ and $H_\gp=\jmap_\gp(\Hfrak_\gp(\Q_\gp))$. 
We first show that if $\sfrak \lhd \gfrak_p$ is the minimal ideal (with respect to inclusion) containing $\Lie(H_\gp)$, then 
$\sfrak$ is defined over $\Q$.
Indeed, since $\sfrak$ is an ideal, it is also the minimal ideal containing $\Lie(\iota(\Hbf(\Q_\gp)))$, which is defined over $\Q$.
Thus for any $\sigma \in \mathrm{Aut}(\overline{\Q_\gp}/\Q)$ the Lie algebra $\sfrak \cap \sfrak^\sigma$ also contains $\Lie(\iota(\Hbf(\Q_\gp)))$; in view of the minimality of $\sfrak$ hence $\sfrak \cap \sfrak^\sigma=\sfrak$. This implies that $\sfrak$ is necessarily defined over $\Q$ as we claimed. 
Combined with our assumption on $\iota(\H)$, thus, $H_\gp$ has non-trivial projection to each $\Q_\gp$-almost simple factor of $G_\gp$. 

Recall now from part~(2) of Lemma~\ref{lem:principal SL2} that the projection of $\theta_{\gp}(\SL_2(\Q_p))$ to each $\Q_\gp$-almost simple factor of $H_\gp$ is non-trivial. 
Together with the fact that $\H$ is simply connected, this implies that $H_\gp$ is generated by $H_\gp$-conjugates of $\theta_{\gp}(\SL_2(\Q_p))$. Altogether, we conclude that the projection of $\theta_{\gp}(\SL_2(\Q_p))$ to each 
$\Q_\gp$-almost simple factor of $G_{\gp,i}$ of $G_{\gp}$ is non-trivial; establishing the first claim.  

To see the second claim, recall that each $\Q_\gp$-almost simple factor $G_{\gp,i}$ is the image of the generic fiber of 
$\Gfrak_{\gp,i}$ under $\rho$ and $\underline{\Gfrak_{\gp,i}}$ is an $\mathbb F_\gp$-almost simple factor of $\underline{\Gfrak_\gp}$. 
Moreover, all 
$\mathbb F_\gp$-almost simple factors of $\underline{\Gfrak_\gp}$ arise this way, see the discussion preceding Lemma~\ref{lem:principal SL2}. 
Now by the first claim, the representation of $\theta_{\gp}(\SL_2(\Q_p))$ on $\Lie(G_{\gp,i})$ 
has a nontrivial highest weight. Thus by Lemma~\ref{lem: reduction highest weight} the representation of $\SL_2(\mathbb F_\gp)$ on 
\[
\underline{\Lie(G_{\gp,i})}=\underline{\Gfrak_{\gp,i}}(\mathbb F_\gp)
\] 
is non-trivial. This implies the second claim and finishes the proof of the lemma.  
\end{proof}

For any $\Q_\gp$-subspace $V \subset \gfrak_\gp$ and $k \geq 0$ we write $V[k] = V \cap \gfrak_\gp[k]$.
A complement $W$ to $V$ is \emph{undistorted} if $W[0] \oplus V[0] = \gfrak_\gp[0]$; note that if this holds, then $W[k] \oplus V[k] = \gfrak_\gp[k]$ for all $k \geq 0$. Arguing as in the proof of Lemma~\ref{lem: reduction completely red}, we have the following. 

\begin{lemma}[Existence of undistorted complements]\label{lem: undistorted complement}
Let $V\subset\gfrak_\gp$ be a subspace invariant under the principal $\SL_2$.
  Then there exists an undistorted complement $W$ to $V$ which is also invariant under the principal $\SL_2$. 
\end{lemma}

\begin{proof}
    First note that since $\Z_\gp$ is a PID, there exists a $\Z_\gp$-basis 
    $(\vpz_i)_{1 \leq i \leq m}$ of $\gfrak_\gp^{\rm hw}[0]$ so that $(\vpz_i)_{1 \leq i \leq d}$ is a $\Z_\gp$-basis for $V^{\rm hw}[0]$. 
    In particular, $(\underline{\vpz_i})_{1\leq i\leq m}$ is an $\mathbb F_\gp$-basis for $\underline{\gfrak_\gp}^{\rm hw}$ and 
    $(\underline{\vpz_i})_{1\leq i\leq d}$ is an $\mathbb F_\gp$-basis for $\underline{V}^{\rm ht}$ by Lemma~\ref{lem: reduction highest weight}.
    
    For every $1\leq i\leq m$, let $W_i$ be the $\Q_\gp$-span of $\theta_\gp(\SL_2(\Q_\gp)).\vpz_i$. 
    By the proof of Lemma~\ref{lem: reduction completely red}, applied with $V$ and $\gfrak_\gp$ and the above basis, we have 
    \[
    \gfrak_\gp[0]= \bigoplus_{i=1}^m W_i[0]\quad\text{and}\quad V[0]=\bigoplus_{i=1}^d W_i[0] 
    \]
    The claim in the lemma thus holds with $W=\oplus_{i=d+1}^m W_i$.   
\end{proof}

\section{Outline of the proof}\label{sec:outlineproof}

\subsection{Standing assumptions}\label{sec:standingassumptions}
For the outline of the proof of Theorems~\ref{thm:main equi}, \ref{thm:asinEMV} in this section, and until the proof of the main theorems in \S\ref{sec:proofmainthms} we make the following
\textsc{Standing Assumptions:}
\begin{itemize}
    \item $\G$ is a $\Q$-anisotropic simply connected semisimple group, $\rho: \G \to \SL_N$ is a homomorphism defined over $\Q$ with central kernel, and
    \begin{align*}
    X = [\rho(\G(\A))].
    \end{align*}
    The Lie algebra $\gfrak$ and its subalgebras are identified with their images under~$\rho$.
    \item $\data = (\H,\iota,g_\data)$ is semisimple simply connected data over $\Q$ consistent with $(\G,\rho)$ i.e.~$\iota: \H \to \rho(\G)$ is a homomorphism defined over $\Q$ with central kernel and $g_\data \in \rho(\G(\A))$.
    \item $\iota(\H)$ is not contained in any proper $\Q$-factor of $\rho(\G)$.
    \item $\mu = \mu_{\data}$ is the invariant probability measure on
    \begin{align*}
    Y_\data = [\iota(\H(\A))g_\data].
    \end{align*}
    \item $\gp$ is a good prime for $X$ and $Y_\data$ as in Proposition~\ref{prop:splitting place} (assumed to be $\gg_N 1$).
    \item $\theta_\gp(\SL_2(\Q_\gp))$ is a fixed principal $\SL_2$ (cf.~Lemma~\ref{lem:principal SL2}) contained in $H_\gp = g_{\data,\gp}^{-1}\iota(\H(\Q_\gp))g_{\data,\gp}$. We denote by
 $\{u(r)\},\{a(t)\}$ the unipotent and diagonal subgroups respectively of the principal $\SL_2$ (cf.\ \eqref{eq:ut in SL2} and~\eqref{eq:at in SL2}).
\end{itemize}
Weights and highest weights are understood with respect to these choices (see \S\ref{sec: adjoint rep}).

\subsection{Strategy}\label{sec:strategy}
The overarching strategy is roughly the same as in previous works --- see e.g.~\cite{EMV,EMMV,AW-realsemisimple} --- accumulating `almost invariance'. 
To illustrate this, we will first phrase the approach in vague terms emphasizing the corresponding imprecisions\footnote{In particular we will ignore less crucial multiplicative constants in the informal discussion. However, Proposition \ref{prop:addinv-intro} and Theorem \ref{thm:effgen-intro} are precise as stated.} by placing the terms that need to be made precise in quotes.

The goal is to show `almost invariance' under more and more highest weight directions in an (undistorted) invariant complement $\rfrak$ of $\hfrak_\gp = \Lie(H_\gp)$. In the course of the argument it would transpire (using \ref{item:mainclaims2} below) that if $\mu$ is `almost invariant' under all highest weight directions then it is `very close' to the uniform measure on $X$.

By induction, suppose we are given $\vpz_1,\ldots,\vpz_n \in \rfrak[0]$ highest weight vectors (of non-trivial weight) which are linearly independent modulo $\gp$. 
Assume further that the measure $\mu$ is `almost invariant' under the one-parameter unipotent subgroups $U_i = \{\exp(t\vpz_i)\colon t \in \Q_\gp\}$ and that $U_1,\ldots,U_n$ together with $\theta_\gp(\SL_2(\Q_\gp))$ `effectively generate' a ball that is not `too small' in a group $M\lneq  \rho(\G(\Q_\gp))=: G_p$.
We do not assume that $M$ contains the group $H_\gp$.

Assuming that $Y_\data$ is not contained in an orbit of `very small' complexity, one wishes to say the following:
\begin{enumerate}[{\color{red}(A)}]
    \item\label{item:mainclaims1} \emph{Additional almost invariance:}
    There is an additional direction $\vpz_{n+1}$ in $\Lie(G_p)$ of highest (non-zero) weight which is `transversal' to $\Lie(M)$, so that $\mu$ is `almost invariant' under the one-parameter unipotent subgroup $\{\exp(t\vpz_{n+1}): t \in \Q_\gp\}$.
    \item\label{item:mainclaims2} \emph{Effective generation:} Assuming $\vpz_{n+1}$ as in \ref{item:mainclaims1} exists,  
    there is a `perturbation' $\vpz_1',\ldots,\vpz_{n+1}'$
    so that $\mu$ is `almost invariant' under the corresponding one-parameter unipotent subgroups $U_1',\ldots,U_{n+1}'$ and in addition
    $U_1',\ldots,U_{n+1}'$ together with $\theta_\gp(\SL_2(\Q_\gp))$ `effectively generate' a group $M'$.
\end{enumerate}
We point out that we do not prove that $M'$ has larger dimension than $M$, but certainly $\dim(M') \geq (n+3)+1$.
Iterating the above one obtains `almost invariance' under $G_p$, and in particular the horospherical subgroups of $G_p$ corresponding to $a(t)$; from here, one can conclude using spectral gap on the ambient space.
We turn to making the above steps precise.

\subsection{Some effective notions}\label{sec:outline-effnotions}
Given a $C^1$-function $f$ on $X$ we write $\lev(f)$ for the level of $f$ i.e.~the least integer $L\geq 1$ such that $f$ is invariant under $\prod_\gep G_\gep[\ord_\gep(L)]$.
Moreover, as in the introduction we fix an inner product on $\mathfrak{gl}_N(\R)$
and define the $C^1$-norm $\norm{f}_{C^1(X)}$ as the maximum of the sup norms of the function and its partial derivatives in directions corresponding to an orthonormal basis of $\gfrak_{\infty}$.
We use the following notion of almost invariance:

\begin{definition}\label{def:almost invariance}
Let $\varepsilon>0$.
We say that $\mu$ is $\varepsilon$-\emph{almost invariant} under $g \in G_p$ if for all $C^1$-functions $f$ on $X$
\begin{align*}
\Big| \int f(\cdot g)  \de \mu - \int f \de \mu \Big| \leq \varepsilon\, \lev(f) \norm{f}_{C^1(X)}.
\end{align*}
Moreover, $\mu$ is $\varepsilon$-almost invariant under a subgroup of $G_p$ if it is $\varepsilon$-almost invariant under every element of that subgroup.
Lastly, $\mu$ is $\varepsilon$-almost invariant under $\vpz \in \gfrak_{\gp}[0]$ if it is $\varepsilon$-almost invariant under $\{\exp(t \vpz): t\in \gp \Z_\gp\}$.
\end{definition}

Notice that the above definition differs from e.g.~the notion of almost invariance in \cite{EMMV} where $L^2$-Sobolev norms were used.
Definition~\ref{def:almost invariance} directly implies that 
\begin{align}\label{eq:trivialalminv}
\mu \text{ is } 2\gp^{-k}\text{-almost invariant under all }g \in G_\gp[k].
\end{align}
Other elementary properties of the definition will be discussed in \S\ref{sec:almost invariance} below.

We will use the following notion of effective generation.
Let $M< G_p$ be a closed subgroup with Lie algebra $\mfrak$.

\begin{definition}\label{def:effective generation}
We say that $M$ is $k$-\emph{generated} by nilpotents $\vpz_1,\ldots,\vpz_{\dim(\mfrak)}\in \mfrak[0]$ if the following holds:
There exists $t\in \Z_{\gp}^{\dim(\mfrak)}$ so that for the map
$\varphi\colon \Q_{\gp}^{\dim(\mfrak)}\to \SL_N(\Q_{\gp})$ defined by
\begin{align}\label{eq:def phi}
\varphi(t_1,\ldots,t_{\dim(\mfrak)})=\exp(t_1\vpz_1) \cdots \exp(t_{\dim(\mfrak)}\vpz_{\dim(\mfrak)})
\end{align}
the derivative ${\rm D}_t\varphi$ has a $\dim(\mfrak)$-minor of absolute value at least ${\gp}^{-k}$.
\end{definition}

In this case (cf.~Lemma~\ref{lem:implicitfctthm}), a quantitative open mapping theorem implies that 
\begin{align*}
\varphi(\Z_\gp^{\dim(\mfrak)})^{-1}\varphi(\Z_\gp^{\dim(\mfrak)})
\end{align*}
contains $M[3k]$ --- the ball of radius $\gp^{-3k}$ in $M$ around the identity.
In particular, this notion will allow us to pass from almost invariance under nilpotent elements to almost invariance under a `small' subgroup.

Since the measure $\mu$ is trivially $2\gp^{-k}$-almost invariant under $M[k]$ by \eqref{eq:trivialalminv}, we need to ensure that 
there is no competition between the quality of almost invariance under the given nilpotents and the quality of effective generation.

\subsection{Additional almost invariance}

We prove in \S\ref{sec:prepclosing}--\S\ref{sec:extrainv} the following precise version of \ref{item:mainclaims1} above.

\begin{proposition}\label{prop:addinv-intro}
There exists $\consta\label{a:addinv-intro}>1$ depending only on $N$ with the following property.

Let $\Mbf < \rho(\G)$ be a proper $\Q_\gp$-subgroup containing $\theta_\gp(\SL_2)$.
Let $k \in \N$ with
\begin{align*}
\gp^{\ref{a:addinv-intro}} \cpl(X)^{\ref{a:addinv-intro}} \leq \gp^k \leq 
\mcpl(Y_\data)^{1/\ref{a:addinv-intro}}.
\end{align*}
Suppose that $M= \Mbf(\Q_\gp) < G_p$ is $k$-generated by some nilpotents of pure non-zero weight and that $\mu$ is $\gp^{-\ref{a:addinv-intro}k}$-almost invariant under $M[3k]$.

Then there exists a highest weight vector $\vpz \in \gfrak_\gp[0]$ (of non-zero weight) with $\vpz \mod p \not\in \mfrak[0] \mod p$ so that $\mu$ is $\gp^{-k/\ref{a:addinv-intro}}$-almost invariant under $\vpz$.
\end{proposition}

Let us indicate some ingredients of the proof of Proposition~\ref{prop:addinv-intro}.
The basic idea, that can already be traced to Ratner's proof of her measure classification theorem~\cite{Ratner-Acta,Ratner-measure}, is that nearby Birkhoff generic points give rise to additional invariance (for the current context, see \cite{Einsiedler-RatnerSL2} or \cite[\S2]{EMV}).
We wish to find two points $x_1,x_2=x_1g\in X$ that are generic in a suitable effective sense so that the small displacement $g$ does not `almost normalize' $M$. If this does not happen, we employ an effective closing lemma in Proposition~\ref{prop:closing-lemma} to get a contradiction to our assumption that $\mcpl(Y_{\data})$ is large (compared to $p^k$). The closing lemma we use here is a variant of the closing lemma that is given in joint work of Margulis, Shah and three of the authors of this paper (E.L., A.M., and A.W.)~\cite{ClosingLemma}. The closing lemma in~\cite{ClosingLemma} is more general (though written only for quotients of real groups), but does not take into account the dependence on the ambient group.

When the displacement $g$ does not `almost normalize' $M$, one may `realign' the points $x_1,x_2$ so that $g$ is of the form $\exp(\vpz)$ for some $\vpz \in \gfrak_\gp[0]$ with $\vpz$ `transversal' to $\mfrak$ and with $\Ad(a(t))\vpz$ diverging `sufficiently quickly' as $|t| \to \infty$ (i.e.~$\Ad(a(t))\vpz$ reaches size $1$ for $|t| \geq \norm{\vpz}^{-\star}$).
That $\vpz$ was arranged to be `transversal' to $\mfrak$ ensures there is no contribution from $M$ to this displacement.
Making sure that $\Ad(a(t))\vpz$ diverges `sufficiently quickly' is more involved; letting~$\Cbf$ be the centralizer group of the principal $\SL_2$ the argument employs an effective variant of the group $\big\{g \in \rho(\G): \Mbf g \subset \overline{\Cbf \Mbf}^z \big\}$. The identity component of this group agrees with the identity component of the normalizer group of $\Mbf$, and if $\Ad(a(t))\vpz$ cannot be made to diverge `sufficiently quickly' the effective analogue would contradict the fact that $g$ does not `almost normalize' $M$.

The realigned points are then used
with a version of the averaging operator
\begin{align}\label{eq:avoperator}
f \mapsto \int_{\Z_\gp} f(\cdot\, u(s)a(t)^{-1}) \de s
\end{align}
to prove additional almost invariance. 
A large set of generic points for this operator is provided by an effective ergodic theorem based on uniformity of the spectral gap; see \cite[\S7.6]{EMMV} or \S\ref{sec:effergthm} below.

\subsection{Effective generation}

Suppose we are given a list of nilpotent elements $\extranil=(\vpz_1,\ldots,\vpz_n)$
and let $\Mbf$ be the Zariski closure of the group generated by $\theta_{\gp}(\SL_2(\Q_\gp))$ together with the unipotent one-parameter subgroups defined by the elements of $\extranil$.
We wish to show that $\Mbf(\Q_{\gp})$ is `$k$-generated' by one-parameter subgroups corresponding to $\extranil$ and the usual one-parameter unipotent subgroups of $\theta_{\gp}(\SL_2(\Q_\gp))$.
If this does not work we instead would like to perturb the nilpotent directions in $\extranil$ `slightly' to decrease the dimension of the group $\Mbf$, which should make it easier to be `$k$-generated'. Iterating this scheme we obtain in \S \ref{sec:effectivegeneration} the following precise version of \ref{item:mainclaims2}.

\begin{theorem}\label{thm:effgen-intro}
 There exists $\constk\label{k:effgen-intro}>0$ 
depending only on $N$ with the following properties.
Let $\extranil=(\vpz_1,\ldots,\vpz_n) \in \gfrak_{\gp}[0]^n$
be a list of nilpotent elements with $n\leq \dim(\gfrak)$ where each $\vpz_j$ for $j=1,\ldots,n$ is of pure weight.
Let $\delta \in (0,1/2)$ and
\begin{align}\label{eq:lowerboundeffgen}
k>(\ref{k:effgen-intro}\delta)^{-2\dim(\gfrak)} (\log_{\gp}(\cpl(X))+1).
\end{align}
Then there exists $\alpha \in (\ref{k:effgen-intro}\delta^{\dim(\gfrak)},1]$ and a new list of nilpotent elements $\widetilde{\extranil}=(\tilde{\vpz}_1,\ldots,\tilde{\vpz}_n)\in \gfrak_{\gp}[0]^n$ 
with the following properties:
\begin{itemize}
\item For $j=1,\ldots,n$ we have $\norm{\tilde{\vpz}_j-\vpz_j} < {\gp}^{-\alpha k+\dim(\gfrak)}$.
\item For each $j=1,\ldots,n$ the nilpotent $\tilde{\vpz}_j$ is of the same pure weight as $\vpz_j$.
If $\vpz_j$ is of highest weight, then so is $\tilde{\vpz}_j$.
\item Let $\widetilde{\Mbf}<\rho(\G)$ be the Zariski closure of the group generated by the one-parameter groups $\{\exp(t \tilde{\vpz}_j)\colon t\in \Q_{\gp}\}$ for $j=1,\ldots,n$ and $\theta_{\gp}(\SL_2)$. 
Then $\widetilde{M} =\widetilde{\Mbf}(\Q_{\gp})$ is $\delta\alpha k$-generated by nilpotents of pure non-zero weight either contained in $\widetilde{\extranil}$ or equal to $\zpz^+,\zpz^-$.
\end{itemize}
\end{theorem}

In practice, we will apply the above theorem to a list of nilpotent elements, say $\vpz_1,\ldots,\vpz_n$, for which the measure $\mu$ is $\gp^{-k}$-almost invariant.
(This list is obtained by adding to the list from the previous iteration step the new direction found in Proposition~\ref{prop:addinv-intro}.)
The newly found directions, denoted $\tilde{\vpz}_1,\ldots,\tilde{\vpz}_n$, leave the measure ${\gp}^{-\alpha k+\star}$-almost invariant and generate a comparatively sizable ball around the identity in the group $\widetilde{M}$ (the difference in size between this ball and the amount of invariance we have is, in logarithmic terms, determined by $\delta$).
Note here that the freedom of the parameter $\delta$ is crucial.
Indeed, in order to continue the iteration we wish to apply Proposition~\ref{prop:addinv-intro} with the new group $\widetilde{M}$ and thus choose $\delta$ `sufficiently small' in relationship to the constant $\ref{a:addinv-intro}$ from Proposition~\ref{prop:addinv-intro}.
A fortiori, this asserts that the quality of almost invariance is way better than what the assumed continuity on~$f$ would give.

\section{Effective generation}\label{sec:effectivegeneration}

The goal of this section is to prove the effective generation result in Theorem~\ref{thm:effgen-intro}.
We show that there is a dichotomy --- either the given list of nilpotents either `effectively generates' a group together with $\theta_\gp(\SL_2)$ or there is a `small perturbation' of the list generating a smaller group, which we then iterate to improve the quality of the generation in terms of the size of the pertubation. 

The main tool for establishing the existence of such a perturbation is an effective version of a theorem of Greenberg \cite{Greenberg1,Greenberg2} --- Theorem~\ref{thm:greenberg} below. 
This should be seen as a version of Hensel's lifting lemma for not necessarily smooth varieties. It is crucial for us to have effective dependence of the size of the pertubation in terms of the polynomials defining the variety.

This section is structured as follows:
In \S\ref{sec:firstres eff gen}, we prove a few preliminary properties regarding effective generation. 
In \S\ref{sec:effGreenberg}, we recall an effective version of Greenberg's theorem; the precise form of this theorem we use is taken from \cite{LMMS}.
In \S\ref{sec:geominvth}, we establish a simple fact regarding $\G$-orbits in certain representations using geometric invariant theory.
In \S\ref{sec:varietyofembeddings}, we construct a precursor to the varieties for which we apply Greenberg's theorem.
In \S\ref{sec:openmappingthm}, we prove an open mapping theorem for certain affine varieties with an `essentially' transitive $\G$-action (e.g.~for closed $\G$-orbits in linear representations).
The rough aim is to show that two nearby $\Q_\gp$-points on such varieties differ by a small element of $G_\gp$.
In \S\ref{sec:proofeffgen}, we finally prove Theorem~\ref{thm:effgen-intro}.

\subsection{First results regarding effective generation}\label{sec:firstres eff gen}

Recall our notion of effective generation from \S\ref{sec:outline-effnotions}.
The following statement was used in the outline in \S\ref{sec:outlineproof} in order to pass from `almost invariance' under a list of nilpotents to `almost invariance' under a `small' group.

\begin{lemma}\label{lem:implicitfctthm}
There exists $\consta\label{a:implicitfctthm} >0$ depending only on $N$ with the following property.
Assume ${\gp}>N$ and let $k\geq 0$.
Let $M < \SL_N(\Q_{\gp})$ be a closed subgroup with Lie algebra $\mfrak$.
Suppose that $M$ is $k$-generated by nilpotents $\vpz_1,\ldots,\vpz_{\dim(\mfrak)}\in \mfrak[0]$.
Then 
\begin{align*}
\varphi(\gp\Z_{\gp}^{\dim(\mfrak)})^{-1}\varphi(\gp\Z_{\gp}^{\dim(\mfrak)}) \supset M[2k+\ref{a:implicitfctthm}] 
\end{align*}
where $\varphi$ is given by
\begin{align*}
\varphi(t_1,\ldots,t_{\dim(\mfrak)})=\exp(t_1\vpz_1) \cdots \exp(t_{\dim(\mfrak)}\vpz_{\dim(\mfrak)}).
\end{align*}
\end{lemma}

Following our notation for Lie algebras, we set $\Z_\gp[m] = \gp^m \Z_\gp$ for $m \geq 0$.
We will make use of a well-known quantitative open mapping lemma (see for instance also \cite[Lemma 53$'$]{SG-superapproximationII})  which we summarize as follows.

\begin{lemma}\label{lem:openmapping}
Let $f: \Z_{\gp}^n \to \Z_{\gp}^n$ be an analytic map with $\Z_{\gp}$-coefficients so that $|\det(\mathrm{D}_0 f) |\geq {\gp}^{-k}$ for some $k \geq 0$.
Then 
\begin{align*}
f(\Z_{\gp}[k+1]^n) \supset f(0)+\Z_{\gp}[2k+1]^n.
\end{align*}
\end{lemma}

\begin{proof}
For concreteness, we will give a self-contained proof using Newton's algorithm. Let $x_0=0$.
Let $y \in f(x_0) + \Z_{\gp}[2k+1]^n$. We define
\begin{align*}
x_1 = x_0+(\mathrm{D}_{x_0} f)^{-1}(y-f(x_0)) \in \Q_{\gp}^n
\end{align*}
so that
\[y = f(x_0) + (\mathrm{D}_{x_0} f) (x_1-x_0)
\]
and $x_1-x_0 \in \Z_{\gp}[k+1]^n$. Moreover, by Taylor expansion at $x_0$
\begin{align*}
\norm{f(x_1)-y} \leq \norm{x_1-x_0}^2 \leq {\gp}^{2k} \norm{y-f(x_0)}^2 
\end{align*}
so that $\norm{f(x_1)-y} < \norm{f(x_0)-y}$.
Iterating this procedure and taking the limit, one finds a point $x\in \Z_{\gp}[k+1]^n$ with $f(x) = y$ as claimed.
\end{proof}

\begin{proof}[Proof of Lemma~\ref{lem:implicitfctthm}]
By assumption, there is a point $t_0'\in \Z_{\gp}^{\dim(\mfrak)}$ for which the derivative $\mathrm{D}_{t_0'}\varphi$ has a minor of size at least $\gp^{-k}$.
As $\gp >N$, the map $\varphi$, its derivative, and all its minors are polynomial maps with coefficients in $\Z_{\gp}$.
Applying a Remez-type inequality (see e.g.~\cite[Lemma 5.4]{LMMS}), it follows that there exists a point $t_0\in \gp\Z_{\gp}^{\dim(\mfrak)}$ at which the derivative of $\varphi$ has a $\dim(\mfrak)$-minor of absolute value at least ${\gp}^{-k-A}$ for some $A = A(N)>0$.

Set $f(\cdot) = \log \circ \varphi(\gp\, \cdot): \Z_{\gp}^{\dim(\mfrak)} \to \mfrak[1]$.
This is an analytic map with coefficients in $\Z_{\gp}$ (since $\gp >N$) whose derivative $\mathrm{D}_{t_0/\gp}f$ has determinant of size at least ${\gp}^{-k-A}$.
By Lemma~\ref{lem:openmapping}
\begin{align*}
f\big(t_0/\gp+\Z_\gp[k+A+1]^{\dim(\mfrak)}\big) \supset f(t_0/\gp)+\mfrak[2k+2A+1].
\end{align*}
Taking the exponential,
the image of $\varphi$ on $\gp\Z_{\gp}^{\dim(\mfrak)}$ contains 
\begin{align*}
\exp\big(f(t_0/\gp)+\mfrak[2k+2A+1]\big) = \varphi(t_0)M[2k+2A+1].
\end{align*}
Here, the equality follows for instance from the Baker-Campbell-Hausdorff formula and the fact that $\mfrak[2k+2A+1]$ is invariant under brackets with $\mfrak[1]$.
Thus,
\begin{align*}
\varphi\big(\gp\Z_{\gp}^{\dim(\mfrak)}\big)^{-1}\varphi\big(\gp\Z_{\gp}^{\dim(\mfrak)}\big)
&\supset M[2k+2A+1]\varphi(t_0)^{-1}\varphi(t_0)M[2k+2A+1]\\
&= M[2k+2A+1]
\end{align*}
and we conclude the lemma with $\ref{a:implicitfctthm}=2A+1$.
\end{proof}

In the proof of Theorem~\ref{thm:main equi}, we know precise invariance with respect to a list of nilpotents effectively generating $H_\gp$.
The existence of such a list is a property of the good prime $\gp$ and the content of the following lemma.

\begin{lemma}\label{lem:Heffgen}
If the good prime $\gp$ is sufficiently large depending on $N$, 
the group $G_\gp$ is $0$-generated by (not necessarily distinct) nilpotents $\vpz_1,\ldots,\vpz_{\dim(\gfrak)}\in \gfrak_\gp[0]$ of non-zero pure weight.
The same holds for $H_\gp$.
\end{lemma}

\begin{proof}
This is certainly well known; for the reader's convenience and lack of explicit reference, we present a proof; we prove the lemma for $G_{\gp}$, the proof for $H_{\gp}$ is identical.  

Let $\mathcal{W}$ denote the set of non-zero weights for $\diagsubgrp$ in $\gfrak_\gp$ where $A$ is the diagonal subgroup of the principal $\SL_2$ in \eqref{eq:at in SL2}.
As in \S\ref{sec: adjoint rep}, we let $\gfrak_{\gp}^{(\lambda)}\subset\gfrak_{\gp}$ denote the space of vectors with weight $\lambda\in \mathcal{W}$.
It follows from Lemma~\ref{lem: weight spaces are integral} that $\underline{\gfrak_{\gp}^{(\lambda)}}=\underline{\gfrak_{\gp}}^{(\lambda)}$ for all~$\lambda$.

We first claim that $\underline{\mathfrak G_{\gp}}$ is generated by the unipotent subgroups $\exp(\cdot \underline{\vpz})$ where $\underline{\vpz} \in \underline{\gfrak_{\gp}}^{(\lambda)}$ and $\lambda \in \mathcal{W}$.
Observe here that $\underline{\mathfrak G_{\gp}}$ is connected since $\G$ is simply connected, see~\cite[\S3.8]{Tits-Corvallis}.
Recall that the special fiber of a simple factor $\mathfrak G_{\gp}'$ of $\mathfrak G_{\gp}$ is an $\mathbb{F}_\gp$-simple factor $\underline{\mathfrak G_{\gp}}'$ of the special fiber $\underline{\mathfrak G_{\gp}}$.
Given any simple factor $\underline{\mathfrak G_{\gp}'}$, the projection of $\underline{\theta_\gp(\SL_2)}$ is non-trivial by Lemma~\ref{lem:nofactorscontainingHp}. The subgroup of the factor generated by all unipotent subgroups for non-zero pure weight is a normal subgroup and contains the projection of $\underline{\theta_\gp(\SL_2)}$ and so is equal to $\underline{\mathfrak G_{\gp}'}$.
This proves the claim.

Using $\gp\gg 1$ and a dimension increase argument (see e.g.~\cite[Prop.~5.2]{EinsiedlerGhosh}), there exists $\vpz_1,\ldots,\vpz_{\dim\G}$ with $\vpz_i\in\gfrak^{(\lambda_i)}$, where $\lambda_i\in\mathcal W$ for all $i$, so that
\[
\underline{\varphi}: \underline{\mathbf A}^{\dim\G}\to \underline{\mathfrak G_{\gp}}
\]
defined by $\underline{\varphi}(t_1,\ldots,t_{\dim\G})=\exp(t_1\underline{\vpz_1})\cdots\exp(t_{\dim\G}\underline{\vpz_{\dim\G}})$ is a dominant map.

The rest of the argument is similar to \cite[\S6.9]{EMMV}; we recall some of the details. 
Using $\gp\gg1$ and $\deg \underline{\varphi}\ll N^\star$ it follows that $\underline{\varphi}$ is a separable map. 
In view of this and using $p\gg1$ again, a simple pigeonhole argument shows that there is some $\underline{t}\in\mathbb F_\gp^{\dim\G}$ so that $
{\rm D}_{\underline{t}}\underline{\varphi}$ has full rank.
Define $\varphi:\Q_\gp^{\dim\G}\to \G(\Q_p)$ by lifting $\underline{\varphi}$ in the obvious way.
The above discussion implies that there is some $t\in\Z_\gp^{\dim\G}$ so that $\det({\rm D}_t\varphi)\in\Z_\gp^\times$, as we claimed. 
\end{proof}

\subsection{An effective version of Greenberg's theorem}\label{sec:effGreenberg}

As mentioned, Greenberg's theorems \cite{Greenberg1,Greenberg2} can be understood as an analogue of Hensel's lifting lemma for non-smooth varieties. 
Vaguely, one would like to assert that whenever $\vpz \in \Z_{\gp}^m$ is a point with
\begin{align*}
f_1(\vpz) \equiv \ldots \equiv f_n(\vpz) \equiv 0 \;\; (\mod \gp^k)
\end{align*}
for $f_1,\ldots,f_n$ polynomials over $\Z$, then $\vpz$ is $p^{-\star k}$-close to a $\Z_\gp$-point on the variety $\{f_1= \ldots = f_n =0\}$.
Alternatively, this can be seen as a $p$-adic analogue of Lojasiewicz's inequality for polynomials.

In Greenberg's work, the dependency on the polynomials $f_1,\ldots,f_n$ is inexplicit though the method can be effectivized using additionally an effective Nullstellensatz.
Recall that the height of a polynomial over $\Z$ is the maximum of the absolute values of its coefficients.

\begin{theorem}\label{thm:greenberg}
Let $f_1,\ldots,f_n \in \Z[t_1,\ldots,t_m]$ have total degree at most $d$ and height at most $h$.
There exists $A>0$ depending only on $n,m,d$ so that the following holds:
Let $\gp$ be a prime.
Suppose that $\wpz_1,\ldots,\wpz_m \in \Z_{\gp}$ are such that 
$$f_j(\wpz_1,\ldots,\wpz_m) \equiv 0 \;\; (\mod \gp^k)$$
for all $j$ and for $k \in \N$ with $k \geq A\log_{\gp}(h)+A$.

Then there exist $\wpz_1',\ldots,\wpz_m' \in \Z_{\gp}$ such that $f_j(\wpz_1',\ldots,\wpz_m') = 0$ for all $j$ and for all $1 \leq i \leq m$
\begin{align*}
\wpz_i' \equiv \wpz_i \;\; (\mod \gp^{\lceil k/A\rceil}).
\end{align*} 
\end{theorem}

The proof of Theorem~\ref{thm:greenberg} can be found in \cite[App.~A]{LMMS}; it follows Greenberg's proof by induction on the dimension making explicit the dependencies on the height.
The base case of the empty variety is a direct consequence of an effective Nullstellensatz --- see e.g.~\cite[Thm.~IV]{MasserWustholz}.

\subsection{Preliminaries from geometric invariant theory}\label{sec:geominvth}

For the proof it will be convenient to know that $\G$-orbits of certain specific vectors in representations of $\G$ are closed.
We establish these facts here using geometric invariant theory. 

Recall the following: suppose we are given a representation $\Mbf \to \GL(V)$ of a reductive group $\Mbf$ over an algebraically closed field of characteristic zero and a vector $\vpz \in V$.
If $\overline{\Mbf.\vpz} \setminus \Mbf.\vpz$ non-empty, there exists a vector $\vpz' \in \overline{\Mbf.\vpz} \setminus\Mbf.\vpz$ and a cocharacter $\lambda: \G_m \to \Mbf$ such that $\lim_{t \to 0}\lambda(t). \vpz = \vpz'$.
This is a result due to Mumford; see for instance \cite[Ch.~2,\S1]{Mumford}, \cite[Thm.~1.4]{Kempf}, or \cite[Thm.~4.2]{Birkes} where the latter contains an elementary proof due to Richardson.

\begin{lemma}\label{lem:closedorbitsembeddings}
Let $\Mbf$ be a reductive group over an algebraically closed field $K$ of characteristic zero.
For $r>0$ let constants $a_{ijk} \in K$, $1 \leq i,j,k\leq r$, be given and assume that there is an (abstract) reductive Lie algebra with structure coefficients $a_{ijk}$.
Let $\mfrak$ be the Lie algebra of $\Mbf$ and consider the subvariety $\Vbf$ of $\mfrak^r$ consisting of tuples $(\vpz_{i})_{1 \leq i \leq r}$ such that $[\vpz_i,\vpz_j] = \sum_{k}a_{ijk}\vpz_k$ for all $1 \leq i,j\leq r$.
Then for any point $\vpz= (\vpz_{i})_i \in \Vbf(K)$ the $\Mbf$-orbit $\Mbf.\vpz$ is closed.
\end{lemma}

Note that we will only apply the lemma for $\Mbf=\G$ and the variety of $\mathfrak{sl}_2$-triples in $\gfrak^3$; see Proposition~\ref{prop:varietyofembeddings} below.

\begin{proof}
Let $\vpz$ be as in the lemma and suppose by contradiction that $\overline{\Mbf.\vpz} \setminus \Mbf.\vpz $ is non-empty.
Then there exists a cocharacter $\lambda: \G_m \to \Mbf$ with $\lim_{t \to 0}\lambda(t).\vpz = \vpz' \in \overline{\Mbf.\vpz} \setminus \Mbf.\vpz$.
By construction, all components $\vpz_i$ of $\vpz$ belong to the Lie algebra of the parabolic subgroup
\begin{align*}
\Pbf = \{g \in \Mbf: \lim_{t \to 0}\lambda(t)g\lambda(t)^{-1} \text{ exists in }\Mbf\}.
\end{align*}
The components $\vpz'_i$ of $\vpz'$ are invariant under the adjoint action by $\lambda(\G_m)$ and hence belong to the Lie algebra of the centralizer $\Lbf$ of $\lambda(\G_m)$ in $\Pbf$.
The Lie algebra $\sfrak$ spanned by $\vpz_1,\ldots,\vpz_r$ is a factor of the abstract Lie algebra over $K$ with structure coefficients $(a_{ijk})_{ijk}$ and hence reductive.
Thus, there exists a Levi subalgebra of $\Lie(\Pbf)$ containing $\sfrak$.
By the Levi-Malcev theorem, there is some $u \in \Rbf_{u}(\Pbf)(K)$ such that $\Ad(u)\sfrak$ is contained in the Levi subalgebra $\Lie(\Lbf)$. 
Here, $\Rbf_{u}(\Pbf)$ is the unipotent radical of $\Pbf$.
But then necessarily 
\begin{align*}
\vpz' = \lim_{t \to 0}\lambda(t).\vpz = \lim_{t \to 0}\lambda(t).(u.\vpz) = u.\vpz
\end{align*}
where we used $\lambda(t) u \lambda(t)^{-1} \to \id$ in the second equality and the fact that $u.\vpz$ is fixed by $\lambda(\G_m)$ in the third one. This contradicts our choice of $\vpz'$ and  proves the lemma.
\end{proof}

\subsection{A variety of homomorphisms}\label{sec:varietyofembeddings}
The proof of Theorem~\ref{thm:effgen-intro} relies on Greenberg's theorem in the form given by Theorem~\ref{thm:greenberg} applied to suitable varieties. 
We will now a building block of these varieties, specifically the variety of homomorphisms $\mathfrak{sl}_2 \to \gfrak$.

\begin{proposition}\label{prop:varietyofembeddings}
Consider the subvariety $\Ebf_{\gfrak}$ of $\gfrak^3$ defined over $\Q$ consisting of tuples $(\wpz_+,\wpz_0,\wpz_-)$ satisfying the relations
\begin{align*}
[\wpz_0,\wpz_+] = 2 \wpz_+,\
[\wpz_0,\wpz_-] = -2 \wpz_-,\
[\wpz_+,\wpz_-] =  \wpz_0
\end{align*}
(i.e.~$\mathfrak{sl}_2$-triples). 
Then $\Ebf_{\gfrak}$ is invariant under the (adjoint) $\G$-action on $\gfrak^{3}$ and consists of finitely many closed orbits under that action.
Moreover, the homomorphism $\theta_\gp$ defines a point in $\Ebf_{\gfrak}(\Z_\gp) = \Ebf_{\gfrak}(\Q_\gp) \cap \gfrak_\gp[0]^3$.
\end{proposition}

As is apparent from the above, the variety $\Ebf_{\gfrak}$ of homomorphisms $\mathfrak{sl}_2 \to \gfrak$ is defined by polynomials of height $\ll \height(\G)^\star$.

\begin{proof}
It is clear that $\Ebf_{\gfrak}$ is $\G$-invariant.
By the representation theory of $\mathfrak{sl}_2$, there are finitely many $\SL_N(\overline{\Q})$-conjugacy classes of Lie algebra homomorphisms $\mathfrak{sl}_2(\overline{\Q}) \to \mathfrak{sl}_N(\overline{\Q})$.
By work of Richardson \cite[Thm.~7.1]{Richardson}, the intersection of any such $\SL_N(\overline{\Q})$-conjugacy class with the set of homomorphisms $\mathfrak{sl}_2(\overline{\Q}) \to \gfrak\otimes \overline{\Q}$ is a finite union of $\G(\overline{\Q})$-conjugacy classes. 
This shows that $\Ebf_{\gfrak}$ consists of finitely many $\G$-orbits.
By Lemma~\ref{lem:closedorbitsembeddings}, any such orbit is closed.
Lastly, the point defined by $\theta_\gp$ is precisely obtained by applying the derivative of $\theta_\gp$ to the standard $\mathfrak{sl}_2$-triple in $\mathfrak{sl}_2(\Z_\gp)$ (we use here a property of the good prime through Lemma~\ref{lem:principal SL2}).
The claim follows.
\end{proof}

The variety $\Ebf_{\gfrak}$ constructed above might consist of various $\G$-orbits that could be e.g.~of different dimension (or not Galois conjugate over $\overline{\Q}$); this can be problematic for our purposes.

\begin{example}\label{exp:sl2triples in sl3}
Consider the variety $\Ebf_{\mathfrak{sl}_3}$.
A non-zero $\mathfrak{sl}_2$-triple in $\mathfrak{sl}_3$ corresponds either to an irreducible or a reducible non-zero three-dimensional representation of $\mathfrak{sl}_2$.
The so-obtained subvarieties of $\Ebf_{\mathfrak{sl}_3}$ are $\G$-invariant and have different dimension ($8$ in the former and $7$ in the latter case).
\end{example}

One can show (using Gr\"obner-based algorithms) that the height of $\Q$-irreducible components of $\Ebf_\gfrak$ (or any other affine variety defined over $\Q$) is also polynomially controlled.
In fact, such estimates are used crucially in the proof of Theorem~\ref{thm:greenberg} given in \cite[App.~A]{LMMS}.
Nevertheless, we employ here a soft argument to control the height of unions of $\Q$-irreducible components of equal dimension.

\begin{lemma}\label{lem:exactdim}
Let $\pi: \SL_N \to \GL_n$ be a rational representation of $\SL_N$ defined over $\Q$ and let $r,d \geq 1$.
Then there exist $C,A>0$ satisfying the following statement.

Let $\Vbf \subset \Abf^n$ be a $\G$-invariant subvariety defined over $\Q$ such that
\begin{itemize}
\item $\Vbf$ is defined by polynomials $f_1,\ldots,f_r \in \Z[x_1,\ldots,x_n]$ of height at most $h>2$ and degree at most $d$, and that
\item $\Vbf$ consists of finitely many closed $\G$-orbits.
\end{itemize}
Then for any dimension $D$ the subvariety of $\Vbf$ consisting of orbits of dimension exactly $D$ is defined by at most $C$ many integral polynomial equations of height at most $Ch^A \height(\G)^A$ and degree at most $C$.
\end{lemma}

\begin{proof}
Fix a $\Q$-basis $\wpz_1,\ldots,\wpz_{\dim(\gfrak)}$ of $\gfrak$ consisting of integral vectors of height $\ll\height(\G)$.
The fact that all orbits are closed implies that for any point $x$ in the variety the derivative $\gfrak \to \mathrm{T}_{x}\Vbf$ of the orbit map $\G \to \G.x$ is surjective.

Consider the $n-D+1$-minors of the derivative $\mathrm{D}_x f := (\partial_{x_j}f_i(x))_{1\leq i \leq r,1 \leq j \leq n}$ of the polynomial map $f=(f_1,\ldots,f_r)\colon \Abf^n \to \Abf^r$. 
The common zero locus of these minors together with $f_1,\ldots,f_r$ defines the subvariety $\Vbf'$ of $\Vbf$ of $\G$-orbits of dimension at least $D$. Note that $\Vbf'$ is also defined over $\Q$.

We now impose additional equations to obtain dimension exactly $D$. 
Let $\Pi(x)$ be the matrix with columns $\mathrm{D}\pi(\wpz_1)x,\ldots,\mathrm{D}\pi(\wpz_{\dim(\gfrak)})x$; this is the derivative of the map $(t_1,\ldots, t_{\dim(\gfrak)}) \to \pi(\exp(\sum_{i}t_i \wpz_i))x$.
The common zero locus of the $D+1$-minors of $\Pi(\cdot)$ defines a further subvariety $\Vbf''$ of $\Vbf'$. 
The geometric components of $\Vbf''$ have dimension at most (and hence exactly) $D$.
We have thus found the desired subvariety of $\Vbf$ and observe that the polynomial equations defining it satisfy the required properties.
\end{proof}

In the following we let $\Ebf_{\gfrak}'\subset \Ebf_{\gfrak}$ be the union of closed $\G$-orbits of dimension equal to the dimension of the orbit of the point corresponding to $\theta_\gp$ as in Proposition~\ref{prop:varietyofembeddings}.
Since $\theta_\gp$ is non-trivial, the latter dimension is non-zero and for any field $K$ of characteristic zero a point in $\Ebf_{\gfrak}'(K)$ yields an embedding $\mathfrak{sl}_2(K) \to \gfrak\otimes K$.
In view of Lemma~\ref{lem:exactdim}, $\Ebf_{\gfrak}'$ is defined by $O_N(1)$-many integral polynomial equations of height $\ll \height(\G)^\star$ and degree $O_N(1)$.

\subsection{An open mapping theorem}\label{sec:openmappingthm}
Suppose that $\Vbf$ is an affine $\Q$-variety consisting of finitely many closed $\G$-orbits.
It seems natural to suspect that nearby points on $\Vbf(\Q_{\gp})$ should be related by a small element of $\G(\Q_{\gp})$.
This also entails a local separation statement for the $\G$-orbits in $\Vbf$.
In this subsection, we prove a version of this making no use of any properties of the specific varieties from \S\ref{sec:varietyofembeddings}, our statements depend effectively on the height of polynomials defining $\Vbf$.

\begin{proposition}\label{prop:localseparatedness}
Let $\pi: \SL_N \to \GL_n$ be a rational representation of $\SL_N$ defined over $\Q$ and let $r,d \geq 1$. 
There exists $A>0$ satisfying the following statement.

Let $\Vbf \subset \Abf^n$ be a $\G$-invariant subvariety defined over $\Q$ such that
\begin{itemize}
\item $\Vbf$ is defined by polynomials $f_1,\ldots,f_r \in \Z[x_1,\ldots,x_n]$ of height at most $h>2$ and degree at most $d$, and
\item $\Vbf$ consists of finitely many closed $\G$-orbits of equal dimension.
\end{itemize}
Let $\gep>2$ be a rational prime.
Then there is some $k_0 \geq 1$ with 
\begin{align*}
\gep^{k_0} \leq \gep^A \height(\G)^Ah^A
\end{align*}
so that for any $x,y \in \Vbf(\Z_{\gep}) = \Vbf(\Q_{\gep}) \cap \Z_\gep^n$ with
\begin{align*}
\|x-y\| < \gep^{-2k_0}
\end{align*}
there exists $g \in G_{\gep}[0]$ with $\pi(g)x = y$ and $\|g-\id\| \leq \|x-y\| \gep^{{k_0}}$.
\end{proposition}

We remark that Proposition~\ref{prop:localseparatedness} could be refined in various ways e.g.~one can drop the integral assumption on the points $x,y$ at the cost of including the denominator in the resulting estimate for $g$.
The version given above is however sufficient for our purposes.

\begin{proof}
Let $D$ be the common dimension of the $\G$-orbits.
We first choose the `level' parameter ${k_0}$.
Let $J_1 \subset \Q[x_1,\ldots,x_n]$ be the ideal generated by $f_1,\ldots,f_r$ and all $n-D$-minors of the derivative $\mathrm{D}_x f := (\partial_{x_j}f_i(x))_{1\leq i \leq r,1 \leq j \leq n}$ of the polynomial map $f=(f_1,\ldots,f_r)\colon \Abf^n \to \Abf^r$.
As $\Vbf$ is smooth and all components have dimension $D$, the zero locus of $J_1$ is empty or, equivalently, $J_1 = \Q[x_1,\ldots,x_n]$.
By the effective Nullstellensatz (see e.g.~\cite[Thm.~IV]{MasserWustholz} or \cite[\S4.11]{LMMS}) there is some non-zero $B_1 \in \N$ of height at most $C_1h^{A_1}$ which is presented by a linear combination in the relations $f_i$ and the minors with integral polynomial coefficients of height $\ll h^\star$ and degree $\ll_{d,r,n} 1$.
Here, $A_1,C_1 >0$ depend only on $d,r,n$.
This implies in particular that if there were a point $x \in \Vbf(\Z_{\gep})$ such that all $n-D$-minors are congruent to $0$ modulo ${\gep}^k$ then $B_1 \equiv 0 \mod {\gep}^k$. 
We assume that ${\gep}^{k_0} > C_1 h^{A_1}$ so that $B_1 \not \equiv 0 \mod {\gep}^{k_0}$.

We proceed now similarly for the $\G$-action. 
Fix a $\Q$-basis $\wpz_1,\ldots,\wpz_{\dim(\gfrak)}$ of $\gfrak$ consisting of integral vectors of height $\ll\height(\G)$.
Recall that all orbits are closed and of dimension $D$.
This implies that the derivative $\gfrak \to \mathrm{T}_x\Vbf$ of the orbit map $\G \to \G.x$ is surjective for any point $x$ in the variety and the rank of the derivative at the identity is $D$, independent of the point $x$.
Let $\Pi(x)$ be the matrix with columns $\mathrm{D}\pi(\wpz_1)x,\ldots,\mathrm{D}\pi(\wpz_{\dim(\gfrak)})x$ for $x \in \Abf^n$; 
for $\Q_\gep$-points $x$, this matrix represents the derivative of the map $(t_1,\ldots, t_{\dim(\gfrak)}) \to \pi(\exp(\sum_{i}t_i \wpz_i))y$ at zero.
Let $J_2 \subset \Q[x_1,\ldots,x_n]$ be the ideal generated by $f_1,\ldots,f_r$ and all $D$-minors of $\Pi(\bullet)$.
As before, $J_2= \Q[x_1,\ldots,x_n]$.
Now proceed similarly to find a non-zero constant $B_2 \in \Z$ of height at most $C_2(\height(\G)h)^{A_2}$ such that if there is a point $x \in \Vbf(\Z_{\gep})$ with all $D$-minors of $\Pi(x)$ congruent to $0$ modulo ${\gep}^k$ then $B_2 \equiv 0 \mod {\gep}^k$.
Again, assume ${k_0}$ is sufficiently large, e.g.
\begin{align*}
k_0 = \lceil \max\{\log_{\gep}(C_1h^{A_1}),\log_{\gep}(C_2(\height(\G)h)^{A_2})\}\rceil +1,
\end{align*}
so that $B_2\not \equiv 0 \mod {\gep}^{k_0}$.

Now let $x,y \in \Vbf(\Z_{\gep})$ with $y-x = {\gep}^{a}\vpz$ for some primitive integral vector $\vpz \in \Z_{\gep}^n$ and $a > 2{k_0}$. 
For all $1 \leq i \leq r$
\begin{align*}
0 = f_i(y) = f_i(x + {\gep}^{a}\vpz) 
= {\gep}^{a}(\partial_{1}f_i(x),\ldots,\partial_{n}f_i(x))\vpz 
+ O({\gep}^{2a})
\end{align*}
so that $(\mathrm{D}_{x}f) \vpz \in \gep^a \Z_\gep^r$.
By the Smith normal form over the PID $\Z_{\gep}$, there exist $h_1\in \GL_{r}(\Z_{\gep})$ and $h_2 \in \GL_n(\Z_{\gep})$ such that $h_1(\mathrm{D}_{x}f)h_2$ is nonzero only on the diagonal where it has entries ${\gep}^{k_1},\gep^{k_2},\ldots,{\gep}^{k_{n-D}},0,\ldots,0$.
Note that our choice of ${k_0}$ implies $\sum_{i\geq 1} k_i \leq  {k_0}$.
Thus, $(h_2^{-1}\vpz)_i \equiv 0 \mod {\gep}^{a-{k_0}}$ for all $1 \leq i \leq n-D$.
It follows that there exist $\vpz' \in \Z_{\gep}^n$ with
\begin{align*}
\vpz \equiv \vpz' \mod {\gep}^{a-{k_0}},\quad (\mathrm{D}_{x}f) \vpz' = 0;
\end{align*}
Explicitly, take $\vpz' = h_2(0,\ldots,0,(h_2^{-1}\vpz)_{n-D+1},\ldots)^t$.
In other words, what we have shown above is that the displacement $\vpz$ is close to being tangential.

We shall now try to realize the above perturbation $\vpz'$ of the displacement $\vpz$ within derivatives coming from $\gfrak$.
For any $t_1,\ldots,t_{\dim(\gfrak)} \in \gep \Z_{\gep}$ we have
\begin{align*}
\pi\big(\exp(\sum_{i}t_i \wpz_i)\big)x &=
x + \sum_{i} t_i\mathrm{D}\pi(\wpz_i)x + O(\max_i |t_i|_\gep^2) \\
&= x + \Pi(x)t + O(\max_i |t_i|_\gep^2)
\end{align*}
where $t = (t_1,\ldots,t_{\dim(\gfrak)})^t$.
As the image of $\Pi(x)$ is exactly the tangent space at the point $x$ (the kernel of $\mathrm{D}_{x}f$), there exists $t \in \Q_\gep^{\dim(\gfrak)}$ with $\Pi(x)t = {\gep}^{a}\vpz'$.
By using the Smith normal form of $\Pi(x)$, we conclude that $t$ can be chosen so that $t \in {\gep}^{a-{k_0}}\Z_\gep^{\dim(\gfrak)}$. 
Then $g = \exp(\sum_i t_i \wpz_i)$ satisfies
\begin{align*}
\pi(g) x &\equiv x + \Pi(x)t
\equiv x + {\gep}^{a} \vpz'
\equiv x + {\gep}^{a} \vpz
\equiv y \mod \ell^{2a-2k_0}.
\end{align*}
As $2a-2{k_0} > a$, we have found a new point on the local $\G(\Q_{\gep})$-orbit through $x$ which is closer to $y$ than $x$ was. 

To conclude we repeat the above procedure.
By the above, there exists $g_1 \in \G(\Z_{\gep})$ such that $\|\pi(g_1)x-y\| < \|x-y\|$ and $\norm{g_1-\id} \leq \norm{x-y}{\gep}^{k_0}$.
By induction, there exist $g_1,g_2,\ldots$ such that
\begin{align*}
\|\pi(g_n\cdots g_1)x-y\| &< \|\pi(g_{n-1}\cdots g_1)x-y\|,\\ 
\norm{g_n-\id} &\leq \norm{\pi(g_{n-1}\cdots g_1)x-y}{\gep}^{k_0}
\end{align*}
for every $n\geq 1$.
In particular, $\norm{g_n\cdots g_1 - \id} \leq \norm{x-y}{\gep}^{k_0}$. Taking the limit we obtain the proposition.
\end{proof}

\subsection{Proof of Theorem~\ref{thm:effgen-intro}}\label{sec:proofeffgen}

The proof of Theorem~\ref{thm:effgen-intro} proceeds by induction on the following statement.

\begin{proposition}\label{prop:effgenind}
There exists $\constk\label{k:effgenind}>0$ with the following property.
For $n\leq \dim(\gfrak)$ let
$\extranil=(\vpz_1,\ldots,\vpz_n) \in \gfrak_{\gp}[0]^n$ be a list of nilpotent elements of pure weight.
Let $\Mbf$ be the Zariski closure of the group generated by $\theta_\gp(\SL_2(\Q_\gp))$ and the one-parameter unipotent groups $\{\exp(t\vpz_i): t \in \Q_{\gp}\}$ for $1\leq i \leq n$.
Then for any 
\begin{align*}
k \geq (\log_{\gp}(\height(\G))+1)/\ref{k:effgenind}
\end{align*}
one of the following is true.
\begin{enumerate}[label=\textnormal{(\alph*)}]
    \item\label{item:ind-generated} The group $M = \Mbf(\Q_{\gp})$ is $k$-generated by a list of nilpotents where each element is contained in the list $\extranil$ or is equal to $\zpz^+$ or $\zpz^-$ (cf.~\eqref{eq:direction z}).
    \item\label{item:ind-perturbation} There exist a list of nilpotent elements $\extranil' = (\vpz_1',\ldots,\vpz_n') \in \gfrak_{\gp}[0]^n$ with
    \begin{align*}
    \norm{\vpz_i'-\vpz_i}_{\gp} \leq {\gp}^{-\ref{k:effgenind} k} \qquad\text{for all }1\leq i \leq n
    \end{align*}
    so that the Zariski closure of the group generated by $\theta_\gp(\SL_2(\Q_\gp))$ and the one-parameter groups $\exp(\Q_\gp\vpz_i')$ for $1 \leq i \leq n$ has strictly smaller dimension than $\Mbf$.
    Moreover, for each $1 \leq i \leq n$ the nilpotent $\vpz_i'$ has the same pure weight as $\vpz_i$ and, if $\vpz_i$ is highest weight, then so is $\vpz_i'$.
\end{enumerate}
\end{proposition}

\begin{proof}
The proof, which will be completed in several steps, relies heavily on the effective version of Greenberg's theorem given in Theorem~\ref{thm:greenberg}. 
We begin by constructing a variety of controlled complexity (using \S\ref{sec:varietyofembeddings}), which serves as the parameter space.

\textsc{Step 0: Construction of the parameter space.}
Let $\lambda_1,\ldots,\lambda_n \in \Z$ be the weights of the nilpotents in $\extranil$ given in the assumptions of the proposition.
Note that there are finitely many options for these weights depending on $N$.
Let $\Ebf_\gfrak'$ be the variety defined after the proof of Lemma~\ref{lem:exactdim}.
In particular, $\Ebf_\gfrak'$ consists of finitely many closed $\G$-orbits and is defined by $O_N(1)$-many integral polynomial equations of height $\ll \height(\G)^\star$ and of degree $O_N(1)$. 

Let $\Vbf_0 \subset \gfrak^{n}$ be the variety of nilpotent tuples and let $\Vbf \subset \Ebf_\gfrak' \times \Vbf_0$ be the subvariety of points $((\wpz_+,\wpz_0,\wpz_-),(\zpz_1,\ldots,\zpz_{n}))$ satisfying the additional equations
\begin{align*}
[\wpz_0,\zpz_i] = \lambda_i \zpz_i \text{ for all } 1\leq i\leq n
\end{align*}
as well as $[\wpz_+,\zpz_i] = 0$ if the nilpotent element with index $1 \leq i \leq n$ given in the proposition is of highest weight.
When convenient, we write $\zpz_{n+1}=\wpz_+$ and $\zpz_{n+2}=\wpz_-$.
By construction, $\Vbf$ is $\G$-invariant with a surjective $\G$-equivariant projection $\Vbf \to \Ebf_\gfrak'$ and is defined by $O_N(1)$-many polynomial equations of height $\ll \height(\G)^\star$ and of degree $O_N(1)$.

The data in the proposition gives rise to an `initial' point $\xpz_0  \in V_{\gp}[0] = \Vbf(\Z_\gp)$; it consists of the $\mathfrak{sl}_2$-triple given by Proposition~\ref{prop:varietyofembeddings} and the nilpotents in $\extranil$ (given by assumption).
\medskip

\textsc{Step 1: Construction of a subvariety.}
We construct the subvariety $\Wbf \subset \Vbf$ of points $\xpz = (\cdot,(\zpz_i))\in \Vbf$
for which the group generated by the one-parameter unipotent subgroups $\{\exp(\cdot \zpz_i)\}_{1\leq i \leq n+2}$ has dimension strictly smaller than $d= \dim(\Mbf)$.
A priori, it is unclear that this is indeed a subvariety and, if it is, that it can be defined with polynomials of controlled height.

Fix a list of indices $\mathcal{I} = (i_1,\ldots,i_{d}) \in \{1,\ldots,n+2\}^{d}$ (repetitions are allowed).
Given a point $\xpz$ of $\Vbf$ and the nilpotent elements $\zpz_i$ within it, define the polynomial map
\begin{align*}
\Phi_{\xpz}^{\mathcal{I}}: t = (t_1,\ldots,t_{d}) \in \Abf^d \mapsto \exp(t_1 \zpz_{i_1})\cdots \exp(t_d \zpz_{i_d}) \in \G.
\end{align*}
The $d$-minors of $\mathrm{D}_{t}\Phi_{\xpz}^{\mathcal{I}}$ are polynomials in $\Q[\Vbf][t]$.
Let $\mathcal{F} \subset\Q[\Vbf]$ be the set of coefficients of these polynomials where we run over the minors of all derivatives for all $\mathcal{I}$ of length $d$.
Let $\Wbf$ be the subvariety of $\Vbf$ defined by the polynomials in $\mathcal{F}$.
This will be the subvariety we apply Greenberg's theorem (Theorem~\ref{thm:greenberg}) to.
Notably, the complexity of $\Wbf$ satisfies the analogous bound as $\Vbf$ does.
Returning to the initial goal of this step, a simple dimension increase argument shows that $\xpz = (\cdot,(\zpz_i)) \in \Wbf$ if and only if the group generated by the one-parameter unipotent subgroups $\{\exp(\cdot \zpz_i)\}$ has dimension strictly smaller than $d$.
In particular, $\Wbf$ is invariant under the $\G$-action.

Specializing to the `initial' point $\xpz_0$, one would like some $d$-minor of the derivative of one of the maps $\Phi_{\xpz_0}^{\mathcal{I}}$ to be large.
We shall see that this amounts to saying that $\xpz_0$ is not too close to $\Wbf(\Q_{\gp})$.

\textsc{Step 2: Applying Greenberg's theorem.}
Let $k \geq 1$.
Suppose first that there exists a polynomial $f \in \mathcal{F}$ for which $|f(\xpz_0)|\geq {\gp}^{-k}$.
We prove \ref{item:ind-generated} in this case.
Let $\mathcal{I}$ be the index list $f$ is associated to. 
The so-defined map $\Phi_{\xpz_0}^{\mathcal{I}}$ maps into $\Mbf$. 
The minor of $\mathrm{D}_t\Phi_{\xpz_0}^{\mathcal{I}}$ of which $f$ is a coefficient is hence a polynomial in $t$ with one coefficient of size at least ${\gp}^{-k}$.
In particular, there exists $t \in \Z_{\gp}^d$ so that $\mathrm{D}_t \Phi_{\xpz_0}^{\mathcal{I}}$ has a minor of size at least ${\gp}^{-k}$ (using that $\gp$ is assumed sufficiently large). 
Therefore $M$ is $k$-generated by the nilpotents $\vpz_{i_1},\ldots,\vpz_{i_d}$ where $\mathcal{I} = (i_1,\ldots,i_d)$ and where $\vpz_{n+1}$ resp.~$\vpz_{n+2}$ are the upper resp.~lower nilpotent in the $\mathfrak{sl}_2$-triple associated to the principal $\SL_2$.

Assume now that $|f(\xpz_0)|< {\gp}^{-k}$ for all $f \in \mathcal{F}$.
We apply the effective version of Greenberg's theorem in Theorem~\ref{thm:greenberg} to $\Wbf$ and let $A>0$ be the corresponding exponent (which depends only on $N$).
If $k$ is assumed sufficiently large, there exists a point $\xpz \in \Wbf(\Z_{\gp})$ with
\begin{align*}
    \norm{\xpz-\xpz_0}\leq {\gp}^{- k/A}.
\end{align*}
The same estimate holds true for the images of $\xpz,\xpz_0$ under the projection $\Vbf \to \Ebf_\gfrak'$.
By Proposition~\ref{prop:localseparatedness}, there exists $g \in G_\gp[0]$ with $\norm{g-\id} \leq {\gp}^{- k/A+\star}\height(\G)^\star$ such that $g.\xpz$ and $\xpz_0$ have the same image under the projection. 
Observe that $g.\xpz\in \Wbf(\Q_{\gp})$ and that, by the estimate on $g$, $\norm{g.\xpz-\xpz_0} \leq {\gp}^{- k/A+\star}\height(\G)^\star$.
Write
\begin{align*}
\xpz' = g.\xpz = (\cdot,(\vpz_i')_{i \leq n+2}).
\end{align*}

\textsc{Step 3: Verifying properties in} \ref{item:ind-perturbation}.
We now show that the nilpotents $\vpz_1',\ldots,\vpz_n'$ satisfy the properties in part \ref{item:ind-perturbation}.
Since $\xpz',\xpz_0$ project to the same point in $\Ebf_\gfrak'$, 
the nilpotents $\vpz_1',\ldots,\vpz_n'$ satisfy the weight (and highest weight) requirement of \ref{item:ind-perturbation}.

Let $\Mbf'$ be the group generated by $\theta_\gp(\SL_2)$ and the one-parameter unipotent subgroups obtained from the nilpotents $\vpz_1',\ldots,\vpz_n'$.
It remains to show that $\dim(\Mbf') < d = \dim(\Mbf)$.
Since $\xpz' \in \Wbf$, the Zariski closure $\Mbf''$ of the group generated by the one-parameter subgroups $\{\exp(t\vpz_i'): t \in \Q_{\gp}\}$ for $1 \leq i \leq n+2$ is less than $\dim(\Mbf)$. 
Notice that $\theta_{\gp}(\SL_2(\Q_\gp))$ is $0$-generated by $\vpz_{n+1}',\vpz_{n+2}',\vpz_{n+1}'$ where $\vpz_{n+1}'= \vpz_{n+1},\vpz_{n+2}'= \vpz_{n+2}$.  
Therefore, $\theta_{\gp}(\SL_2(\Q_\gp)) \subset \Mbf''(\Q_{\gp})$ and $\Mbf'' = \Mbf'$ which proves (b) and hence the proposition.
\end{proof}

\begin{proof}[Proof of Theorem~\ref{thm:effgen-intro}]
    The theorem follows from Proposition~\ref{prop:effgenind} by recursion as follows.

We begin by applying Proposition~\ref{prop:effgenind} for $k_1 = \lfloor\delta k\rfloor$.
If Option \ref{item:ind-generated} holds for this choice, we conclude with $\widetilde{\extranil} = \extranil$ and $\alpha=1$.
Otherwise, there exists $\extranil^{(1)}$ as in Option \ref{item:ind-perturbation} so that in particular $\norm{\extranil^{(1)}-\extranil} \leq {\gp}^{-\ref{k:effgenind}k_1}$.

Define recursively $k_j = \lfloor\delta\ref{k:effgenind}k_{j-1}\rfloor$ so that
\begin{align*}
(\delta \ref{k:effgenind})^{j-1} \delta k -j \leq k_j \leq (\delta \ref{k:effgenind})^{j-1} \delta k.
\end{align*}
At the $j$-th step of the induction, we are given a list of nilpotents $\mathcal{N}^{(j-1)}$ satisfying, in particular, $\norm{\mathcal{N}^{(j-1)}-\mathcal{N}}\leq {\gp}^{-\ref{k:effgenind}k_{j-1}}$.
Applying Proposition~\ref{prop:effgenind} with $k_j$ we either conclude with $\alpha = (\delta \ref{k:effgenind})^{j-1}$ and the list $\mathcal{N}^{(j-1)}$ or we find a new list $\mathcal{N}^{(j)}$.

At each step the group generated as in \ref{item:ind-perturbation} decreases in dimension and hence the induction stops after less than $\dim(\gfrak)$ many steps.
\end{proof}

\section{Diophantine points and an effective avoidance principle}\label{sec:prepclosing}

In this section, we recall the notion of Diophantine points introduced by E.L., Margulis, A.M., and Shah in \cite{LMMS} and the effective avoidance results therein.
Given the desired controlled dependence on $\G$ we will also use diameter estimates due to A.M., Salehi-Golsefidy, and Thillmany \cite{MSGT-diameter}.

Let $\mathcal{S} = \{\infty,\gp\} \subset \Sigma$. 
The definitions in \S\ref{sec:setup} transfer seamlessly to $\Q_{\mathcal{S}}$-points where we write $\content_{\mathcal{S}}(\cdot)$ for the corresponding content on $\Q_{\mathcal{S}}^n$ for $\Q_\mathcal{S} = \R \times \Q_{\gp}$ (see specifically \eqref{eq:def-content}).
We write $G_{\mathcal{S}} = \rho(\G(\Q_\mathcal{S}))$, $\Gamma_{\mathcal{S}} = G_{\mathcal{S}} \cap \SL_N(\Z[1/\gp])$, and $X_{\mathcal{S}} = \lquot{\Gamma_{\mathcal{S}}}{G_{\mathcal{S}}}$.
By strong approximation, there is a quotient map
\begin{align*}
\pi_{\mathcal{S}}: X \to X_{\mathcal{S}}.
\end{align*}
Indeed, since $\gp$ is a good prime, all $\Q_\gp$-simple factors of $\G$ are $\Q_\gp$-isotropic. We can thus apply strong approximation for $\G$ (recall that $\G$ is assumed simply connected), the set of places $\mathcal{S}$, and the compact open subgroup $\prod_{\gep\in \Sigma\setminus\mathcal{S}} K_\gep \subset \prod_{\gep \in \Sigma\setminus\mathcal{S}}'\G(\Q_\gep)$ (where $K_\gep = \rho^{-1}(\SL_N(\Z_\gep))$ as in \S\ref{sec:setup}).

\subsection{Bounds on the height in the cusp}

Since $\G$ is assumed to be $\Q$-anisotropic, $X$ is compact.
The following lemma establishes an upper bound for the height in the cusp when $X$ is viewed as a subset of $\lquot{\SL_N(\Q)}{\SL_N(\A)}$.

\Ganisotropic
\begin{proposition}\label{prop:smallvectors}
There exists $\consta\label{a:smallvectors}>0$ such that for any $g \in \G(\A)$ we have
\begin{align*}
\min_{\wpz \in \Q^N\setminus\{0\}}\content(\rho(g)\wpz) \gg \height(\G)^{-\ref{a:smallvectors}}.
\end{align*}
\end{proposition}

The proof below uses geometric invariant theory.
Recall that, by a result of Hilbert, the ring $\overline{\Q}[x_1,\ldots,x_N]^\G$ of $\G$-invariant polynomials in $\overline{\Q}[x_1,\ldots,x_N]$ (where the action is through $\rho$) is finitely generated since $\G$ is semisimple.

\begin{lemma}\label{lem:GITeff}
There exist $\G$-invariant polynomials $f_1,\ldots,f_r \in \Z[x_1,\ldots,x_N]$ of height $\ll \height(\G)^\star$ and degree $O_N(1)$ which generate the ring $\overline{\Q}[x_1,\ldots,x_N]^\G$. Moreover, $r = O_N(1)$.
\end{lemma}

\begin{proof}
We first claim that there are polynomials $f_1',\ldots,f_{r'}' \in \overline{\Q}[x_1,\ldots,x_N]$ of degree at most $d = O_N(1)$ generating $\overline{\Q}[x_1,\ldots,x_N]^\G$ where $r'=O_N(1)$.
This statement depends only on the $\overline{\Q}$-isomorphism class of the group $\G$ and of the representation over $\overline{\Q}$. 
There are finitely many $\overline{\Q}$-isomorphism classes of simply connected groups over $\overline{\Q}$ with dimension at most $\dim(\SL_N)$ and for each such class there are finitely many isomorphism classes of representations in each given dimension.
This implies the above initial claim.

The vector space $V$ of $\G$-invariant polynomials of degree at most $d$ is defined over $\Q$ (since $\G$ and $\rho$ are).
If $\hat{\rho}$ is the $\G$-action on $\Q[x_1,\ldots,x_N]$ and $\vpz_{1},\ldots,\vpz_{\dim(\gfrak)}$ is a $\Q$-basis of $\gfrak$ of integral vectors of height $\ll \height(\G)$ then
\begin{align*}
V = \big\{f \in \Q[x_1,\ldots,x_N]: \deg(f) \leq d,\ \mathrm{D}\hat{\rho}(\vpz_i)f = 0 \text{ for all } 1 \leq i \leq \dim(\gfrak)\big\}.
\end{align*}
By construction, $V$ generates $\overline{\Q}[x_1,\ldots,x_N]^{\G}$.
By Siegel's lemma there exists a basis $f_1,\ldots,f_{\dim(V)}\in \Z[x_1,\ldots,x_N]$ of $V$ over $\Q$ where the coefficients of the polynomials $f_i$ are $\ll \height(\G)^\star$ in absolute value.
This proves the lemma.
\end{proof}

\begin{proof}[Proof of Proposition~\ref{prop:smallvectors}]
We first claim that all $\G$-orbits through points in $\Q^N$ are closed.
It follows from geometric invariant theory and, more specifically, a result of Kempf \cite{Kempf} that for any $\wpz\in \Q^N$ with $\overline{\G.\wpz} \setminus \G.\wpz$ non-empty there exists a cocharacter $\lambda:\G_m \to \G$ defined over $\Q$ with $\lim_{t \to 0}\lambda(t)\wpz \not\in \G.\wpz$.
But $\G$ is $\Q$-anisotropic and so the claim follows.

By geometric invariant theory, see e.g.~\cite[\S1.2]{Mumford} or \cite[\S4.4]{PopovVinberg}, there exists for any two closed $\G$-orbits $\G.\wpz_1\neq \G.{\wpz_2}$ an invariant polynomial $f$ with $f(\wpz_1) \neq f(\wpz_2)$.
Equivalently, the closed orbits are separated by some $f_i$ where $f_1,\ldots,f_r$ are as in Lemma~\ref{lem:GITeff}.
Without loss of generality, we may assume $f_1 =1$ and $f_i(0) = 0$ for all $i>1$.
We may also suppose that the polynomials $f_i$ are homogeneous.

Now let $\wpz \in \Q^N$ be non-zero and let $f_i$ for $i>1$ be such that $f_i(\wpz) \neq 0$.
Then
\begin{align*}
\prod_{\gep \in \Sigma}|f_i(\wpz)|_\gep = \prod_{\gep\in \Sigma}|f_i(g_\gep.\wpz)|_\gep 
\ll \height(\G)^\star \content(g\wpz)^{\deg(f_i)}.
\end{align*}
The left-hand side is equal to $1$ since $f_i(\wpz) \neq 0$ and we deduce that $\content(g\wpz) \gg \height(\G)^{-\star}$.
Hence, the proposition follows.
\end{proof}

\subsection{Diameter estimate}

The following is a consequence of {\cite{MSGT-diameter}} in combination with volume and height comparison in Proposition~\ref{prop: vol complexity'} (Proposition~\ref{prop: vol complexity}) and the estimate of the height in the cusp in Proposition~\ref{prop:smallvectors}.

\Ganisotropic
\begin{theorem}\label{thm:diameter}
There exists $\consta\label{a:diameterestimate}>0$ depending only on $N$ with the following property.
For any $g \in \G(\A)$ there exists $\gamma \in \G(\Q)$ such that
\begin{itemize}
    \item $\norm{\gamma g}_{\gep} = 1$  for any prime $\gep \neq \gp$,
    \item $\norm{\gamma g}_{\gp} \leq \gp^{\ref{a:diameterestimate}}\height(\G)^{\ref{a:diameterestimate}}$, and
    \item $\norm{\gamma g}_\infty \ll 1$.
\end{itemize}
In particular,
\begin{align*}
\norm{\gamma g} \ll \gp^{\ref{a:diameterestimate}}\height(\G)^{\ref{a:diameterestimate}}.
\end{align*}
\end{theorem}

We remark for later purposes that $\gp$ is only required to be a good prime for $X$ (and not for $Y_\data$) to obtain the above theorem.

\begin{proof}
It is shown in {\cite[Thm.~5.5]{MSGT-diameter}} and its proof that the above theorem holds when the second item is replaced by
\begin{align*}
\norm{\gamma g}_{\gp} 
\ll \big(\min_{\wpz \neq 0} \content(g\wpz)\big)^{-\star} \vol(X)^{\star} \gp^\star.
\end{align*}
By Proposition~\ref{prop: vol complexity'} we have $\vol(X) \ll \height(\G)^\star$ and by Proposition~\ref{prop:smallvectors} we have $\min_{\wpz \neq 0} \content(g\wpz) \gg \height(\G)^{-\star}$.
Thus, Theorem~\ref{thm:diameter} follows.
\end{proof}

\subsection{Diophantine points and effective linearization}\label{sec:Diophantine points} 

A connected $\Q$-subgroup $\Lbf\subset \SL_N$ is said to belong to class $\mathcal H$ if $\Lbf(\mathbb C)$ is generated by unipotent subgroups, or equivalently if the radical of $\Lbf$ is unipotent. 
Since $\G$ is assumed $\Q$-anisotropic, any $\Q$-subgroup of $\rho(\G)$ of class $\mathcal{H}$ is semisimple.
We write $\mathcal{H}_\G$ for the collection of connected semisimple $\Q$-subgroups of $\rho(\G)$.

Recall from \S\ref{sec:setup} that given $\Lbf \in \mathcal{H}_\G$ we write $\vpz_{\Lbf}$ for one of the two primitive integral vectors in the rational line
\begin{align*}
\bigwedge^{\dim(\Lbf)} \Lie(\Lbf) \subset \bigwedge^{\dim(\Lbf)} \gfrak.
\end{align*}
We write $\eta_{\Lbf}: g\in \SL_N \mapsto g^{-1}.\vpz_{\Lbf}$ for the right orbit map through $\vpz_{\Lbf}$.

\begin{defin}[{\cite{LMMS}}]\label{def:Diophantine-intro} 
Let $\epsilon : \R^+ \to (0,1)$ be a monotone decreasing function and let $t\in\R^+$.
A point $\Gamma_{\mathcal{S}}g$ in $X_{\mathcal{S}}$ is called \textbf{$(\epsilon,t)$-Diophantine} (for the action of $U=\{u(s)\}$ with generator $\zpz$ as in \eqref{eq:ut in SL2}) if for all $\Lbf \in \mathcal{H}_\G$ with 
$\{e\}\neq\Lbf\neq\rho(\G)$ and $\content_\mathcal{S}(\eta_{\Lbf}(g))< e^t$
\begin{align}\label{eq:dioph-cond-intro-2}
\|\zpz\wedge\eta_{\Lbf}(g)\|\geq\epsilon\bigl(\mathsf \content_\mathcal{S}(\eta_{\Lbf}(g))\bigr).
\end{align}
\end{defin}
Here, we wrote $\zpz \wedge \eta_{\Lbf}(g) = \zpz \wedge \eta_{\Lbf}(g)_\gp$ for simplicity.

We turn to phrasing the main theorem of \cite{LMMS} in our setting.

\begin{theorem}[{\cite{LMMS}}]\label{thm:linearization}
There exists a constant $\consta\label{a:linearization}>0$ depending only on $N$ with the following property.
Let $g \in G_\mathcal{S}$, $t>0$, $k>0$, and $\eta \in (0,\frac{1}{2})$.
Suppose that for all $r>0$
\begin{align*}
\epsilon(r) 
\leq r^{-\ref{a:linearization}} (\height(\G)^{-1}\eta\,{\gp}^{-1})^{\ref{a:linearization}}.
\end{align*}
Then at least one of the following is true:
\begin{enumerate}[label=\textnormal{(\theenumi)}]
\item \emph{(Many Diophantine points)} We have
\begin{align*}
\big|\big\{|s| \leq {\gp}^{k}: \Gamma_{\mathcal{S}} g u(s) \text{ is not } &(\epsilon,t)\text{-Diophantine}\big\}\big| \leq (\height(\G)\gp)^{\ref{a:linearization}} \eta^{\frac{1}{\ref{a:linearization}}}{\gp}^k.
\end{align*}
\item \emph{(Obstruction from a class $\mathcal{H}$-subgroup)} 
There exist a non-trivial proper subgroup $\Lbf \in \mathcal{H}_\G$ such that for all $|s|\leq {\gp}^k$
\begin{align*}
\content_{\mathcal{S}}(\eta_{\Lbf}(gu(s))) 
&\leq  (\height(\G)\eta^{-1}\gp\mathrm{e}^t)^{\ref{a:linearization}},\\
\norm{\zpz \wedge \eta_{\Lbf}(gu(s))} 
&\leq {\gp}^{-k/\ref{a:linearization}} (\height(\G)\eta^{-1}\gp \mathrm{e}^t)^{\ref{a:linearization}}.
\end{align*}
\item \emph{(Obstruction from a normal-subgroup)} 
There exist a non-trivial proper normal $\Q$-subgroup $\Lbf \lhd \rho(\G)$ such that
\begin{align*}
\norm{\zpz \wedge \vpz_{\Lbf}}
&\leq \epsilon\big(\height(\Lbf)^{\frac{1}{\ref{a:linearization}}} (\height(\G)\gp)^{-\ref{a:linearization}}\eta
\big)^{\frac{1}{\ref{a:linearization}}}.
\end{align*}
\end{enumerate}
\end{theorem}

As phrased, the above uses that $\G$ is $\Q$-anisotropic. We have also used Theorem~\ref{thm:diameter} to make the dependency on $\G$ in the main theorem of \cite{LMMS} explicit (see \cite[Lemma 2.8]{LMMS}).
We have additionally absorbed implicit multiplicative constants using the fact that $\gp$ is assumed sufficiently large.

The following corollary of Theorem~\ref{thm:linearization} establishes an abundance of Diophantine points on $Y_\data$.

\begin{corollary}\label{cor:manydio}
There exists a constant $\consta\label{a:manydio}>\ref{a:linearization}$ depending only on $N$ with the following property.
Let $\eta \in (0,1/2)$.
Suppose that for all $r>0$
\begin{align}\label{eq:eps-manydio}
\epsilon(r) \leq r^{-\ref{a:linearization}} (\height(\G)^{-1}\eta\,{\gp}^{-1})^{\ref{a:manydio}}
\end{align}
For any $t>0$ at least one of the following is true:
\begin{enumerate}[label=\textnormal{(\alph*)}]
\item \emph{(Many Diophantine points)} We have
\begin{align*}
\mu\big(\big\{ y \in Y_\data:
\pi_{\mathcal{S}}(y) \text{ is not } &(\epsilon,t)\text{-Diophantine} \big\}\big) \leq (\height(\G)\gp)^{\ref{a:linearization}} \eta^{\frac{1}{\ref{a:linearization}}}.
\end{align*}
\item \emph{(Obstruction from an intermediate semisimple subgroup)}  There exists a non-trivial proper semisimple $\Q$-subgroup $\Lbf < \rho(\G)$ containing $\iota(\H)$ such that
\begin{align*}
\cpl([\Lbf(\A)g_\data]) =
\content(\eta_\Lbf(g_\data)) \leq \big(\height(\G)\eta^{-1}\gp\,\mathrm{e}^t \big)^{\ref{a:manydio}}.
\end{align*}
\end{enumerate}
\end{corollary}

For the proof, we will require the following lemma giving an effective isolation statement for Lie ideals of $\gfrak$.

\begin{lemma}\label{lem:Ufarfromideals}
For any proper normal $\Q$-subgroup $\Lbf \lhd \rho(\G)$ we have
\[
\|\zpz\wedge \vpz_{\Lbf}\| =1.
\]
\end{lemma}

\begin{proof}
Suppose that $\|\zpz\wedge \vpz_{\Lbf}\| <1$.
Since $\Lbf$ is normal, this implies that $\|\wpz\wedge \vpz_{\Lbf}\| <1$ for any $\wpz \in \hfrak'[0]$ where $\hfrak' = \Lie(\theta_\gp(\SL_2(\Q_{\gp})))$.
In particular, the reduction $\underline{\hfrak'}$ of $\hfrak'$ modulo $\gp$ is contained in a proper ideal of $\underline{\gfrak}$.
Since we assumed that $\iota(\Hbf)$ is not contained in a proper $\Q$-factor of $\rho(\G)$, this contradicts Lemma~\ref{lem:nofactorscontainingHp}.
\end{proof}

\begin{proof}[Proof of Corollary~\ref{cor:manydio}]
By Theorem~\ref{thm:diameter}, there exists $\gamma \in \rho(\G(\Q))$ so that $g = \gamma g_\data$ satisfies
\begin{align*}
g_\gep \in \SL_N(\Z_\gep) \text{ for all }\gep \neq \gp,\ \norm{g_\infty}\ll 1,\ \text{and } \norm{g_\gp} \leq \gp^{\ref{a:diameterestimate}}\height(\G)^{\ref{a:diameterestimate}}.
\end{align*}
Recall that the $U$-action on $Y_\data$ is ergodic.
There exists 
\begin{align*}
h\in g_\data^{-1}\iota(\Hbf(\A))g_\data \cap \{h' \in \SL_N(\A):\norm{h'}\leq \tfrac{3}{2}\}
\end{align*}
such that $x = [gh] \in Y_{\data}$ is Birkhoff generic and satisfies the pointwise ergodic theorem for the characteristic function of the set of $(\epsilon,t)$-Diophantine points.
We apply Theorem~\ref{thm:linearization} to $\pi_{\mathcal{S}}(x)$ for all $k \in \N$.
One of the options (1)-(3) holds infinitely often (i.e.~for infinitely many $k\in \N$). 
If (1) holds infinitely often, we conclude (a) (for any $\ref{a:manydio}\geq \ref{a:linearization}$).

Assume (2) holds infinitely often. Notice that for any $R>0$ there are finitely many $\Q$-subgroups $\Lbf'$ with $\content(\eta_{\Lbf'}(gh))\leq R$.
We may thus assume that (2) holds infinitely often for some proper subgroup $\Lbf' \in \mathcal{H}_\G$ with
\begin{align*}
\content(\eta_{\Lbf'}(gh)) = \content_\mathcal{S}(\eta_{\Lbf'}(\pi_\mathcal{S}(gh))) \leq (\height(\G)\eta^{-1}\gp\,\mathrm{e}^t)^{\ref{a:linearization}}.
\end{align*}
Here, we used $(gh)_\gep \in \SL_N(\Z_\gep)$ for $\gep \neq \gp$ in the first equality.
As the second inequality in (2) holds for infinitely many $k$, we have $\norm{\zpz \wedge \eta_{\Lbf'}(gh))}  = 0$.
In particular, $\zpz \in \Ad(gh)^{-1}\Lie(\Lbf')$ and hence $U \subset (gh)^{-1}\Lbf' gh = (g_\data h)^{-1}\Lbf (g_\data h)$ for $\Lbf = \gamma^{-1} \Lbf' \gamma$. 
By density of $xU$ in $Y_\data$ this shows that 
\begin{align*}
Y_{\data} = \overline{xU} = \overline{[g h] U} = \overline{[g_\data h] U}\subset [\Lbf(\A)g_\data h]
\end{align*}
and so $Y_{\data}\subset [\Lbf(\A)g_\data]$ since $Y_\data$ is invariant under $h$.
In particular, there exists for any $h' \in \H(\R)$ some $\gamma' \in \SL_N(\Q)$ and $m \in \Lbf(\A)$ with $\iota(h') = \gamma' m$.
Since $h'_\gep = \id$ for any prime $\gep$, this shows $\gamma' \in \Lbf(\Q)$ and so $\iota(h') \in \Lbf(\R)$. 
Thus, $\iota(\H(\R)) \subset \Lbf(\R)$ and so $\iota(\H) \subset \Lbf$ by Zariski-density.
By the above bound on $\content(\eta_{\Lbf'}(g)) = \content(\eta_\Lbf(g_\data))$, (b) follows.

Assume (3) in Theorem~\ref{thm:linearization} holds infinitely often (or once) for some non-trivial proper normal subgroup $\Lbf= \Lbf' \lhd \rho(\G)$.
Then by Lemma~\ref{lem:Ufarfromideals}
\begin{align*}
{\gp}^{-\star}\height(\G)^{-\star} \leq \norm{\zpz \wedge \vpz_{\Lbf}}
\leq \big(\height(\G)\eta^{-1}\gp\big)^{-\star \ref{a:manydio}+\star} \height(\Lbf)^{-\star}
\end{align*}
and so $\height(\Lbf) \ll \big(\height(\G)\eta^{-1}\gp\big)^{-\star \ref{a:manydio}+\star}$.
If $\ref{a:manydio}>0$ is chosen sufficiently large, this gives a contradiction as $\height(\G),\height(\Lbf) \geq 1$ 
and $\eta\in(0,1/2)$.
\end{proof}

For future convenience, we choose $\eta = \eta_0 >0$ with
\begin{align}\label{eq:def-etaG}
(\height(\G)\gp)^{\ref{a:linearization}} \eta_0^{\frac{1}{\ref{a:linearization}}} = \frac{1}{9}
\end{align}
and set
\begin{align}\label{eq:epsdef}
\epsilon_0(r) := r^{-\ref{a:linearization}} (\height(\G)^{-1}\eta_0{\gp}^{-1})^{\ref{a:manydio}}.
\end{align}
In particular, for all $r>0$
\begin{align*}
\epsilon_0(r) \gg r^{-\star} \height(\G)^{-\star} \gp^{-\star}.
\end{align*}
To simplify the terminology, we will say that $x \in X_\mathcal{S}$ is \textbf{$T$-Diophantine} if it is $(\epsilon_0,\log(T))$-Diophantine.

\begin{corollary}\label{cor:FubiniDio}
Suppose that
\begin{align}\label{eq:mcpllowerbound}
\mcpl(Y_\data) > \big(\height(\G)\eta_0^{-1}\gp\big)^{2\ref{a:manydio}}.
\end{align}
For any $n \geq 0$, the set of $y \in Y_\data$ with
\begin{align*}
\big|\big\{|s| \leq \gp^n: \pi_\mathcal{S}(y)u(s) \text{ is } \mcpl(Y_\data)^{\frac{1}{2\ref{a:manydio}}}\text{-Diophantine}\big\}\big| \leq \tfrac{2}{3}{\gp}^n
\end{align*}
has $\mu$-measure at most $\frac{1}{3}$.
\end{corollary}

\begin{proof}
By Corollary~\ref{cor:manydio} and in view of the choice of $\eta_0$ in \eqref{eq:def-etaG}, we either have as in (a)
\begin{align}\label{eq:measurediomincpl}
\mu\big(\big\{ y \in Y_\data:
\pi_{\mathcal{S}}(y) \text{ is not } \mcpl(Y_\data)^{\frac{1}{2\ref{a:manydio}}}\text{-Diophantine} \big\}\big) \leq \tfrac{1}{9}
\end{align}
or as in (b)
\begin{align*}
\mcpl(Y_\data) \leq \big(\height(\G)\eta_0^{-1} \gp\, \mcpl(Y_\data)^{\frac{1}{2\ref{a:manydio}}}\big)^{\ref{a:manydio}}.
\end{align*}
The latter contradicts our assumption on the minimal complexity in \eqref{eq:mcpllowerbound} and so \eqref{eq:measurediomincpl} holds.

Given $n \geq 0$ and $y \in Y_\data$ set
\begin{align*}
f(y) = \gp^{-n}\big|\big\{|s| \leq \gp^n: \pi_\mathcal{S}(y)u(s) \text{ is not } \mcpl(Y_\data)^{\frac{1}{2\ref{a:manydio}}}\text{-Diophantine}\big\}\big|
\end{align*}
By $U$-invariance of $\mu$ and \eqref{eq:measurediomincpl} we have $\int f \de \mu \leq \frac{1}{9}$.
The corollary thus follows from the Chebyshev inequality.
\end{proof}

\section{An effective closing lemma}\label{s;closing-lem}\label{sec:closing-lem}

The main result of this section is the effective closing lemma in Proposition~\ref{prop:closing-lemma}.
Throughout the section, we fix a Lie subalgebra $\mfrak\subset\gfrak_\gp$ defined over $\Q_\gp$ and our standing assumption is that
\be\label{eq:sfrak is H-invariant}
\mfrak\text{ is $\Ad(U)$-invariant}.
\ee
Denote by $\hat{\vpz}_{\mfrak} \in \mathbb{P}(\wedge^{\dim(\mfrak)}\gfrak_\gp)$ the point corresponding to $\mfrak$.
The projective space $\mathbb{P}(\wedge^{\dim(\mfrak)}\gfrak_\gp)$ is equipped with the metric given by $\metric(\hat{\vpz},\hat{\wpz}) = \min_{\alpha \in \Z_{\gp}^\times} \norm{\vpz-\alpha\wpz}$ for any choice of unit vectors in $\vpz\in \hat{\vpz},\wpz \in \hat{\wpz}$.
Note that $G_p$ acts on $\mathbb{P}(\wedge^{\dim(\mfrak)}\gfrak_\gp)$ through the adjoint representation.
We say that $\mfrak$ is \textbf{$\varepsilon$-normalized} by an element $g \in G_\gp$ if
\[
\metric\big(g.\hat{\vpz}_\mfrak,\hat{\vpz}_\mfrak\big) \leq \varepsilon.
\]
We will study this notion further in \S\ref{sec:alignment} below.
Here, we prove the following.

\begin{proposition}\label{prop:closing-lemma}
There exists $\constk\label{k:effclosing}\in (0,1)$ so that the following holds.
Let
\begin{align*}
n \geq \tfrac{1}{\ref{k:effclosing}}(\log_\gp(\height(\G)) +1).
\end{align*}
Assume there is a point $y\in X$ and a measurable subset
\[
{\mathsf E}\subset \{r\in \Q_{\gp}: |r|_{\gp}\leq {\gp}^n\}
\] 
with the following properties:
\begin{enumerate}[label=\textnormal{(\theenumi)}]
\item $\pi_\plw(yu(s))$ is ${\gp}^{n}$-Diophantine for all $s\in{\mathsf E}$, 
\item $|\mathsf{E}|>{\gp}^{n(1-\ref{k:effclosing})}$, and 
\item for all $s,s'\in{\mathsf E}$, we have
\begin{align*}
yu(s)=yu(s')\rho(g_{ss'})
\end{align*}
where $g_{ss'}\in \G(\R)\times K_f$ satisfies $\norm{g_{ss'}}_\infty \leq 2$ and is so that the Lie algebra $\mfrak$ is $\gp^{-n}$-normalized by $(g_{ss'})_{\gp}$.
\end{enumerate}
Then $\mfrak$ is a semisimple ideal of $\gfrak_{\gp}$.
\end{proposition}

For quotients of real groups, an analogous result was obtained in \cite{ClosingLemma} and, in fact, significantly strengthened using \cite{LMMS}.
For our purposes, the above proposition is sufficient.
The proof we give below follows the argument given in \cite{ClosingLemma} by adapting it to the current $\gp$-adic setting; we include it for completeness.

\subsection{Almost invariant subalgebras}

The following proposition, which is of independent interest, will play an important role in the proof of 
Proposition~\ref{prop:closing-lemma}.
For a lattice element $\gamma \in \SL_N(\Q)$ the height $\height(\gamma)$ is the Euclidean norm of the smallest non-zero multiple of $\gamma$ with integer coefficients.

\begin{proposition}\label{prop:exp-return-semisimple}
For any $r >0$ there exist $A>0$ and $C>0$ (depending on $r$ and $N$) with the following property.

Let $T>2$, $\delta>0$, and suppose that we have the following.
\begin{enumerate}[label=\textnormal{(\theenumi)}]
\item $\{\gamma_1,\ldots,\gamma_r\}\subset\rho(\G)(\Q)$ with $\height(\gamma_i)\leq T$.
\item The Zariski closure of $\langle\gamma_1,\ldots,\gamma_r\rangle$ equals $\rho(\G)$. 
\item For every $1\leq i\leq r$ the Lie algebra $\mfrak$ is $\delta$-normalized by $\gamma_i$ i.e.~we have
\begin{align}\label{eq:vR-almost-inv}
\metric(\gamma_i.\hat{\vpz}_{\mfrak},\hat{\vpz}_{\mfrak}) 
\leq \delta.
\end{align}
\end{enumerate}
Then either 
\begin{align*}
C\delta \geq (\height(\G)\gp\,T)^{-A}
\end{align*}
or $\mfrak$ is an ideal of $\gfrak_{\gp}$.
\end{proposition}

The proof invokes the effective version of Greenberg's theorem (Theorem~\ref{thm:greenberg}) and the following isolation property for ideals of $\gfrak_\gp$.

\begin{lemma}\label{lem:isolationideals}
There exists $\consta\label{a:isolation of ideals}>1$ depending only on $N$ so that the following holds.   
Suppose that $\mfrak$ is a Lie subalgebra of $\gfrak_{\gp}$ and that there exists a Lie ideal $\mfrak'$ of $\gfrak_{\gp}$ with 
\begin{align*}
\metric(\hat{\vpz}_{\mfrak'},\hat{\vpz}_{\mfrak}) \leq \gp^{-\ref{a:isolation of ideals}}\height(\G)^{-\ref{a:isolation of ideals}}.
\end{align*}
Then $\mfrak= \mfrak'$ is an ideal.
\end{lemma}

The proof of Lemma~\ref{lem:isolationideals} is, mutatis mutandis, the proof of \cite[Lemma 3.3]{ClosingLemma}; we omit it here.

\begin{proof}[Proof of Proposition~\ref{prop:exp-return-semisimple}]
Take a basis $\wpz_1,\ldots,\wpz_s$ of the $\Z_{\gp}$-submodule $\mfrak[0]$.
For any $1\leq i \leq r$ and $1 \leq j \leq s$ we have
\begin{align}\label{eq:altalmostnormalizer}
\norm{\Ad(\gamma_i)\wpz_j \wedge \wpz_1 \wedge \ldots \wedge \ldots \wpz_s}
&= \norm{\Ad(\gamma_i)\wpz_j} \Big\|\frac{\Ad(\gamma_i)\wpz_j}{\norm{\Ad(\gamma_i)\wpz_j}} \wedge \wpz_1 \wedge \ldots \wedge \wpz_s\Big\|\nonumber\\
&\ll T^\star \delta
\end{align}
using $\height(\gamma_i) \leq T$ for the first term and \eqref{eq:vR-almost-inv} for the second.
We want to perturb the vectors $\wpz_1,\ldots,\wpz_s$ to obtain new vectors nearby that span a subalgebra normalized by all $\gamma_i$ (and hence by $\G$).
For this, we define
\begin{align*}
\Zbf_1 &= \big\{(\wpz_1',\ldots,\wpz_s') \in \gfrak^s: [\wpz_i',\wpz_j']\wedge \wpz_1'\wedge \ldots \wedge \wpz_s'=0\text{ for all }i,j\big\},\\
\Zbf_2 &= \big\{(\wpz_1',\ldots,\wpz_s') \in \gfrak^s: \Ad(\gamma_i)\wpz_j' \wedge \wpz_1' \wedge \ldots \wedge \wpz_s' = 0 \text{ for all }i,j\big\}.
\end{align*}
Both $\Zbf_1$ and $\Zbf_2$ are defined by integral equations with height $\ll \height(\G)^\star T^\star$.
As $\mfrak$ is a Lie algebra, $(\wpz_1,\ldots,\wpz_s)\in \Zbf_1(\Q_\gp)$. 
This, \eqref{eq:altalmostnormalizer}, and Theorem~\ref{thm:greenberg} imply that there exists $(\wpz_1',\ldots,\wpz_s') \in (\Zbf_1\cap \Zbf_2)(\Q_\gp)$ with $\norm{\wpz_i'-\wpz_i} \ll \delta^\star$ for all $i$ unless $\delta \gg \height(\G)^{-\star}T^{-\star} \gp^{-\star}$.
The subspace $\mfrak'$ spanned by $\wpz_1',\ldots,\wpz_s'$ is a Lie algebra and satisfies $\metric(\hat{\vpz}_{\mfrak'},\hat{\vpz}_{\mfrak})\ll \delta^\star$.
Also, as $(\wpz_1',\ldots,\wpz_s') \in \Zbf_2$, the Lie algebra $\mfrak'$ is normalized by all $\gamma_i$ and, thus, is a Lie ideal by Zariski density of the group $\langle \gamma_1,\ldots, \gamma_k\rangle$ in $\rho(\G)$.
Lemma~\ref{lem:isolationideals} implies that $\mfrak$ needs to be an ideal as well unless, again, $\delta \gg \height(\G)^{-\star}T^{-\star} \gp^{-\star}$.
This proves the proposition.
\end{proof}

\subsection{Proof of Proposition~\ref{prop:closing-lemma}}\label{sec: proof of closing}

We will complete the proof of Proposition~\ref{prop:closing-lemma} in various steps.
The proof uses an inductive process to construct a nontrivial group $\Lbf<\rho(\G)$ of class $\Hcal$ and of controlled height 
so that a piece of the $U$-orbit through $y$ stays very close to a translate of the orbit $[\Lbf(\A)]$.
Then using our assumption (1) in Proposition~\ref{prop:closing-lemma}, 
we conclude that $\Lbf=\rho(\Gbf)$. After this is established, the proposition follows from Proposition~\ref{prop:exp-return-semisimple}.

\subsection*{Notation and setup for the proof}
Let $y\in X$ and $\mathsf{E}$ be as in the statement of Proposition~\ref{prop:closing-lemma} satisfying $|\mathsf{E}| > {\gp}^{n(1-\kappa)}$ for some $\kappa\in (0,1/2)$ which will be determined later.
In particular, we have for $s,s'\in \mathsf{E}$
\begin{align}\label{eq:excep-return-1} 
yu(s)=yu(s')\rho(g_{ss'})
\end{align}
and
\begin{align}\label{eq:excep-return-2} 
\metric\big((g_{ss'})_{\gp}.\hat{\vpz}_\mfrak,\hat{\vpz}_\mfrak\big)
\leq {\gp}^{-n}.
\end{align}
Furthermore, $(g_{ss'})_\gep \in K_\gep$ for any $\gep$ and $\norm{g_{ss'}}_\infty \leq 2$.

We set $\alpha = \kappa^{1/(2(1+\dim(\G)))}$ and may determine $\alpha$ instead of $\kappa$ in the proof below (depending only on $N$). 
For every $0\leq j\leq 1+\dG$, let 
\be\label{eq:Sm-Tm}
\alpha_j=\alpha^{1+\dG-j}\qquad\text{and}\qquad R_j:={\gp}^{\lceil n\alpha_j\rceil}.
\ee
We have $R_{1+\dG} = p^n$ and $R_{j+1}^\alpha \leq R_{j} \leq R_{j+1}^\alpha \gp$ for every $0 \leq j \leq \dG$.
In view of our assumptions on $n$, we assume that
\begin{align}\label{eq:effcls-Rmlarge}
R_0^{\kappa} \geq \gp\,\height(\G).
\end{align}

\subsection*{A pigeonhole principle argument}
We need {\em good} density points for ${\mathsf E}$ simultaneously in all scales $R_j$ for $0\leq j\leq\dim(\G)+1.$ 
This is obtained using a simple pigeonholing argument. 

We proceed inductively, let $j=\dim(\G)$.
Take a subdivision of 
\[
\mathsf B(0, R_{j+1})=\{r\in \Q_{\gp}: |r|_{\gp}\leq R_{j+1}\}
\] 
into $R_{j+1}/R_j$-disjoint balls $\mathsf B(s,R_j)=\{r: |s-r|_{\gp}\leq R_j\}$.
By the pigeonhole principle, 
there is some $t_{j}\in\mathsf B(0, R_{j+1})$
such that 
\[
\big|\mathsf B(t_{j},R_j)\cap {\mathsf E}\big|
\geq \frac{{\gp}^{n(1-\kappa)}R_j}{R_{j+1}}=R_j{\gp}^{-n\kappa}\geq {\gp}^{n(\alpha-\kappa)}.
\]

Replace $\mathsf B(0, R_{\dG+1})$ with 
$\mathsf B({t_{\dG}}, R_\dG)$ and repeat the above argument.  
Iterating this process $(\dG+1)$-times we get 
\[
 \{t_0,\ldots,t_{\dG}\}\subset \mathsf B(0, R_{\dG+1})
\]
such that the following holds for all $0\leq j\leq \dG$: 
\begin{enumerate}
\item[(a)] $t_{j'}\in \mathsf B(t_j, R_{j}),$ for all $0\leq j'\leq j$, and 
\item[(b)] $\big|\mathsf B(t_{j}, R_j)\cap {\mathsf E}\big|\geq {\gp}^{n(\alpha_j-\kappa)}$. (Here, recall that $\alpha_j>\kappa$.)  
\end{enumerate} 
Put $t_{\dG+1}:=0$.  

\subsection*{Finding $\plw$-arithmetic elements with controlled height}

Write $y=[g]$ for $g \in \rho(\G(\A))$.
Fix some $s_{0}\in \mathsf B(t_0,R_{0})\cap{\mathsf E}$ and let $\gamma_{0}\in \rho(\G(\Q))$ be so that for 
$g_0=\gamma_0 gu(s_0)\in \rho(\G(\A))$ we have
\be\label{eq:g0-in-K0}
\begin{aligned}
    &g_{0,\gep}\in \SL_N(\Z_\gep)\qquad\text{for $\gep\neq \gp$},\\ 
    &\|g_{0,\gp}\|_\gp\ll {\gp}^{\ref{a:diameterestimate}} \height(\G)^{\ref{a:diameterestimate}}, \text{ and}\\
    &\norm{g_{0,\infty}}_\infty \ll 1.
\end{aligned}
\ee
Note that such a lattice element $\gamma_0$ exists thanks to Theorem~\ref{thm:diameter}.

By \eqref{eq:excep-return-1} and~\eqref{eq:excep-return-2} for $s'=s_0$ there exists for all $s\in{\mathsf E}$ a lattice element $\gamma_s\in\rho(\G(\A))\cap \SL_N(\Q) \subset \rho(\G)(\Q)$ so that
\begin{align}\label{eq:t0-exp-return}
\gamma_sg_0u(-s_0)u(s)=g_0\mathsf g_s \qquad\text{and}\qquad 
\metric(\mathsf{g}_{s,\gp}.\hat{\vpz}_\mfrak,\hat{\vpz}_\mfrak)\leq \gp^{-n}
\end{align}
where $\mathsf g_s=\rho(g_{ss_0})$.
This and~\eqref{eq:g0-in-K0} in particular imply that $\gamma_{s,\gep}\in \SL_N(\Z_\gep)$ for all $\gep\neq \gp$.
Thus, $\gamma_s \in \Gamma_{\mathcal{S}}$.
By~\eqref{eq:t0-exp-return}, \eqref{eq:g0-in-K0} and \eqref{eq:effcls-Rmlarge} we further have for all $s\in \mathsf E\cap\mathsf B(t_{j}, R_{j})$
\begin{equation}\label{eq:effcls-sizelatticeel}
\begin{split}
\|\gamma_s\|_\gp &\ll {\gp}^\star \height(\G)^\star R_{j}^\star \ll R_j^\star,\\
\|\gamma_s\|_\infty &\ll 1
\end{split}
\end{equation}
and, in particular,
\begin{align*}
\norm{\gamma_s} &= \max\{\|\gamma_s\|_\gp,\|\gamma_s\|_\infty\} \ll R_j^\star,\\
\height(\gamma_s) &\ll R_j^\star.
\end{align*}
Lastly, notice that if $\gamma_s = \gamma_{s'}$ for $s,s' \in \mathsf{E}$ then by \eqref{eq:t0-exp-return}
\begin{align*}
u(s) = u(s_0)g_0^{-1}\gamma_s^{-1}g_0 \mathsf{g}_s
= u(s_0)g_0^{-1}\gamma_{s'}^{-1}g_0 \mathsf{g}_s
= u(s')\mathsf{g}_{s'}^{-1}\mathsf{g}_s
\end{align*}
and so $u(s-s') \in \SL_N(\Z_\gp)$ using $(\mathsf{g}_s)_\gp,(\mathsf{g}_{s'})_\gp \in \SL_N(\Z_\gp)$.
By choice of the unipotent subgroup $U$ in \eqref{eq:ut in SL2}, we have $s-s' \in \Z_\gp$.
Since $\big|\mathsf B(t_{j}, R_j)\cap {\mathsf E}\big|\geq {\gp}^{n(\alpha_j-\kappa)}$, this proves
\begin{align}\label{eq:manylattel}
\#\{\gamma_s: s \in B(t_{j}, R_j)\cap {\mathsf E}\}
\geq {\gp}^{n(\alpha_j-\kappa)}.
\end{align}

\subsection*{Finding a controlled semisimple $\Q$-group}
For $j \in \{0,\ldots,\dim(\G)+1\}$, define
\begin{align*}
\Lbf_{j}' &= \overline{\Bigl\langle\{\gamma_{s}: s\in\mathsf B(t_j, R_j)\cap \mathsf{E}\}\Bigr\rangle}^z \text{ and}\\
\Lbf_j &= \text{the identity component of }\Lbf_j'.
\end{align*}
These are $\Q$-subgroups of $\rho(\G)$.
Dimension considerations imply that 
there exists some $0\leq j_0\leq \dG,$ so that $\Lbf_{j_0}=\Lbf_{j_0+1}$.
Note that $\Lbf_{j_0}$ is not necessarily in $\hcal$ (i.e.~might not be semisimple). 
Let $\Lbf:=\Lbf_{j_0}^\hcal$ be the maximal subgroup of class $\hcal$ which is contained in $\Lbf_{j_0}$.
Since $\Lbf_{j_0}$ is reductive, $\Lbf = [\Lbf_{j_0},\Lbf_{j_0}]$.
By \eqref{eq:effcls-sizelatticeel} and \cite[Prop.~2.6]{ClosingLemma} we have
\begin{align}\label{eq:cl-ht-M-1}
\height(\Lbf)\ll R_{j_0}^{\star}.
\end{align}

We now investigate properties of $\Lbf$; the goal is to show that $\Lbf=\rho(\Gbf)$. 
This will be established in two steps in Lemmas~\ref{lem: M neq 1} and~\ref{lem: M is G}. 

\begin{lemma}\label{lem: M neq 1}
The group $\Lbf$ is non-trivial. 
\end{lemma}

To clarify, it is also true that $\Lbf_{j}^{\mathcal{H}} \neq \{1\}$ for any $j$ i.e.~the special choice of $j_0$ plays no role in this lemma.

\begin{proof}[Proof of Lemma~\ref{lem: M neq 1}]
We need to show that $\Lbf_{j_0}$ is not a torus.
The idea is to exhibit a `polynomial amount' of lattice elements in $\Lbf_{j_0}$ and combine this with logarithmic growth of the number of $\mathcal{S}$-arithmetic lattice elements in tori.

By \cite[Lemma 2.4]{ClosingLemma} we have $[\Lbf_{j_0}':\Lbf_{j_0}] \ll_N 1$.
By \eqref{eq:manylattel}, there exists a coset $\gamma_s \Lbf_{j_0}$, $s \in B(t_{j_0}, R_{j_0})\cap {\mathsf E}$, which contains $\gg {\gp}^{n(\alpha_{j_0}-\kappa)}$ many lattice elements $\gamma_{s'}$ for $s' \in B(t_{j_0}, R_{j_0})\cap {\mathsf E}$.

In view of \eqref{eq:effcls-sizelatticeel} we have $\norm{\gamma_{s'}}\leq C R_{j_0}^B$ for all $s' \in B(t_{j_0}, R_{j_0})\cap {\mathsf E}$ and some $B,C>0$ absolute.
Now note that
\begin{align*}
\#\big\{\gamma\in\Lbf_{j_0}(\Z[1/\gp]):\ & \norm{\gamma}\leq C^2 R_{j_0}^{2B}\big\}\\
&\geq
\#\big\{\gamma_s^{-1}\gamma_{s'}:s' \in B(t_{j_0}, R_{j_0})\cap {\mathsf E},\ \gamma_{s'} \in \gamma_s \Lbf_{j_0}\big\}\\
&\gg {\gp}^{n(\alpha_{j_0}-\kappa)} \geq R_{j_0}^{\frac{1}{2}}.
\end{align*}
For any $\Q$-torus $\Tbf < \SL_N$ and any $R\gg 1$ we have (see e.g.~\cite[Lemma 6.3]{EL23-nonmaximal})
\begin{align}\label{eq:pts on tori}
\#\{\gamma \in \Tbf(\Z[1/\gp]):\norm{\gamma} \leq R\} \ll_N \log(R)^\star.
\end{align}
This shows that $\Lbf_{j_0}$ is not a $\Q$-torus assuming \eqref{eq:effcls-Rmlarge} and hence $\Lbf$ is non-trivial.
\end{proof}

\begin{lemma}\label{lem: M is G}
We have $\Lbf=\rho(\Gbf)$.
\end{lemma}

\begin{proof}
Assume contrary to our claim that $\Lbf\neq\rho(\Gbf)$.

\textsc{Case 1:}
Suppose that $\Lbf$ is not normal.

Notice that for any $s \in B(t_{j_0+1},R_{j_0+1})$ the lattice element $\gamma_s$ satisfies $\gamma_s^{-1}.\vpz_{\Lbf} = \chi(\gamma_s)\vpz_{\Lbf}$ for some $\mathcal{S}$-adic unit $\chi(\gamma_s) \in \Z[1/\gp]^\times$.
Here, we used that $\Lbf$ is a normal subgroup of $\Lbf_{j_0+1}'$ and that $\gamma_s \in \SL_N(\Z_\gep)$ for $\gep \neq \gp$.
By \eqref{eq:effcls-sizelatticeel} we have $|\chi(\gamma_s)|_\infty,|\chi(\gamma_s)|_{\gp} \ll R_{j_0+1}^\star$.
Notice that the number of units $u \in \Z[1/\gp]^\times$ satisfying the same bound is $\ll \log(R_{j_0+1}) \ll \log(R_{j_0})$.
 
For every $s\in \mathsf E\cap \mathsf B(t_{j_0+1},R_{j_0+1})$, put
\begin{align}\label{eq:def-Cv}
\mathsf {J}(s)
:=\big\{t\in \mathsf E\cap \mathsf B(t_{j_0+1},R_{j_0+1}): 
\gamma_t^{-1}. \vpz_{\Lbf} = \gamma_s^{-1}.\vpz_{\Lbf}\big\}.
\end{align}
The above estimate on the number of multiplicative factors $\chi(\gamma_s)$ implies that there exists some 
$s\in \mathsf E\cap \mathsf B(t_{j_0+1},R_{j_0+1})$ so that
\be\label{eq:Cv-big}
\begin{aligned}
\#\mathsf J(s)
&\gg{{\gp}^{n(\alpha_{j_0+1}-\kappa)}}{\bigl(\log(R_{j_0})\bigr)}^{-1}
 \geq R_{j_0+1}R_{j_0}^{-2},
\end{aligned}
\ee
where in the last step we used $\kappa<\alpha_{j_0}$. 

Let $s$ be so that~\eqref{eq:Cv-big} is satisfied. 
Using~\eqref{eq:g0-in-K0}, \eqref{eq:cl-ht-M-1} and the given bounds on $\mathsf g_t$ we have 
\begin{align*}
\norm{\eta_{\Lbf}(g_0\mathsf g_t)} \ll {\gp}^\star \height(\G)^\star \height(\Lbf) \ll R_{j_0}^\star
\end{align*}
for any $t \in \mathsf{E}$.
On the other hand, if $t \in \mathsf{J}(s)$ we have
\begin{align*}
\eta_\Lbf(g_0\mathsf g_t)
&=\eta_\Lbf\big(\gamma_tg_0u(-s_0)u(t)\big)&&\text{by~\eqref{eq:t0-exp-return}}\\
&=\eta_{\Lbf}\big(\gamma_sg_0u(-s_0)u(s)u(-s)u(t)\big)&&\text{by~\eqref{eq:def-Cv}}\\
&=\eta_{\Lbf}\big(g_0\mathsf g_su(t-s)\big)&&\text{by~\eqref{eq:t0-exp-return}}.
\end{align*}
Altogether, we conclude that  
\begin{align}\label{eq:effcls-boundalongorbit}
\|\eta_\Lbf\bigl(g_0\mathsf{g}_s u(t-s)\bigr)
\|\ll R_{j_0}^\star\qquad\text{for all $t\in \mathsf J(s)$}.
\end{align}

The above map
\begin{align*}
t \in \{|t|\leq R_{j_0+1}\} \mapsto \eta_\Lbf\bigl(g_{0,\gp}\mathsf{g}_{s,\gp} u(t)\bigr)
\end{align*}
is a polynomial map. 
The Remez inequality (see e.g.~\cite[Lemma 5.4]{LMMS}) together with \eqref{eq:effcls-boundalongorbit} and \eqref{eq:Cv-big} implies
\begin{align*}
\norm{\eta_\Lbf\bigl(g_{0,\gp}\mathsf{g}_{s,\gp} u(t)\bigr)} \ll R_{j_0}^\star
\end{align*}
for all $|t|_\gp\leq R_{j_0+1}$.
By \cite[Prop.~5.8]{LMMS} we obtain that
\begin{align}\label{eq:effcls-part1contradiction}
\norm{\zpz \wedge \eta_{\Nbf_{\Lbf}^\mathcal{H}}(g_{0,\gp}\mathsf{g}_{s,\gp})}_\gp
\ll R_{j_0}^\star R_{j_0+1}^{-\star}.
\end{align}
where $\Nbf_{\Lbf}$ is the normalizer of $\Lbf$ in $\rho(\G)$ and as before $\Nbf_{\Lbf}^{\mathcal{H}} = [\Nbf_{\Lbf},\Nbf_{\Lbf}]$.

On the other hand, we have $\height(\Nbf_\Lbf^{\mathcal{H}}) \ll \height(\Lbf)^\star$ by \cite[Lemma~4.2]{LMMS} and so $\content(\eta_{\Nbf_\Lbf^{\mathcal{H}}}(g_0\mathsf{g}_s))\ll R_{j_0}^\star$ using additionally \eqref{eq:cl-ht-M-1} and \eqref{eq:g0-in-K0}.
For $\alpha$ sufficiently small, we have $\content(\eta_{\Nbf_\Lbf^{\mathcal{H}}}(g_0\mathsf{g}_s))\leq R_{j_0+1}$.
As $\pi_{\mathcal{S}}([g_0\mathsf{g}_s])=\pi_{\mathcal{S}}(yu(s))$ is $\gp^n$-Diophantine (and in particular $R_{j_0+1}$-Diophantine) by assumption, we deduce from the definitions (see Definition~\ref{def:Diophantine-intro} and \eqref{eq:epsdef})
\begin{align}\label{eq:effcls-part2contradiction}
\norm{\zpz \wedge \eta_{\Nbf_\Lbf^{\mathcal{H}}}(g_{0,\gp}\mathsf{g}_{s,\gp})} 
\geq \epsilon_0(\content(\eta_{\Nbf_\Lbf^{\mathcal{H}}}(g_0\mathsf{g}_s)))
\gg R_{j_0}^{-\star} \height(\G)^{-\star} \gp^{-\star} \gg R_{j_0}^{-\star}.
\end{align}
The two inequalities \eqref{eq:effcls-part1contradiction} and \eqref{eq:effcls-part2contradiction} together yield $R_{j_0+1}\ll R_{j_0}^\star$. However, for $\alpha$ sufficiently small in comparison to the exponent, we obtain a contradiction.

\textsc{Case 2: $\Lbf$ is normal}

The proof is largely analogous in this case, but necessarily needs to use a different representation.
Let $(\varrho,\vpz)$ be a Chevalley pair for the subgroup $\Lbf$ of $\SL_N$ as in \cite[Prop.~2.2]{ClosingLemma}.
By restriction to the subspace of $\Lbf$-invariant vectors we obtain a representation $\hat{\varrho}: \rho(\G)\to \SL_m$ for some $m\ll_N 1$ whose kernel is exactly $\Lbf$.
In particular, the identity component of the image $\hat{\varrho}(\Lbf_{j_0+1}')$ is a $\Q$-torus.
Combined with \eqref{eq:pts on tori} and the estimate on the number of connected components in \cite[Lemma~2.4]{ClosingLemma} we get
\begin{align*}
\#\{\Lbf \gamma_{s}: s \in \mathsf E\cap \mathsf B(t_{j_0+1},R_{j_0+1})\} \ll \log(R_{j_0})^\star.
\end{align*}
For $s \in \mathsf E\cap \mathsf B(t_{j_0+1},R_{j_0+1})$ let
\begin{align*}
\mathsf {J}(s) =\{t\in \mathsf E\cap \mathsf B(t_{j_0+1},R_{j_0+1}): 
\Lbf \gamma_s = \Lbf \gamma_t\}.
\end{align*}
For some $s$ we have as in the first case $\#\mathsf {J}(s) \geq R_{j_0+1}R_{j_0}^{-2}$.
For $t \in \mathsf {J}(s)$
\begin{align*}
\varrho(g_0\mathsf g_t)^{-1}\vpz
&=\varrho\big(\gamma_tg_0u(-s_0)u(t)\big)^{-1}\vpz\\
&=\varrho\big(\gamma_sg_0u(-s_0)u(s)u(-s)u(t)\big)^{-1}\vpz
=\varrho\big(g_0\mathsf g_su(t-s)\big)^{-1}\vpz.
\end{align*}
By the same argument as in the first case using the Remez inequality and the height bound on $\vpz$ from \cite[Prop.~2.2]{ClosingLemma} we obtain
\begin{align*}
\norm{\varrho(g_0\mathsf{g}_su(t))^{-1}.\vpz}_\gp \ll R_{j_0}^\star
\end{align*}
for all $|t|_\gp \leq R_{j_0+1}$.
Using \eqref{eq:g0-in-K0} and taking the derivative we deduce
\begin{align*}
\norm{\mathrm{D}\varrho\big(\Ad(g_{0,\gp}\mathsf{g}_{s,\gp})\zpz\big)\vpz} \ll R_{j_0}^\star R_{j_0+1}^{-\star}.
\end{align*}
The map $\wpz \mapsto \mathrm{D}\varrho(\wpz)\vpz$ is a linear map with coefficients controlled by the entries of $\vpz$ and `almost annihilates' $\Ad(g_{0,\gp}\mathsf{g}_{s,\gp})\zpz$.
Note also that $\mathrm{D}\varrho(\wpz)\vpz = 0$ implies $\wpz\in \Lie(\Lbf) = \lfrak$ (by definition of the representation).
A simple argument involving the Smith normal form (as e.g.~in the proof of Proposition~\ref{prop:localseparatedness}) then shows that there exists $\wpz \in \lfrak(\Q_p)$ such that $\norm{\wpz-\Ad(g_{0,\gp}\mathsf{g}_{s,\gp})\zpz}\ll R_{j_0}^\star R_{j_0+1}^{-\star}$.
Since $\Lbf$ is normal, this implies
\begin{align*}
\norm{\zpz \wedge \vpz_{\Lbf}} \ll  R_{j_0}^\star R_{j_0+1}^{-\star}.
\end{align*}
From here, one concludes in the same manner as in the first case (or invokes Lemma~\ref{lem:isolationideals}).
\end{proof}

\begin{proof}[Proof of Proposition~\ref{prop:exp-return-semisimple}]
We use all of the notation from the current section.
By Lemma~\ref{lem: M is G} and the construction of $\Lbf$, 
the group $\langle \gamma_s: s \in B(t_{j_0},R_{j_0}) \rangle$ is Zariski-dense in $\rho(\G)$. 
In view of \cite[Lemma 2.4]{ClosingLemma}, we may select $s_1,\ldots,s_r \in B(t_{j_0},R_{j_0})$ for some $r \ll \dim(\G)$ so that $\langle \gamma_{s_1},\ldots,\gamma_{s_r}\rangle$ is also Zariski dense.
Moreover, by $U$-invariance of $\mfrak$ and \eqref{eq:t0-exp-return} we have
\begin{align*}
\metric(\gamma_s g_{0,\gp}.\hat{\vpz}_{\mfrak},g_{0,\gp}.\hat{\vpz}_{\mfrak})
&= \metric(\gamma_s g_{0,\gp}u(s-s_0).\hat{\vpz}_{\mfrak},g_{0,\gp}.\hat{\vpz}_{\mfrak})
= \metric(g_{0,\gp}\mathsf{g}_{s,\gp}.\hat{\vpz}_{\mfrak},g_{0,\gp}.\hat{\vpz}_{\mfrak})\\
&\ll \gp^\star \height(\G)^\star \metric(\mathsf{g}_{s,\gp}.\hat{\vpz}_{\mfrak},\hat{\vpz}_{\mfrak})
\leq \gp^\star \height(\G)^\star \gp^{-n}
\ll R_{j_0+1}^{-\frac{1}{2}}.
\end{align*}
Thus, \eqref{eq:vR-almost-inv} holds for the Lie algebra $\Ad(g_{0,\gp})\mfrak$ and some $\delta>0$ with $\delta \ll R_{j_0+1}^{-\frac{1}{2}}$.
The lattice elements $\gamma_{s_i}$ satisfy $\height(\gamma_{s_i})\leq T$ for some $T \ll R_{j_0}^\star$ by \eqref{eq:effcls-sizelatticeel}.
Applying Proposition~\ref{prop:exp-return-semisimple} we conclude for $\alpha$ sufficiently small.
\end{proof}

\section{An alignment lemma}\label{sec:alignment}

In this section we show that whenever $x_1,x_2=x_1g \in X$ are two points where the `small' displacement $g$ does not `almost normalize' an `effectively generated' group $M$, then $x_1,x_2$ can be moved by $m_1,m_2 \in M$ respectively so that the new displacement is in an undistorted complement of $\mfrak$.
Here, the quotation marks `$\ldots$' are to be understood in the sense of \S\ref{sec:outlineproof} and will be made precise (which will include different scales).

Recall that an element $g \in G_\gp[0]$ $\varepsilon$-normalizes a subalgebra $\mfrak \subset\gfrak_{\gp}$ if
\[
\metric\big(g.\hat{\vpz}_\mfrak,\hat{\vpz}_\mfrak\big) \leq \varepsilon.
\]
Here, $\hat{\vpz}_\mfrak \in \mathbb{P}(\wedge^{\dim(\mfrak)}\gfrak_\gp)$ denotes the point corresponding to $\mfrak \subset \gfrak_\gp$ and $\metric(\cdot,\cdot)$ the metric as in \S\ref{sec:closing-lem}.

We note that unlike the normalizer of $\mfrak$, which can change drastically under `small perturbations', 
the notion of $\varepsilon$-normalizer is `stable' under `small perturbations' of the Lie algebra. 
It is worthwhile noting that points in the $\varepsilon$-normalizer need not be $\varepsilon^\star$-close to the normalizer.
For example, for any nilpotent $\vpz \in \mathfrak{sl}_2(\Z_\gp)$ and any $k \geq 0$ the normalizer of the nilpotent Lie algebra $ \Q_\gp (\vpz,\gp^{k}\vpz) \subset \mathfrak{sl}_2(\Q_\gp)\times \mathfrak{sl}_2(\Q_\gp)$ is a $2$-dimensional solvable subgroup of $\SL_2(\Q_\gp)\times\SL_2(\Q_\gp)$.
On the other hand, the $\gp^{-k}$-normalizer contains e.g.~$\{0\}\times \SL_2(\Z_\gp)$.

For any $\Q_{\gp}$-subgroup $\Mbf$ with $H_\gp \subset \Mbf(\Q_\gp)=M$ we denote its Lie algebra by~$\mfrak$ and by $\rfrak_{\mfrak} \subset \gfrak$ a choice of undistorted $\theta_\gp(\SL_2(\Q_{\gp}))$-invariant complement to $\mfrak$ (see Lemma~\ref{lem: undistorted complement}).
Recall that given a vector $\wpz$ in a subrepresentation $V\subset \gfrak_\gp$ we write $\wpz^{\mathrm{nt}}$ for the component in the sum of non-trivial irreducible subrepresentations of $V$ (see the discussion before Lemma~\ref{lem: weight spaces are integral}).

\begin{proposition}[Alignment]\label{prop:alignment}
Let $k_1 > \ref{a:implicitfctthm}$ and $k_2 \geq  3k_1$.
Let $\Mbf\subset\rho(\Gbf)$ be a $\Q_{\gp}$-subgroup containing $\theta_\gp(\SL_2)$
so that $M= \Mbf(\Q_{\gp})$ is $k_1$-generated by nilpotents of pure non-zero weight.
Suppose that $\mathcal{O}_1,\mathcal{O}_2 \subset M[3k_1]$ are two subsets of relative measure at least $\frac{2}{3}$.

Let $g\in G_{\gp}[3k_1]$ be an element that does not ${\gp}^{-k_2}$-normalize $\mfrak$.
Then there exist $m_1 \in \mathcal{O}_1$ and $m_2 \in \mathcal{O}_2$ such that $\exp(\wpz) = m_1 g m_2$ satisfies 
\begin{align*}
\wpz\in \rfrak_{\mfrak}[3k_1]\quad \text{and}\quad \norm{\wpz^{\mathrm{nt}}} \geq {\gp}^{-4k_2}.
\end{align*}
\end{proposition}

Let $C < G_\gp$ be the centralizer of $\theta_{\gp}(\SL_2(\Q_{\gp}))$ and let $\cfrak$ be its Lie algebra.
We need the following lemma.

\begin{lemma}\label{lemma:alignmenttechnical}
Let $k_1>\ref{a:implicitfctthm}$, $k_2\geq 3k_1$, and let $\Mbf \subset \rho(\G)$ be a $\Q_\gp$-group containing $\theta_\gp(\SL_2)$.
Assume that $M= \Mbf(\Q_\gp)$ is $k_1$-generated by nilpotents of pure non-zero weight.
Let $\mathcal{O} \subset M[3k_1]$ be a subset of relative measure $>\frac{1}{\gp}$.

Suppose that $g \in G_p[3k_1]$ satisfies
\begin{align}\label{eq:centralizer non-avoid variety}
\mathcal{O}g \subset C[3k_1]M[3k_1]G_\gp[4k_2].
\end{align}
Then $\mfrak$ is ${\gp}^{-k_2}$-normalized by $g$.
\end{lemma}

As an ineffective analogue, the reader may verify that $\{g \in \G: \Mbf g \subset \overline{\Cbf \Mbf}^z\}$ is a group whose identity component is exactly the identity component of the normalizer of $\Mbf$ (here $\Cbf$ is the centralizer of $\theta_\gp(\SL_2)$).
Indeed, if $g \in \G$ is close enough to the identity and satisfies $\Mbf g \subset \overline{\Cbf \Mbf}^z$ then $g \in \Cbf\Mbf$, i.e.\ $g =cm$ with $c \in \Cbf$ and $m\in \Mbf$. 
It follows from $\Mbf g \subset \overline{\Cbf \Mbf}^z$ that $c^{-1} \Mbf c \subset \overline{\Cbf \Mbf}^z$. But this implies that $\Ad_c^{-1}$ fixes the horospherical subalgebras $\mfrak^+$ and $\mfrak^-$. 
Hence, $\Ad_c^{-1}$ fixes $\mfrak$ as the horospherical subalgebras generate (as a Lie algebra) all of $\mfrak$. It follows that $c$ and, hence, $g$ are in the normalizer of $\Mbf$.

\begin{proof}[Proof of Lemma~\ref{lemma:alignmenttechnical}]
We first reduce the claim to the case of $g \in C[3k_1]$.
Covering $M[3k_1]$ by cosets of $M[2k_2]$, one sees that there exists $m_1 \in M[3k_1]$ such that 
\begin{align*}
\mathcal{O}' =
\mathcal{O}m_1 \cap M[2k_2] \subset M[2k_2]
\end{align*}
has relative measure $>\frac{1}{\gp}$. 
Multiplying $m_1$ by a suitable element of $M[2k_2]$ on the right, we may further assume that $\mathcal{O}'$ contains the identity.
In particular, $m_1^{-1}g \in C[3k_1]M[3k_1]G_\gp[4k_2]$ (since $m_1^{-1} \in \mathcal{O}$) and so there exists $c \in C[3k_1]$ with $c \in m_1^{-1}g M[3k_1]G_\gp[4k_2]$.
The new element $c$ satisfies
\begin{align*}
(\mathcal{O}m_1)c \subset \mathcal{O}gM[3k_1]G_\gp[4k_2] \subset C[3k_1]M[3k_1]G_\gp[4k_2].
\end{align*}
As $c \in m_1^{-1}g M[3k_1]G_\gp[4k_2]$, it suffices to verify the conclusion of the lemma for $c$ and $\mathcal O m_1$.
Since $c \in C[3k_1]$ we have for any $\vpz \in \log(\mathcal{O}m_1)$
\begin{align}
\exp(\Ad_{c}^{-1} \vpz) \in c^{-1}C[3k_1]M[3k_1]G_\gp[4k_2]
=C[3k_1]M[3k_1]G_\gp[4k_2].\label{eq:conjintoCM}
\end{align}

For the remainder of the proof, we will only rely on \eqref{eq:conjintoCM} for $\vpz \in  \log(\mathcal{O}')$ (so we will not invoke \eqref{eq:centralizer non-avoid variety} further) and on the fact that $c \in C[3k_1]$.
The set of $\vpz \in \mfrak[3k_1]$ satisfying \eqref{eq:conjintoCM} is rather complicated in view of the Baker-Campbell-Hausdorff formula and the higher order commutator terms in that formula.
The situation is entirely different in the (much) smaller neighborhood $\mfrak[2k_2]$ where commutator terms are absorbed by $G_\gp[4k_2]$:

\begin{claim*}
For any $\vpz \in \mfrak[2k_2]$ we have
\begin{align}\label{eq:manyvecin c+m}
\Ad_{c}^{-1} \vpz \in (\cfrak+\mfrak)[2k_2]+ \gfrak_\gp[4k_2].
\end{align}
\end{claim*}

\begin{proof}[Proof of Claim]
Notice that for any $k \geq 1$ and any $c_0 \in C[k]$ there exist $c_1 \in \exp((\cfrak \cap \rfrak_\mfrak)[k])$ and $c_2 \in (C \cap M)[k]$ with $c_0 = c_1c_2$.
Indeed, the map
\begin{align*}
\cfrak[1] = (\rfrak_\mfrak \cap \cfrak)[1]+(\mfrak \cap \cfrak)[1] \to C[1],\
\vpz = \vpz_1 + \vpz_2 \mapsto \exp(\vpz_1)\exp(\vpz_2)
\end{align*}
is surjective by the inverse function theorem as the derivative at zero is the identity.

We first show \eqref{eq:manyvecin c+m} for all vectors in $\log(\mathcal{O'})$.
Given $\vpz \in \log(\mathcal{O}')$ we may write using \eqref{eq:conjintoCM} and the above observation
\begin{align*}
 \exp(\Ad_{c}^{-1} \vpz) = 
 \exp(\wpz_1)\exp(\wpz_2) \exp(\wpz_3)
\end{align*}
for $\wpz_1 \in (\cfrak \cap \rfrak_\mfrak)[3k_1]$, $\wpz_2 \in \mfrak[3k_1]$, and $\wpz_3 \in \gfrak_\gp[4k_2]$ (where we have absorbed the occurring $(C \cap M)[3k_1]$ term into $M[3k_1]$).
By the Baker-Campbell-Hausdorff formula we have
\begin{align*}
\Ad_{c}^{-1} \vpz = \wpz_1+ \wpz_2 + O(\max\{\norm{\wpz_1}\norm{\wpz_2},\gp^{-4k_2}\})
\end{align*}
where we abuse the Landau notation $O(B)$ to mean an element of $\gfrak_\gp$ of size at most $B$.
Since $\norm{\wpz_1}\norm{\wpz_2} \leq \max\{\norm{\wpz_1},\norm{\wpz_2}\}^2 = \norm{\wpz_1+\wpz_2}^2$ (as $\rfrak_\mfrak$ is an undistorted complement), $\wpz_1+\wpz_2$ is the highest order term in the above right-hand side (unless it is of size $\leq \gp^{-4k_2}$, in which case \eqref{eq:manyvecin c+m} certainly holds).
Using additionally $\Ad_{c}^{-1} \vpz \in \gfrak_\gp[2k_2]$ by definition of $\mathcal{O}'$,
we have $\wpz_1+ \wpz_2 \in \gfrak_{\gp}[2k_2]$ and so $\wpz_1,\wpz_2 \in \gfrak_\gp[2k_2]$ using again that $\rfrak_\mfrak$ is undistorted.
In particular, we have $\norm{\wpz_1}\norm{\wpz_2} \leq \gp^{-4k_2}$ and
\begin{align*}
\Ad_{c}^{-1} \vpz \in \wpz_1 + \wpz_2 + \gfrak_\gp[4k_2] \subset \cfrak[2k_2]+\mfrak[2k_2]+\gfrak_\gp[4k_2].
\end{align*}
Hence \eqref{eq:manyvecin c+m} follows for vectors in $\log(\mathcal{O}')$.

We upgrade this statement to general vectors in $\mfrak[2k_2]$ using linearity of \eqref{eq:manyvecin c+m}.
Let $\vpz_1,\ldots,\vpz_{\dim(\mfrak)}$ be a $\Z_\gp$-basis of $\mfrak[2k_2]$. By Fubini's theorem and using that $\mathcal{O}'$ has relative measure  $>\frac{1}{\gp}$, there exist $t_1,\ldots,t_{\dim(\mfrak)} \in \Z_\gp$ and $t_1' \in \Z_\gp$ such that $t_1-t_1' \in \Z_\gp^\times$ and
\begin{align*}
t_1\vpz_1 + \ldots+t_{\dim(\mfrak)}\vpz_{\dim(\mfrak)},\,
t_1'\vpz_1+t_2 \vpz_2 + \ldots+t_{\dim(\mfrak)}\vpz_{\dim(\mfrak)}
\in \log(\mathcal{O'}).
\end{align*}
Since \eqref{eq:manyvecin c+m} holds for linear combinations of elements of $\log(\mathcal{O}')$,
we obtain by taking the difference of the above two vectors that $(t_1-t_1') \Ad_{c}^{-1}\vpz_1 \in (\cfrak+\mfrak)[2k_2]+ \gfrak_\gp[4k_2]$ and so $\vpz_1$ satisfies \eqref{eq:manyvecin c+m}.
One proceeds analogously to verify \eqref{eq:manyvecin c+m} for the other basis vectors $\vpz_2,\ldots,\vpz_{\dim(\mfrak)}$.
This proves the claim by linearity.
\end{proof}

We now proceed with the proof of Lemma~\ref{lemma:alignmenttechnical}. Given any $k \geq 0$ and any $\vpz \in \mfrak[k]$ we have by the claim that
\begin{align*}
\Ad_{c}^{-1} \vpz 
&= \gp^{-2k_2+k} \Ad_{c}^{-1} (\gp^{2k_2-k}\vpz)\\
&\in \gp^{-2k_2+k} \big((\cfrak+\mfrak)[2k_2]+ \gfrak_\gp[4k_2]\big)
= (\cfrak+\mfrak)[k]+ \gfrak_\gp[2k_2+k].
\end{align*}
In particular,
\begin{align*}
\Ad_{c}^{-1}(\mfrak[k]) \subset (\cfrak+\mfrak)[k]+ \gfrak_\gp[2k_2].
\end{align*}
Now notice that $\Ad_{c}^{-1} \vpz \in \gfrak_\gp^+$ if $\vpz \in \mfrak^+$ and $\Ad_{c}^{-1} \vpz \in \gfrak_\gp^-$ if $\vpz \in \mfrak^-$.
Since $(\cfrak+\mfrak)^\pm = \mfrak^\pm$, we obtain that
\begin{align*}
\Ad_{c}^{-1} (\mfrak^\pm[k]) \subset \mfrak^\pm[k]+ \gfrak_\gp[2k_2] \subset \mfrak[k]+ \gfrak_\gp[2k_2].
\end{align*}
or equivalently
\begin{align}\label{eq:pmMpreserved}
c^{-1}M^\pm[k]c \subset M[k]G_\gp[2k_2]
\end{align}

We now use effective generation to upgrade \eqref{eq:pmMpreserved} to a similar statement for $M[3k_1]$.
Recall that $M$ is assumed to be $k_1$-generated by nilpotents of pure non-zero weight.
In particular, any element of $M[3k_1]$ can be written as a product of (at most $2 \dim(\mfrak)$ many) elements in $M^\pm[1]$ --- see Lemma~\ref{lem:implicitfctthm}.
Since $M[1]G_\gp[2k_2]$ is a group, this and \eqref{eq:pmMpreserved} for $k=1$ implies
\begin{align*}
c^{-1}M[3k_1]c \subset (M[1]G_\gp[2k_2]) \cap G_\gp[3k_1] = M[3k_1]G_\gp[2k_2].
\end{align*}
In particular, $\Ad_{c}^{-1}(\mfrak[3k_1]) \subset \mfrak[3k_1]+ \gfrak_\gp[2k_2]$ and so
\begin{align*}
\Ad_{c}^{-1}(\mfrak[0]) \subset \mfrak[0]+ \gfrak_\gp[2k_2-3k_1].
\end{align*}
For a $\Z_\gp$-basis $\vpz_1,\ldots,\vpz_{\dim(\mfrak)}$ of $\mfrak[0]$ this implies that $\Ad_{c}^{-1}\vpz_i \equiv \vpz_i' \mod \gp^{2k_2-3k_1}$ for some $\vpz_i' \in \mfrak[0]$. Notice that the vectors $\vpz_1',\ldots, \vpz_{\dim(\mfrak)}'$ are linearly independent modulo $\gp$ and hence also form a $\Z_\gp$-basis of $\mfrak[0]$.
Thus,
\begin{align*}
c^{-1}.(\vpz_1\wedge\ldots\wedge \vpz_{\dim(\mfrak)})
\equiv \vpz_1'\wedge \ldots\wedge \vpz_{\dim(\mfrak)}' \mod \gp^{2k_2-3k_1}
\end{align*}
which proves the lemma as $k_2 \geq 3k_1$.
\end{proof}

\begin{proof}[Proof of Proposition~\ref{prop:alignment}]
Since $\rfrak_\mfrak$ is an undistorted complement, the map $\Psi$ sending 
$\vpz = \vpz_1 + \vpz_2$ with $\vpz_1 \in \rfrak_\mfrak[1]$ and $\vpz_2 \in \mfrak[1]$ to $\exp(\vpz_1) \exp(\vpz_2)$ 
is a diffeomorphism $\gfrak_\gp[1] \to G_\gp[1]$ whose derivative has unit determinant at every point, with the derivative at zero being the identity.
In particular, as $g \in G_p[3k_1]$, we may write for $m \in M[3k_1]$
\begin{align}\label{eq:wmPhi-def}
m g = \exp(\wpz_m) \Phi(m)^{-1}
\end{align}
for some $\wpz_m \in \rfrak_{\mfrak}[3k_1]$ and some $\Phi(m) \in M[3k_1]$. 

\begin{claim*}
The map $\Phi$ defines a diffeomorphism $M[3k_1] \to M[3k_1]$ whose derivative has unit determinant at every point.
In particular, $\Phi$ is measure-preserving.
\end{claim*}

\begin{proof}[Proof of Claim]
We assume first that $g = \exp(\wpz)$ for some $\wpz \in \rfrak_\mfrak[3k_1]$. 
Let $\vpz \in \mfrak[0]$.
By the Baker-Campbell-Hausdorff formula we have for $|t|<1$
\begin{align}\label{eq:exp(tv)exp(w)}
\exp(t \vpz)\exp(\wpz) = \exp(\wpz+ t \vpz + O(\norm{\wpz} \norm{\vpz}|t|) + O(|t|^2)).
\end{align}
As $\Psi$ is analytic, so is its inverse $\Psi^{-1}:G_\gp[1] \to \gfrak_\gp[1]$ (cf.~\cite[Ch.~III]{Serre-LieGroups}).
In particular, we may write for $|t| < 1$
\begin{align*}
\Psi^{-1}(\exp(t \vpz)\exp(\wpz)) = \phi_\rfrak(t) + \phi_\mfrak(t)
\end{align*}
with analytic curves
\begin{align*}
\phi_\rfrak(t) &= \wpz + t a_1 + t^2 a_2 + \ldots \in \rfrak_\mfrak[1],\\
\phi_\mfrak(t) &= t b_1 + t^2 b_2 + \ldots \in \mfrak[1].
\end{align*}
where we used that $\wpz \in \rfrak_\mfrak$.
Using the Baker-Campbell-Hausdorff formula for $\exp(\phi_\rfrak(t))\exp(\phi_\mfrak(t))$ we have
\begin{align*}
\exp\big(\phi_\rfrak(t))\exp(\phi_\mfrak(t)\big)
= \exp\big(\wpz + t(a_1+b_1)+ O(|t| \norm{\wpz}\norm{b_1}) + O(|t|^2)\big).
\end{align*}
Comparing linear terms in $t$ with \eqref{eq:exp(tv)exp(w)} we have
\begin{align*}
\vpz + O(\gp^{-3k_1}\norm{\vpz}) 
= a_1 + b_1 + O(\gp^{-3k_1}\norm{b_1})
\end{align*}
and so, since $\vpz \in \mfrak$, $\norm{b_1} = \norm{\vpz}$ and $b_1 = \vpz + O(\gp^{-3k_1}\norm{\vpz})$. This shows the claim if $g \in\exp(\rfrak_\mfrak[3k_1])$ and the derivative of $\Phi$ is taken at the identity.

The statement for general $g$ and the derivative taken at the identity follows from writing $g = g'm'$ for some $m' \in M[3k_1]$ and $g' \in \exp(\rfrak_\mfrak[3k_1])$.
Moreover, taking the derivative at any other point $m \in M[3k_1]$ reduces to the already proven part of the claim by replacing $g$ with $m g$. The claim follows.
\end{proof}

By the claim, $\mathcal{O} := \mathcal{O}_1 \cap \Phi^{-1}(\mathcal{O}_2)$ is a subset of relative measure at least~$\frac{1}{3}$.
By Lemma~\ref{lemma:alignmenttechnical} and since $g$ does not $\gp^{-k_2}$-normalize $\mfrak$ there exists $m \in \mathcal{O}$ such that $mg \not \in C[3k_1]M[3k_1]G_\gp[4k_2]$ (as long as $k_2 \geq 3k_1$).
In particular, $\wpz_m$ defined in \eqref{eq:wmPhi-def} satisfies $\wpz_m \not \in \cfrak[3k_1] + \gfrak_\gp[4k_2]$ as otherwise
\begin{align*}
mg =  \exp(\wpz_m) \Phi(m)^{-1} \in C[3k_1]G_\gp[4k_2]\Phi(m)^{-1} \subset C[3k_1]M[3k_1]G_\gp[4k_2],
\end{align*}
where we used that $G_\gp[4k_2]$ is a normal subgroup of $G_\gp[0]$. 
In other words, we have asserted that $\norm{(\wpz_m)^\mathrm{nt}} > \gp^{-4k_2}$ which proves the proposition for $m_1=m\in \mathcal O_1$ and $m_2=\Phi(m)\in\mathcal O_2$. 
\end{proof}

\section{Attaining extra almost invariance}\label{sec:extrainv}

In this section, we prove Proposition~\ref{prop:addinv-intro} which provides extra almost invariance. 
To do so, we combine results from previous sections with an effective ergodic theorem for a version of the averaging operator
\begin{align*}
f \mapsto \int_{\Z_\gp} f(\cdot u(-s)a(t)^{-1}) \de s.
\end{align*}

\subsection{Spectral input}
Recall that the representation of $\theta_\gp(\SL_2(\Q_\gp))$ on 
\begin{align*}
L^2_0(Y_\data) = \big\{f \in L^2(Y_\data): \textstyle \int_X f \de \mu = 0\big\}
\end{align*}
is $\tempered$-tempered for some $\tempered>0$ depending only on $\dim(\Hbf)$.
That is, matrix coefficients of the $\tempered$-fold tensor product of $L^2_0(Y_{\data})$ are in $L^{2+\varepsilon}\big(\theta_\gp(\SL_2(\Q_\gp))\big)$ for every $\varepsilon >0$.
If $\Hbf(\Q_\gp)$ has rank at least $2$, this follows from property $(T)$.
Otherwise, it is a strong version of property $(\tau)$; see \cite{Selberg-ThreeSixteenth,Jacquet-Langlands,Burger-Sarnak,Cl-tau,Oh-propertyT,Gorodnik-Maucourant-Oh-adelic}. We also refer to the discussion in \cite[\S4]{EMMV}.

Similarly, the representation of $\theta_\gp(\SL_2(\Q_\gp))$ on $L^2_0(X)$ is $\tempered$-tempered.

\subsection{An effective ergodic theorem}\label{sec:effergthm}

In this subsection, we establish an effective ergodic theorem.
Similar results were also crucially used in e.g.~\cite{EMV,AW-realsemisimple,EMMV} though we use an adaptation here. 
In particular, we will use $L^\infty$-norms to avoid a degree increase step that was necessary for $L^2$-Sobolev norms.

Recall that given a $C^1$-function $f$ on $X$ we write $\lev(f)$ for the level of $f$, i.e.~least integer $L\geq 1$ such that $f$ is invariant under $\prod_\gep G_\gep[\ord_\gep(L)]$.
Moreover, as in the introduction we fix an inner product on $\mathfrak{gl}_N(\R)$
and define the $C^1$-norm $\norm{f}_{C^1(X)} = \norm{f}_{C^1}$ as the maximum of the sup norms of the function and its partial derivatives in directions corresponding to an orthonormal basis of $\gfrak_{\infty}$.
These definitions imply that for any $h \in G_\gp[k]$ and any $C^1$-function $f$
\begin{align}\label{eq:Lipschitzconstant}
\sup_{x \in X}|f(xh)-f(x)| \leq 2 \gp^{-k} \lev(f)\norm{f}_{C^1}.
\end{align}

\begin{definition}\label{def:k1k2typical}
Let $k_2 \geq k_1 \geq 1$ be integers.
A point $x \in X$ is $[k_1,k_2]$-typical for the measure $\mu$ if for any $k \in [k_1,k_2]\cap \Z$ and any ball $B\subset \Z_{\gp}$ of radius at least $\gp^{-\frac{k}{8\tempered}}$ 
the average
\begin{align*}
D_{k,B}f(x) := 
\frac{1}{|B|}\int_{B}f(x u(-s) a^{-k}) \de s- \int f \de \mu 
\end{align*}
satisfies
\begin{align}\label{eq:k1k2typical}
    |D_{k,B}f(x)|  \leq \gp^{-\frac{1}{100\dim(\G)\tempered} k} \lev(f) \norm{f}_{C^1}
\end{align}
for all $f \in C^1(X)$.
Here, $|\cdot|$ is the Haar probability measure on $\Z_\gp$.
\end{definition}

Note that Definition~\ref{def:k1k2typical} requires \eqref{eq:k1k2typical} to hold for \emph{every} $C^1$-function $f$. It is \emph{not} assumed that the point $x$ is in the support of the measure~$\mu$.
We also note that the explicit exponent will be of little importance in what follows, though we point out that it depends only on $\dim(\G)$ (and not on $\mu$).

\begin{proposition}\label{prop:ergthm1}
There exists $\consta\label{a:ergthm}\geq 1$ with the following property:
For any $k_0\geq 1$ with
\begin{align}\label{eq:k0}
\gp^{k_0} \geq \gp^{\ref{a:ergthm}}\cpl(X)^{\ref{a:ergthm}}
\end{align}
the set of points $x \in X$ which are not $[k_0,\infty)$-typical has $\mu$-measure at most $ \gp^{-\frac{1}{4\tempered}k_0}$.
\end{proposition}

We will first show the existence of a large set of points for which the averages in \eqref{eq:k1k2typical} are well-behaved with respect to an \emph{individual} function. (Indeed, combine Lemma~\ref{lem:generic for one fct} below with the Chebyshev inequality.)
We will then use a partition of unity to bootstrap to \emph{all} functions.

\begin{lemma}\label{lem:generic for one fct}
For any $f\in C^1(X)$ and any ball $B \subset \Z_\gp$ of radius at least $\gp^{-k}$ we have
\begin{align*}
\norm{D_{k,B}(f)}_{L^2(\mu)}^2 
\ll \gp^{-\frac{1}{2\tempered}k} \lev(f)^{3} \norm{f}_{\infty}^2.
\end{align*}
\end{lemma}

\begin{proof}
As the representation of $\theta_\gp(\SL_2(\Q_\gp))$ on $L_0^2(Y_{\data})$ is $\tempered$-tempered, we have for any $f \in C^1(Y_{\data})$ and any $s \in \Q_\gp$
\begin{align*}
\big| \langle u(s).&f,f \rangle - \mu(f)\mu(\overline{f}) \big| 
\ll (1+|s|)^{-\frac{1}{2\tempered}} \dim \langle \theta_{\gp}(\SL_2(\Z_{\gp})).f\rangle\,  \norm{f}_{L^2(\mu)}^2.
\end{align*}
Notice also that the dimension $\dim \langle \theta_{\gp}(\SL_2(\Z_{\gp})).f\rangle$ is at most the index of the subgroup
\begin{align*}
\big\{g\in \SL_2(\Z_\gp):\theta_\gp(g) \in G_\gp[\ord_\gp(\lev(f))]\big\}
\end{align*}
in $\SL_2(\Z_\gp)$. 
Since $|\SL_2(\Z/\gp^n\Z)| \leq \gp^{3n}$ for any $n\geq 1$, we have
\begin{align*}
\dim \langle \theta_{\gp}(\SL_2(\Z_{\gp})).f\rangle \leq \lev(f)^3
\end{align*}
In particular,
\begin{align}\label{eq:diagonal matrix coefficients}
\big| \langle u(s).&f,f \rangle - \mu(f)\mu(\overline{f}) \big|
\ll (1+|s|)^{-\frac{1}{2\tempered}} \lev(f)^3 \norm{f}_{\infty}^2.
\end{align}

We finally turn to the problem posed in the lemma. 
Fix $f\in C^1(X)$, $k\geq 1$, and a ball $B$ of radius equal to $\gp^{-\delta k}$ for some $\delta \leq 1$. 
Combining Fubini's theorem and invariance of $\mu$ under the principal $\SL_2$ we have for $f \in C^1(X)$
\begin{align*}
\norm{D_{k,B}(f)}_{L^2(\mu)}^2
= \frac{1}{|B_0|^2}\int_{B_0^2} \big(\langle a^k u(s-s')a^{-k}.f,f\rangle_{L^2(\mu)}-\mu(f)\mu(\overline{f})\big) \de s \de s',
\end{align*}
where $B_0$ is the translate of $B$ containing zero.
Notice that one of the two integration variables $s,s'$ may be removed by substitution in the above. Also, the conjugated ball satisfies $a^ku(B_0)a^{-k} = u(\gp^{-(2-\delta)k}\Z_\gp)$.
Using \eqref{eq:diagonal matrix coefficients} and these observations
\begin{equation*}
\begin{split}
\norm{D_{k,B}(f)}_{L^2(\mu)}^2 
&\ll \gp^{-(2-\delta)k} \int_{|s|\leq \gp^{(2-\delta)k}} (1+|s|)^{-\frac{1}{2\tempered}}\lev(f)^3 \norm{f}_{\infty}^2 \de s \\
&\ll \gp^{-\frac{2-\delta}{2\tempered}k}\lev(f)^3 \norm{f}_{\infty}^2
\end{split}
\end{equation*}
proving the lemma.
\end{proof}

We now construct a partition of unity.
For this, we control first the injectivity radius around each point.
There exist constants $\consta\label{a:injradius}>0$ and $\constc\label{c:injradius}>0$ so that for any $x \in X$ and any $g\in \rho(\G(\R))$ 
\begin{align}\label{eq:injradius}
\Big(\norm{g-\id} \leq \ref{c:injradius}\, \cpl(X)^{-\ref{a:injradius}} \text{ and }xg =x \Big) 
\implies g =\id.
\end{align}
Indeed, by Proposition~\ref{prop:smallvectors} we have $\min_{\wpz \in \Q^N\setminus\{0\}} \content(\rho(g)\wpz)\gg \cpl(X)^{-\star}$ for any $g \in \G(\A)$ (vaguely, there are no short vectors for any coset in the compact space $X$).
Thus, \eqref{eq:injradius} follows from e.g.~\cite[Lemma~2.3]{MSGT-diameter}). 

\begin{lemma}\label{lem:partitionofunity}
For any $L\geq 1$ and any $\varepsilon\in (0,\ref{c:injradius}\, \cpl(X)^{-\ref{a:injradius}})$ there exists a smooth partition of unity $\{\chi_{\varepsilon,L, i}:i \in \mathcal{I}_{\varepsilon,L}\}$ of $X$ with the following properties:
\begin{itemize}
\item The number of functions $\#\mathcal{I}_{\varepsilon,L}$ satisfies
\begin{align*}
\#\mathcal{I}_{\varepsilon,L} \ll \varepsilon^{-\dim(\G)} L^{N^2}\vol(X).
\end{align*}
\item For any $i\in \mathcal{I}_{\varepsilon,L}$, the function $\chi_{\varepsilon,L, i}$ is invariant under $\prod_\gep G_\gep[\ord_\gep(L)]$.
\item For any $i\in \mathcal{I}_{\varepsilon,L}$, there exists $x \in X$ such that $\chi_{\varepsilon,L, i}$ is supported on the neighborhood
\begin{align*}
x \Big(\{g \in \rho(\G(\R)): \norm{g-\id}\leq \varepsilon\} \times \prod_\gep G_\gep[\ord_\gep(L)]\Big).
\end{align*}
\item We have $\norm{\chi_{\varepsilon,L, i}}_{C^1(X)}\ll \varepsilon^{-1}$ for all $i\in \mathcal{I}_{\varepsilon,L}$.
\end{itemize}
\end{lemma}

\begin{proof}
While such constructions are mostly standard, we give here a proof simplified by the fact that $X$ is compact.
We may as well construct a partition of unity of the smooth manifold $X[L] := X/\prod_\gep G_\gep[\ord_\gep(L)]$.
Note that $X[L]$ is a compact manifold of volume $\ll \vol(X)L^{N^2}$ where the volume is intended with respect to a Riemannian metric induced by the fixed inner product on the Lie algebra.
Indeed, for any prime $\ell$ the index of $G_\gep[\ord_\gep(L)]$ in $G_\gep[0]$ is at most $\gep^{\ord_\gep(L)N^2}$.

Let $\mathcal{P}$ be a measurable partition of $X[L]$ consisting of sets of diameter at most $\varepsilon/4$.
We may assume that the cardinality of $\mathcal{P}$ is $\ll \varepsilon^{-\dim(\G)} \vol(X)L^{N^2}$.
Let $m$ be a Haar measure on $\rho(\G(\R))$.
Let $\phi$ be a smooth density on $\rho(\G(\R))$ with $\int \phi \de m =1$ which is supported on the ball of radius at most $\varepsilon/2$ around the identity and which satisfies $\norm{\phi}_{C^1(X)}\ll \varepsilon^{-1}$.
For any measurable function $f$ on $X[L]$ the convolution
\begin{align*}
\phi \ast f(x) = \int f(xg) \phi(g) \de m(g)
\end{align*}
is smooth.
One may now check that $\{\phi \ast (1|_P)\colon P \in \mathcal{P}\}$ satisfies the required properties.
\end{proof}

\begin{proof}[Proof of Proposition~\ref{prop:ergthm1}] 
Our aim for a given $k_0 \geq 1$ is to find `many' $[k_0,\infty)$-typical points. For simplicity, we fix constants $\delta,\delta'\in (0,1)$ to be determined below and to be used in the definition of typical points as follows:  $\delta$ will be the speed of decay of the averages and $\delta'$ will be appear in the size of the expanded balls.

In view of \eqref{eq:k0} we may assume that $k_0$ satisfies $\gp^{-\delta k_0}< \ref{c:injradius}\, \cpl(X)^{-\ref{a:injradius}}$.
For any $k \geq k_0$ and $L\geq 1$ let $\{\chi_{k,L, i}:i \in \mathcal{I}_{k,L}\}$ be the partition of unity from Lemma~\ref{lem:partitionofunity}, where we used $\varepsilon = \varepsilon(k) = \gp^{-\delta k}$ and simplified the notation accordingly.
Notice that the estimate
\begin{align*}
|D_{k,B}f(x)| \leq \gp^{-\delta k}\lev(f) \norm{f}_{C^1(X)}
\end{align*}
is trivial for functions $f$ of level at least $2\gp^{\delta k}$.

Given any ball $B\subset \Z_{\gp}$ of radius at least $\gp^{-\delta'k}$ we have by Lemma~\ref{lem:generic for one fct}
\begin{align*}
\norm{D_{k,B}(\chi_{k,L,i})}_{L^2(\mu)}^2 
\ll \gp^{-\frac{1}{2\tempered}k} L^{3}.
\end{align*}
In particular, by the Chebyshev inequality
\begin{align}\label{eq:discrepancy-onecharfct}
\mu\big(\{x\colon |D_{k,B}(\chi_{k,L,i})(x)| > \gp^{-\delta k} L (\#\mathcal{I}_{k,L})^{-1}\}\big)
\ll \gp^{2\delta k-\frac{1}{2\tempered}k} L (\#\mathcal{I}_{k,L})^2.
\end{align}
Write $\mathsf{Bad}(k_0)$ for the set of points $x \in X$ for which there exists $k \geq k_0$, a ball $B\subset \Z_\gp$ of radius at least $\gp^{-\delta'k}$, a level $L \leq 2\gp^{\delta k}$, and some $i \in \mathcal{I}_{k,L}$ with
\begin{align*}
|D_{k,B}(\chi_{L,k,i})(x)| > \gp^{-\delta k} L (\#\mathcal{I}_{k,L})^{-1}.
\end{align*}
Notice that there are at most $2\gp^{\delta' k}$ many balls $B \subset \Z_\gp$ of radius at least $\gp^{-\delta'k}$.
Using this and the bound on the number $\#\mathcal{I}_{k,L}$ of elements in the partition of unity from Lemma~\ref{lem:partitionofunity} we obtain from \eqref{eq:discrepancy-onecharfct}
\begin{align*}
\mu(\mathsf{Bad}(k_0)) 
&\ll \sum_{k \geq k_0}\gp^{(\delta'+2\delta -\frac{1}{2\tempered})k} \sum_{L \leq 2\gp^{\delta k}}L (\#\mathcal{I}_{k,L})^3 \\
&\ll  \sum_{k \geq k_0}\gp^{(\delta' +(2+3\dim(\G))\delta-\frac{1}{2\tempered})k} \sum_{L \leq 2\gp^{\delta k}}L^{1+3N^2} \vol(X)^3 \\
&\ll  \sum_{k \geq k_0}\gp^{(\delta'+(2+3N^2)\delta-\frac{1}{2\tempered})k} \gp^{\delta k(2+3N^2)} \vol(X)^3\\
&\ll \sum_{k \geq k_0}\gp^{(\delta'+(6N^2 +4)\delta-\frac{1}{2\tempered})k} \vol(X)^3.
\end{align*}
Thus, for $\delta = (6N^2+4)^{-1}(8\tempered)^{-1}$ and $\delta' = \frac{1}{8\tempered}$ we obtain 
\begin{align*}
\mu(\mathsf{Bad}(k_0)) \ll \gp^{-\frac{1}{4\tempered}k_0}\vol(X)^3.
\end{align*}
In fact, choosing $\delta$ slightly smaller (as in Definition~\ref{def:k1k2typical}) we can guarantee that $\mu(\mathsf{Bad}(k_0)) \leq \gp^{-\frac{1}{4\tempered}k_0}$ in view of \eqref{eq:k0} and Proposition~\ref{prop: vol complexity'}.

We will now show that any $x \not\in \mathsf{Bad}(k_0)$ is $[k_0,\infty)$-typical.
So let $x \not\in \mathsf{Bad}(k_0)$ and suppose that $f \in C^1(X)$ and $k \geq k_0$ are given. Recall that if the level $L=\lev(f)$ satisfies $L \leq 2\gp^{\delta k}$, we are done.
So assume otherwise.
Write $x_i$ for the `center points' in the definition of the partition of unity $\{\chi_{k,L, i}:i \in \mathcal{I}_{k,L}\}$. Then
\begin{align*}
f = \sum_{i \in \mathcal{I}_{k,L}} f(x_i) \chi_{k,L, i} + O\big(\gp^{-\delta k} \norm{f}_{C^1(X)}\big)
\end{align*}
by the mean value theorem. Moreover, by definition of the set $\mathsf{Bad}(k_0)$ we have for any ball $B$ of radius at least $\gp^{-\delta' k}$
\begin{align*}
\Big|\sum_{i \in \mathcal{I}_{k,L}} f(x_i) D_{k,B}\chi_{k,L, i}(x)\Big|
\leq \gp^{-\delta k} L \norm{f}_\infty.
\end{align*}
Together, these estimates show that
\begin{align*}
|D_{k,B}f(x)| \ll \gp^{-\delta k} L \norm{f}_{C^1(X)}.
\end{align*}
Hence the proposition follows by recalling that $\gp$ is assumed to be sufficiently large and decreasing $\delta$ slightly.
\end{proof}

\begin{proposition}\label{prop:ergthm2}
Suppose that $\mu$ is $\gp^{-k}$-almost invariant under all elements of a closed subgroup $M_0 < G_{\gp}[0]$. 
Then for any $k_0\geq 1$ with \eqref{eq:k0} the proportion of pairs $(x,m) \in X \times M_0$ for which $xm$ is not $[k_0,k/(4\dim\G)]$-typical is at most $\gp^{-\frac{1}{4\tempered} k_0}$.

In particular, for every $n \geq 1$ there exists a subset $X' \subset X$ with $\mu(X') \geq \frac{2}{3}$ and with the following property:
For any $x \in X'$ the measure of the set of $s \in \Q_{\gp}$ with $|s| \leq {\gp}^n$ and
\begin{align*}
\vol\big(\{m \in M_0: xu(s)m \text{ is } [k_0, \tfrac{1}{4\dim\G}k]\text{-typical}\}\big) \geq \tfrac{2}{3}
\end{align*}
is at least $\frac{2}{3}{\gp}^n$.
Here, $\vol(\cdot)$ denotes the Haar probability measure on $M_0$.
\end{proposition}

\begin{proof}
The proof of the first statement is largely analogous to the proof of Proposition~\ref{prop:ergthm1}.
Given a smooth function $\chi$ on $X$ notice that the translate $\chi(\cdot u(s) a^{-k'})$ has level at most $\gp^{2\dim(\G)k'}\lev(\chi)$ for any $k'$ and any $s \in \Z_\gp$.
Correspondingly, for any $k' \geq 0$ and for any ball $B\subset\Z_\gp$ of radius at least $\gp^{-k'}$
\begin{align*}
\int_{M_0} \int_X |(D_{k',B}\chi)&(xm)|^2 \de \mu(x) \de \vol(m)\\
&\leq \gp^{-k+2\dim(\G) k'} \lev(\chi)\norm{\chi}_{C^1(X)}^2 + \norm{D_{k',B}\chi}_{L^2(\mu)}^2 \\
&\ll \gp^{-k+2\dim(\G) k'} \lev(\chi)\norm{\chi}_{C^1(X)}^2 + \gp^{-\frac{1}{2\tempered}k'} \lev(\chi)^{3} \norm{\chi}_{\infty}^2\\
&\ll \gp^{-\frac{1}{2\tempered}k'} \lev(\chi)^3\norm{\chi}_{C^1(X)}^2
\end{align*}
where we used $\gp^{-k}$-almost invariance for the first inequality (cf.~Definition~\ref{def:almost invariance}), Lemma~\ref{lem:generic for one fct} for the second, and where we assumed $k' \leq k/(4\dim\G)$ in the third.
From here, one may proceed in exactly the same fashion as in the proof of Proposition~\ref{prop:ergthm1}. 
The only minor difference is that the above estimate involves derivatives of the elements of the partition of unity; these are controlled by Lemma~\ref{lem:partitionofunity}.

The second statement follows from Fubini's theorem and the first statement. 
Indeed, set for $x \in Y_\data$ and $|s|\leq \gp^n$
\begin{align*}
f(x,s) = \vol\big(\{m \in M_0: xu(s)m \text{ is not } [k_0,k/(4\dim\G)]\text{-typical}\}\big).
\end{align*}
By $U$-invariance of $\mu$, the first statement in the proposition, and since $k_0$ is assumed sufficiently large, we have
\begin{align*}
\gp^{-n}\int \int_{|s|\leq \gp^n} f(x,s) \de s \de \mu(x) \leq \gp^{-\star k_0} < \tfrac{1}{27}.
\end{align*}
By Chebyshev's inequality, the $\mu$-measure of $x \in Y_\data$ with $\gp^{-n}\int_{|s|\leq \gp^n} f(x,s) \de s > \frac{1}{9}$ is less than $\frac{1}{3}$.
For any $x \in Y_\data$ outside this exceptional set, i.e.~with 
\begin{align*}
\gp^{-n}\int_{|s|\leq \gp^n} f(x,s) \de s \leq \tfrac{1}{9},
\end{align*}
another application of Chebyhev's inequality shows that
\begin{align*}
\big|\big\{|s| \leq \gp^n: f(x,s) \geq \tfrac{1}{3}\big\}\big| \leq \tfrac{1}{3}\gp^n.
\end{align*}
This proves the proposition.
\end{proof}

\subsection{An elementary property of almost invariance}\label{sec:almost invariance}

In the following, we verify a few elementary properties of the notion of almost invariance defined in \S\ref{sec:outline-effnotions}.
We note that almost invariance under $\gp$-adic directions in the Lie algebra is significantly less well-behaved than almost invariance under archimedean directions (see e.g.~\cite[\S 8]{EMV} for properties of the latter).
For instance, if $\mu$ is `almost invariant' under two vectors $\vpz,\wpz\in \gfrak_{\gp}[0]$ it is unclear to what extent, in general, $\mu$ is also `almost invariant' under their sum $\vpz+ \wpz$ (or other linear combinations).
Issues arising from this will be resolved via the effective generation result in Theorem~\ref{thm:effgen-intro}.

The following properties are direct consequences of the definition (Definition~\ref{def:almost invariance}):
\begin{enumerate}[label=(A\theenumi)]
\item\label{item:alminv-trivialalminv} $\mu$ is $2\gp^{-k}$-almost invariant under all $g \in G_\gp[k]$ (see also \eqref{eq:trivialalminv}).
\item\label{item:alminv-Lipcont} 
If $\mu$ is $\gp^{-k}$-almost invariant under $g \in G_\gp[0]$, then for any $h\in G_\gp[k']$ the measure $\mu$ is $(\gp^{-k}+2\gp^{-k'})$-almost invariant under $gh$.
\item\label{item:alminv-group} If $\mu$ is $\gp^{-k_1}$-almost invariant under $g_1 \in G_\gp[0]$ and $\gp^{-k_2}$-almost invariant under $g_2 \in G_\gp[0]$, then $\mu$ is also $\gp^{-k_1}$-almost invariant under $g_1^{-1}$ and $(\gp^{-k_1}+\gp^{-k_2})$-almost invariant under $g_1g_2$. 
(Indeed, $\lev(f(\cdot g_i^{-1})=\lev(f)$ for any $f$ and $i=1,2$.)
\item\label{item:alminv-conjugation} If $\mu$ is $\gp^{-k}$-almost invariant under $g \in G_\gp$ and $h \in \theta_\gp(\SL_2(\Q_\gp))$, then $\mu$ is $\gp^{-k}\norm{h}^2$-almost invariant under $gh$ and $hg$. (Indeed, the level of $f(\cdot h)$ is at most $\norm{h}^2\lev(f)$ for any $f$.)
In particular, if $\mu$ is $\gp^{-k}$-almost invariant under $\vpz \in \gfrak_\gp[0]$ then $\mu$ is $\norm{h}^4\gp^{-k}$-almost invariant under $\Ad(h)\vpz$.
\end{enumerate}

Here, we prove the following (arguably minimalistic) statement.

\begin{proposition}\label{prop:from1elto1param}
There exists $\consta\label{a:alminv highest weight}\geq 1$ with the following property.

Let $V \subset \gfrak_\gp$ be a subspace invariant under $\theta_\gp(\SL_2(\Q_{\gp}))$ and let
\begin{align*}
\vpz = \sum_{\lambda>0} \vpz_\lambda \in V^{\mathrm{hw}}[0]\setminus V^{\mathrm{hw}}[1]
\end{align*}
be a vector with pure-weight components $\vpz_\lambda \in V^{(\lambda)}[0]$. 
Suppose that $\mu$ is $\gp^{-k}$-almost invariant under $\exp(\vpz)$.
Then there exists a highest weight vector $\wpz \in V^{\mathrm{hw}}[0]\setminus V^{\mathrm{hw}}[1]$ of pure weight such that $\mu$ is $\gp^{-k/\ref{a:alminv highest weight}+\ref{a:alminv highest weight}}$-almost invariant under $\wpz$. 
\end{proposition}

\begin{proof}
In view of the conclusion, we may assume that $k>B$ for some $B=B(N)\geq 1$ to be determined in the course of the proof.
Indeed, recall from \ref{item:alminv-trivialalminv} that $\mu$ is trivially $\gp$-almost invariant under any vector in $\gfrak_\gp[0]$.

We first claim that there exists $\vpz'$ with the same properties as $\vpz$ and with pure non-zero weight $\lambda$ such that $\mu$ is $\gp^{-\star k + \star}$-almost invariant under $\exp(\vpz')$.
We find $\vpz'$ by induction and a simple case distinction according to whether or not the component of largest weight in $\vpz$ is `small'.

Let $\lambda_0$ be maximal with $\vpz_{\lambda_0} \neq 0$. 
If $\lambda_0 =1$, we are done, so assume otherwise.
Fix two parameters $k_1,k_2< k$ to be determined later (with $k_1,k_2 \geq \star k$).

Assume first that $\norm{\vpz_{\lambda_0}}\geq \gp^{-k_1}$.
By $\gp^{k_2}$-fold iteration of $\exp(\vpz)$ using \ref{item:alminv-group}, the measure $\mu$ is $\gp^{k_2-k}$-almost invariant under $\exp(\gp^{k_2} \vpz)$.
Let 
\begin{align*}
k_3= \lfloor (k_2- \log_\gp(\norm{\vpz_{\lambda_0}}))/\lambda_0 \rfloor.
\end{align*}
Observe that 
\begin{align*}
1 \geq \gp^{-k_2+\lambda_0 k_3} \norm{\vpz_{\lambda_0}}=
\norm{\Ad(a^{k_3})\gp^{k_2} \vpz_{\lambda_0}} > \gp^{-\lambda_0}
\end{align*}
and $k_3 \leq \frac{1}{\lambda_0}(k_1+k_2)$.
In particular,
\begin{align*}
\vpz'=\Ad(a^{k_3})\gp^{k_2} \vpz_{\lambda_0} \in V[0]\setminus V[\lambda_0].
\end{align*}
For any smaller weight $\lambda < \lambda_0$ we have
\begin{align*}
\norm{\Ad(a^{k_3})\gp^{k_2} \vpz_{\lambda}} \leq \gp^{-k_2 +\lambda k_3}
\leq \gp^{\frac{\lambda}{\lambda_0}k_1-(1-\frac{\lambda}{\lambda_0})k_2}.
\end{align*}
Notice that we have not specified either of the parameters $k_1,k_2$ up to this point; we choose $k_1 = \lfloor \frac{1}{2}(\frac{\lambda_0}{\lambda_0-1}-1)k_2\rfloor$ such that $\norm{\Ad(a^{k_3})\gp^{k_2} \vpz_{\lambda}} \leq \gp^{-\star k_2}$ for any $\lambda <\lambda_0$.
By \ref{item:alminv-conjugation} and the earlier iteration, $\mu$ is $\gp^{\star k_3 +k_2 -k}$-almost invariant under $\exp(\Ad(a^{k_3})\gp^{k_2} \vpz)$ where $k_3 \leq k_1+k_2 \leq 2 k_2$ and
\begin{align*}
\Ad(a^{k_3})\gp^{k_2} \vpz = \vpz' + O(\gp^{-\star k_2}).
\end{align*}
Thus, $\mu$ is also $(\gp^{-\star k_2}+ \gp^{\star k_2-k})$-almost invariant under $\exp(\vpz')$ by \ref{item:alminv-Lipcont}.
Choosing $k_2 = \lfloor \star k\rfloor$ appropriately, $\mu$ is $\gp^{-\star k}$-almost invariant under $\exp(\vpz')$.
We may also apply $a$ and iterate by a suitable power of $\gp$ so that $\vpz' \in V[0]\setminus V[1]$ in which case $\mu$ is $\gp^{-\star k+\star}$-almost invariant under $\exp(\vpz')$.
This proves the claim under the assumption $\norm{\vpz_{\lambda_0}}\geq \gp^{-k_1}$.

If $\norm{\vpz_{\lambda_0}}< \gp^{-k_1}$, the measure $\mu$ is $(\gp^{-k}+\gp^{-k_1})$-almost invariant under the element $\exp(\sum_{\lambda<\lambda_0}\vpz_\lambda)$ by \ref{item:alminv-Lipcont}.
Since $k_1 \gg k$, $\mu$ is $\gp^{-\star k}$-almost invariant under $\exp(\sum_{\lambda<\lambda_0}\vpz_\lambda)$.
In that case, we may replace the vector $\vpz$ with $\sum_{\lambda<\lambda_0}\vpz_\lambda$ and restart the above argument with new parameters.

We now show that $\mu$ is $\gp^{-\star k+\star}$-almost invariant under $\vpz'$.
We first claim that $\mu$ is $\gp^{-\star k+\star}$-almost invariant under $\exp(t^\lambda\vpz')$ for all $t \in \Z_{\gp}$.
Indeed, $\mu$ is $|t|^{\star}\gp^{-\star k+\star}$-almost invariant under $a(t) \exp(\vpz') a(t)^{-1} = \exp(t^\lambda \vpz')$ which implies the claim for $|t| \geq \gp^{-\star k+\star}$. 
Otherwise, the claim follows directly from \ref{item:alminv-trivialalminv}.

By Hensel's lemma, any $t' \in \Z_{\gp}$ with $t' \equiv 1 \mod \gp$ is of the form $t' = t^\lambda$ for some $t\in \Z_{\gp}$ (using $\gp >\dim(\gfrak)>\lambda$).
As  $\mu$ is $\gp^{-\star k+\star}$-almost invariant under $\exp((t_1^\lambda-t_2^\lambda)\vpz')$ for all $t_1,t_2 \in \Z_{\gp}$ by \ref{item:alminv-group}, this shows that $\mu$ is  $\gp^{-\star k+\star}$-almost invariant under $\exp(t\vpz')$ for all $t \in \gp\Z_{\gp}$. 
This proves the proposition.
\end{proof}

\subsection{Extra almost invariance from transversal pairs of typical points}

The following proposition establishes `additional almost invariance' assuming the existence of `transversal' typical points; we will assert the latter in Proposition~\ref{prop:transversalpoints} below.

\begin{proposition}\label{prop:addalminv}
There exists $\consta\label{a:extraalminv}>1$ depending only on $N$ with the following property.

Let $\rfrak \subset \gfrak_{\gp}$ be an undistorted $\theta_\gp(\SL_2(\Q_{\gp}))$-invariant subspace.
Let $k_1,k_2 \in \N$ with $k_2 > k_1 >\ref{a:extraalminv}$.
Suppose that there exist two $[k_1/(3\dim(\gfrak)),k_2]$-typical points $x_1,x_2 \in X$ so that
\begin{align*}
x_2 = x_1 g\exp(\vpz)
\end{align*}
where $g \in \rho(K_f)$ centralizes $\theta_\gp(\SL_2(\Q_\gp))$ and where $\vpz \in \rfrak$ satisfies
\begin{align*}
\norm{\vpz} \leq \gp^{-k_1}\quad \text{and}\quad\norm{\vpz^{\mathrm{nt}}} \geq \gp^{-k_2}.
\end{align*}
Then $\mu$ is $\gp^{-k_1/\ref{a:extraalminv}}$-almost invariant under a highest weight vector $\wpz\in \rfrak[0] \setminus \rfrak[1]$ of non-zero weight.
\end{proposition}

We will use the following simple observation relying on polynomial behaviour.

\begin{lemma}
Let $\rfrak \subset \gfrak_{\gp}$ be a $\theta_\gp(\SL_2(\Q_{\gp}))$-invariant subspace. 
Write $\rfrak = \sum_i \rfrak_i$ for irreducible subrepresentations $\rfrak_i$ as in Lemma~\ref{lem: reduction completely red}.
Let $\vpz = \sum_i \vpz_i\in \rfrak[0]$ where  $\vpz_i \in \rfrak_i[0]$.
Then the measure of $s \in \Z_{\gp}$ satisfying
\begin{align}\label{eq:goodshear}
\norm{\big(\Ad(u(s)) \vpz_i\big)^{\mathrm{hw}}} = \norm{\Ad(u(s)) \vpz_i} = \norm{\vpz_i} \text{ for all $i$}
\end{align}
is at least $1-\frac{\dim(\rfrak)}{\gp}$. 
In particular, there exists $s\in \Z_\gp$ with \eqref{eq:goodshear}.
\end{lemma}

\begin{proof}
Given any polynomial $f \in \Z_\gp[s]$ of degree $d < \gp$ with at least one coefficient of absolute value one, the measure of $\{s \in \Z_\gp: \gp \mid f(s)\}$ is exactly the proportion of zeros of $f \mod \gp$ in $\mathbb{F}_\gp$.
Since $f \mod \gp$ is non-zero, it has at most $d$ zeros and so 
\begin{align*}
|\{s \in \Z_\gp: \gp \nmid f(s)\}| \geq 1-d/\gp.
\end{align*}

Returning to the problem in the lemma, for every $i$ choose $\wpz_i \in \rfrak_i^{\mathrm{hw}}[0]\setminus\rfrak_i^{\mathrm{hw}}[1]$ and let $f_i(s)$ be the polynomial of degree $\dim(\rfrak_i)-1$ with
\begin{align*}
\Big(\Ad(u(s)) \frac{\vpz_i}{\norm{\vpz_i}}\Big)^{\mathrm{hw}} = f_i(s) \wpz_i.
\end{align*}
By Lemma~\ref{lem: weight spaces are integral}, each of these polynomials has at least one coefficient of absolute value one.
The lemma thus follows by applying the above general observation to the polynomial $f(s)=\prod_i f_i(s)$ of degree at most $\dim(\rfrak)$.
\end{proof}

\begin{proof}[Proof of Proposition~\ref{prop:addalminv}]
We assume that $k_1 > C$ for some $C=C(N)$ to be determined.

Write $\rfrak = \sum_i \rfrak_i$ as a direct sum of irreducible subrepresentations with highest weights $\lambda_i$ with $\rfrak[0] = \sum_i \rfrak_i[0]$ in view of Lemma~\ref{lem: reduction completely red}.
Express $\vpz = \sum_i\vpz_i$ accordingly where $\vpz_i \in \rfrak_i[1]$ for each $i$.
Let $s_0 \in \Z_\gp$ be as in \eqref{eq:goodshear} and let $k \geq 0$ be minimal such that 
\begin{align*}
\norm{\Ad(a^k u(s_0))\vpz} > \gp^{-\dim(\gfrak)}
\end{align*}
By the assumptions on $\vpz^{\mathrm{nt}}$, $k$ indeed exists and satisfies $k_1/\dim(\gfrak) \leq k \leq k_2$.
Moreover,
\begin{align*}
\wpz :=  (\Ad(a^k u(s_0))\vpz)^{\mathrm{hw}} =
\Ad(a^k u(s_0)) \vpz + O(\gp^{-k}).
\end{align*}

Let $B$ be the ball around $s_0$ with radius $\gp^{-\lfloor \frac{1}{8\tempered} k\rfloor}$.
Then any $s \in B$ also satisfies \eqref{eq:goodshear} and we have
\begin{align}\label{eq:asympinsmallball}
\Ad(a^k u(s))\vpz =\wpz + O(\gp^{-\star k}).
\end{align}
Indeed, $\Ad(a^k u(s))\vpz = (\Ad(a^k u(s))\vpz)^{\mathrm{hw}} + O(\gp^{-k})$ and
\begin{align*}
(\Ad(a^k u(s))\vpz)^{\mathrm{hw}}
&= \sum_i (\Ad(a^k u(s))\vpz_i)^{\mathrm{hw}}
= \sum_i \gp^{-k \lambda_i} (\Ad(u(s))\vpz_i)^{\mathrm{hw}}\\
&= \sum_i \gp^{-k \lambda_i} \big( (\Ad(u(s_0))\vpz_i)^{\mathrm{hw}} + O(\gp^{-\star k}\norm{\vpz_i})\big)\\
&= \sum_i \gp^{-k \lambda_i} (\Ad(u(s_0))\vpz_i)^{\mathrm{hw}} + O(\gp^{-\star k})
= \wpz + O(\gp^{-\star k})
\end{align*}
where in the fourth equality we used that $\norm{\vpz_i} \leq \gp^{-k\lambda_i}$ by choice of $k$.

Recall that both points $x_1,x_2$ are assumed to be $[k_1/\dim(\gfrak),k_2]$-typical.
Thus, we have for any $f \in C^1(X)$
\begin{align*}
\int f \de \mu
&= \int_{B} f\big(x_2 u(-s)a^{-k}\big)\de s + O(\gp^{-\star k}\lev(f)\norm{f}_{C^1(X)}) \\
&= \int_{B} f\big(x_1 u(-s) a^{-k} g \exp(\Ad(a^k u(s))\vpz)\big)\de s + O(\gp^{-\star k}\lev(f)\norm{f}_{C^1(X)})\\
&= \int_{B} f\big(x_1u(-s)a^{-k} g\exp(\wpz)\big)\de s + O(\gp^{-\star k} \lev(f)\norm{f}_{C^1(X)})\\
&= \int f(\cdot g\exp(\wpz))\de \mu + O(\gp^{-\star k}\lev(f)\norm{f}_{C^1(X)}),
\end{align*}
where we used \eqref{eq:asympinsmallball} together with \eqref{eq:Lipschitzconstant} for the third equality.
This proves that $\mu$ is $\gp^{-\star k_1}$-almost invariant under $g\exp(\wpz)$ (using $\gp$ sufficiently large and $k \gg k_1$).
Note that by construction $\wpz = \wpz^{\mathrm{nt}}$, $\wpz \in \rfrak[0]\setminus\rfrak[\dim(\gfrak)]$ and $\wpz$ is centralized by $U$ i.e.~$\wpz \in \rfrak^{\mathrm{hw}}$.

We claim that $\mu$ is in fact $\gp^{-\star k_1}$-almost invariant under $\exp(\wpz)$ (removing $g$).
To see this, we apply the above argument for $k$ much smaller than $k_1$.
More explicitly, taking $k = \lfloor k_1/(2\dim(\gfrak))\rfloor$ we have $\norm{\Ad(a^ku(s))\vpz} \leq \gp^{-k_1/2}$ for any $s \in \Z_\gp$.
Thus, for $f \in C^1(X)$ using \eqref{eq:Lipschitzconstant}
\begin{align*}
f(x_2u(-s)a^{-k}) 
&= f\big(x_1u(-s)a^{-k} g \exp(\Ad(a^ku(s))\vpz)\big) \\
&= f(x_1 u(-s)a^{-k} g) + O\big(\gp^{-k_1/3} \lev(f) \norm{f}_{C^1}\big).
\end{align*}
Using that $x_1,x_2$ are $[k_1/(3\dim(\gfrak)),k_2]$-typical and proceeding as before, we obtain that $\mu$ is $\gp^{-\star k_1}$-almost invariant under $g$. Hence, $\mu$ is $\gp^{-\star k_1}$-almost invariant under $g^{-1}g\exp(\wpz) = \exp(\wpz)$ (cf.~\ref{item:alminv-group}).

Since $\wpz=\wpz^{\mathrm{hw}}=\wpz^{\mathrm{nt}}$, there is some $m \leq \dim(\gfrak)$ such that $\wpz' := \gp^{m}\Ad(a)\wpz\in \rfrak[0]\setminus\rfrak[1]$.
The measure $\mu$ is $\ll \gp^{-\star k_1+\star}$-almost invariant under $\exp(\wpz')$ by \ref{item:alminv-group} and \ref{item:alminv-conjugation}.
The proposition now follows Proposition~\ref{prop:from1elto1param} applied to $\wpz'$.
\end{proof}

\subsection{Proof of Proposition~\ref{prop:addinv-intro}}

We first prove the following technical proposition which relies on the previous results from \S\ref{sec:prepclosing}--\S\ref{sec:alignment} (including the effective closing lemma in Proposition~\ref{prop:closing-lemma}).
Under suitable conditions, it establishes the existence of a pair of typical points with a `transversal' displacement that is useful in view of Proposition~\ref{prop:addalminv}.

\begin{proposition}\label{prop:transversalpoints}
There exist $\consta\label{a:transversal}>0$ depending only on $N$ with the following property.
Let $k_0$ be minimal with \eqref{eq:k0} and let $k_1 \geq \ref{a:transversal} k_0$.
Let $k\geq 1$ be a further integer with
\begin{align*}
\cpl(X)^{\ref{a:transversal}}\gp^{\ref{a:transversal}}\leq \gp^{k} \leq \mcpl(Y_\data)^{1/\ref{a:transversal}}.
\end{align*}
Suppose that $\Mbf < \rho(\G)$ is a $\Q_{\gp}$-subgroup containing $\theta_\gp(\SL_2)$ such that
\begin{enumerate}[label=\textnormal{(\theenumi)}]
\item $M = \Mbf(\Q_\gp)$ is $k$-generated by nilpotents of pure non-zero weight, and 
\item $\mu$ is ${\gp}^{-k_1}$-almost invariant under $M[3k]$.
\end{enumerate}
Let $\rfrak_\mfrak \subset \gfrak_\gp$ be an undistorted $\theta(\SL_2(\Q_{\gp}))$-invariant complement to $\mfrak =\Lie(M)$. 
Then there exist two points $x_1,x_2 \in X$ which are both $[k_0, k_1/(4\dim \G)]$-typical and where $x_2 = x_1 g$ for $g \in \rho(\G(\A))$ satisfying $g_\gep \in \SL_N(\Z_\gep)$ for all $\gep \neq \gp$, $\norm{g_\infty}\leq 2$, and $g_\gp = \exp(\vpz)$ for $\vpz \in \rfrak_\mfrak$ with
\begin{align*}
\norm{\vpz} \leq {\gp}^{-3k},\
\norm{\vpz^{\mathrm{nt}}} \geq {\gp}^{-\ref{a:transversal}k}.
\end{align*}
\end{proposition}

Note that, in practice, we will take $k_1$ to be much larger than $k$ so that the assumption on almost invariance does not follow from mere Lipschitz continuity (see \eqref{eq:trivialalminv}).

\begin{proof}[Proof of Proposition~\ref{prop:transversalpoints}]
Let $\vol(\cdot)$ denote the Haar measure for $M$ normalized for simplicity such that $\vol(M[3k]) =1$.
We use the specialized notion of Diophantine points from \eqref{eq:epsdef}.
We may assume 
\begin{align}\label{eq:mincpl not tiny}
\mcpl(Y_\data)>(\height(\G){\gp})^{A^2}=(\cpl(X){\gp})^{A^2}
\end{align}
for some very large power $A>0$ (otherwise, there exists no $k>0$ as in the proposition and we conclude).
In particular, \eqref{eq:mcpllowerbound} holds.
Lastly, we let $n >0$ be an auxiliary integer to be chosen later (essentially as a multiple of $k$) and assume it satisfies 
\begin{align}\label{eq:n large}
 (\height(\G){\gp})^A \leq \gp^n \leq \mcpl(Y_\data)^{1/A}.
\end{align}
and
\begin{align}\label{eq:intervallength}
\ref{k:effclosing}n/(2\dim \G )\geq 3k.
\end{align}

We begin by constructing a `good' set of points.
In vague words, these will be points $x \in X$ satisfying, in particular, that for many points along a long piece of the $U$-trajectory many translates by the almost invariance group $M[3k]$ are `typical'.
We also require many points along the $U$-trajectory to be `Diophantine'.
For the precise construction, we combine Corollary~\ref{cor:FubiniDio} and Proposition~\ref{prop:ergthm2} (applied for the group $M_0 = M[3k]$).
Thus, we obtain a set $Y_{\mathrm{good}}\subset Y_\data$ of $\mu$-measure at least $\frac{1}{3}$ with the following properties for any $y \in Y_{\mathrm{good}}$:
\begin{itemize}
\item The set of $s \in \Q_{\gp}$ with $|s| \leq {\gp}^n$ for which $\pi_\mathcal{S}(y)u(s)$ is $\mcpl(Y_\data)^{\frac{1}{2\ref{a:manydio}}}$-Diophantine has measure at least $\frac{2}{3}\gp^n$.
\item The measure of the set of $s \in \Q_{\gp}$ with $|s| \leq {\gp}^n$ and
\begin{align*}
\vol\big(\{m \in M[3k]: yu(s)m \text{ is } [k_0,k_1/(4\dim\G)]\text{-typical}\}\big) \geq \tfrac{2}{3}
\end{align*}
is at least $\frac{2}{3}{\gp}^n$.
\end{itemize}

In the following we fix a point $y \in Y_{\mathrm{good}}$ and let $S \subset \{s: |s|\leq \gp^n\}$ be the set of `good times' for this point i.e.~so that for all $s \in S$ we have
\begin{align*}
\vol(\{m \in M[3k]: yu(s)m \text{ is } [k_0,k_1/(4\dim\G)]\text{-typical}\}) \geq \tfrac{2}{3}
\end{align*}
and $\pi_{\mathcal{S}}(y)u(s)$ is $\mcpl(Y_\data)^{\frac{1}{2\ref{a:manydio}}}$-Diophantine (and in particular $\gp^n$-Diophantine when $A\geq 2\ref{a:manydio}$).
Note that $|S| \geq \frac{1}{3}{\gp}^n$ by construction of the set $Y_{\mathrm{good}}$.

We now use the pigeonhole principle to find two points in the $u(S)$-orbit of $x$ which are `close' to each other, but do not lie on the same local $M$-orbit.
In fact, we will construct them so that their displacement does not `almost normalize' $M$.

Cover $X$ by balls of the form 
\begin{align*}
B(x) = \big\{xg: \norm{g}_\infty \leq \tfrac{4}{3},\, 
g_\gep \in G_\gep[0] \text{ for } \gep \neq \gp,\, 
g_\gp \in G_\gp[\lceil \ref{k:effclosing}n/(2\dim \G) \rceil ]\big\}
\end{align*}
for $x \in X$.
There is a finite cover of $X$ by such balls with multiplicity bounded in terms of $N$ where the number of balls is $\ll {\gp}^{\ref{k:effclosing}n/2+\star} \vol(X)\ll {\gp}^{\ref{k:effclosing}n/2+\star} \cpl(X)^\star$ by definition of the volume and Proposition~\ref{prop: vol complexity'}.
In view of our assumption in \eqref{eq:n large}, the number of these balls is less than $\gp^{\ref{k:effclosing}n}$.
In particular, there exists $x \in X$ such that
\begin{align*}
\mathsf{E}= \{s \in S: yu(s) \in B(x)\}
\end{align*}
satisfies $|\mathsf{E}|> \gp^{(1-\ref{k:effclosing})n}$.
Note that for any $s,s' \in \mathsf{E}$ we have $xu(s)=xu(s')g$ for some $g \in \rho(\G(\A))$ with
$\norm{g}_\infty \leq 2$, $g_\gep \in \SL_N(\Z_\gep)$ for all $\gep \neq \gp$, and $g_\gp \in G_\gp[\lceil \ref{k:effclosing}n/(2\dim \G)\rceil]$.

We now wish to apply the effective closing lemma in Proposition~\ref{prop:closing-lemma} to the point $y$ and the set of times $\mathsf{E}$. 
If the assumption in (3) therein holds, the closing lemma is indeed applicable and we obtain that $\mfrak$ is a semisimple ideal of $\gfrak_\gp$. 
But this is impossible by Lemma~\ref{lem:nofactorscontainingHp} and so (3) cannot hold. 
Thus, there exist $s_1,s_2 \in \mathsf{E}$ such that
\begin{align*}
yu(s_2) = yu(s_1)g
\end{align*}
for $g\in \rho(\G(\A))$ with $\norm{g}_\infty \leq 2$, $g_\gep \in \SL_N(\Z_\gep)$ for $\gep \neq \gp$, $g_\gp \in G_\gp[\lceil \ref{k:effclosing}n/(2\dim \G) \rceil ]$ and, crucially,
\begin{align*}
\metric(g_\gp.\hat{\vpz}_\mfrak,\hat{\vpz}_\mfrak) > \gp^{-n}.
\end{align*}
In particular, $\mfrak$ is not $\gp^{-n}$-normalized by $g_\gp$.

We now apply the alignment lemma in Proposition~\ref{prop:alignment} to the above found displacement $g_\gp$.
Since $g_\gp \in G_\gp[\lceil \ref{k:effclosing}n/(2\dim \G) \rceil ]$ and we assumed \eqref{eq:intervallength}, we are able to apply Proposition~\ref{prop:alignment}.
Explicitly, let $\mathcal{O}_1 \subset M[3k]$ be the subset of points $m$ for which $yu(s_1)m^{-1}$ is $[k_0,k_1/(4\dim\G)]$-typical and let $\mathcal{O}_2 \subset M[3k]$ be the subset of points $m$ for which $yu(s_2)m$ is $[k_0,k_1/(4\dim\G)]$-typical.
Since $s_1,s_2 \in S$, the subsets $\mathcal{O}_1,\mathcal{O}_2$ have relative measure at least $\frac{2}{3}$.
By Proposition~\ref{prop:alignment}, there exist $m_1 \in \mathcal{O}_1$ and $m_2 \in \mathcal{O}_2$ such that $m_1g_\gp m_2 = \exp(\vpz)$ for $\vpz \in \rfrak_\mfrak$ with 
\begin{align*}
\norm{\vpz} \leq \gp^{-3k},\ \norm{\vpz^{\mathrm{nt}}} \geq \gp^{-4 n}.
\end{align*}
We choose $n$ minimal such that \eqref{eq:intervallength} holds.
In particular, this guarantees $\norm{\vpz^{\mathrm{nt}}} \geq \gp^{-\star k}$.
Finally, the two points
\begin{align*}
x_1 = y u(s_1)m_1^{-1},\ x_2 = yu(s_2)m_2
\end{align*}
satisfy the requirements of the proposition with displacement $\exp(\vpz)$.
\end{proof}

\begin{proof}[Proof of Proposition~\ref{prop:addinv-intro}]
As in the proposition, let $\Mbf < \rho(\G)$ be a $\Q_\gp$-subgroup containing $\theta_\gp(\SL_2)$ so that $M= \Mbf(\Q_\gp) < G_p$ is $k$-generated by some nilpotents of pure non-zero weight.
Suppose that $\mu$ is $\gp^{-n}$-almost invariant under $M[3k]$ for some auxiliary parameter $n$.
When $n,k$ are sufficiently large, we may apply Proposition~\ref{prop:transversalpoints} with $k$ and with $n$ instead of $k_1$.
We thus obtain two points $x_1,x_2 \in X$ which are $[k_0,n/(4\dim\G)]$-typical so that $x_2 = x_1 \rho(g)$ for $g \in \G(\A)$ satisfying $g_\gep \in K_\gep$ for all $\gep \neq \gp$, $\norm{g_\infty}\leq 2$, and $g_\gp = \exp(\vpz)$ for $\vpz \in \rfrak_\mfrak$ with
\begin{align*}
\norm{\vpz} \leq {\gp}^{-3k},\
\norm{\vpz^{\mathrm{nt}}} \geq {\gp}^{-\ref{a:transversal}k}.
\end{align*}

When $k$ is sufficiently large and $\ref{a:transversal}k \leq n/(4\dim\G)$, we may apply Proposition~\ref{prop:addalminv} and obtain that $\mu$ is $\gp^{-\star k}$-almost invariant under a highest weight vector $\wpz \in \rfrak_\mfrak[0]\setminus\rfrak_\mfrak[1]$ of non-zero weight.
This implies Proposition~\ref{prop:addinv-intro}.
\end{proof}

\section{Proof of the main theorems}\label{sec:proofmainthms}

In this section, we prove Theorems~\ref{thm:main equi} and \ref{thm:asinEMV}.
For clarity, we shall do so over $F=\Q$ first and then reduce to that case.

\subsection{The case $F= \Q$}

We adopt our standing assumptions from \S\ref{sec:standingassumptions} and prove the following seemingly weaker claim.

\begin{proposition}\label{prop:notweakerprop}
There exists $\consta\label{a:notweakerprop}>0$ depending only on $N$ so that for any $f \in C^1(X)$ we have
\begin{align*}
\Big| \int_X f - \int_{Y_\data} f \de \mu \Big| \leq \frac{\cpl(X)^{\ref{a:notweakerprop}} \gp^{\ref{a:notweakerprop}}}{\mcpl(Y_\data)^{1/{\ref{a:notweakerprop}}}} \lev(f) \norm{f}_{C^1(X)}
\end{align*}
\end{proposition}

Since $\gp$ can potentially be of size at least $\log(\cpl(Y_\data))$, the conclusion of the above proposition is empty if $Y_\data$ is contained in an intermediate period of complexity at most logarithmic in the complexity of $Y_\data$.
The same is not the case in e.g.~Theorem~\ref{thm:main equi}.
We will remedy this issue later by inducing over intermediate periods.

We will use the following lemma relying on spectral gap for the ambient space~$X$.

\begin{lemma}\label{lem:concludingstep}
There exists $\consta\label{a:concludingstep}>0$ with the following property.
Let $k\geq 1$ with $\gp^k \geq \gp^{\ref{a:concludingstep}}\cpl(X)^{\ref{a:concludingstep}}$.
Suppose that $\mu$ is $\gp^{-\ref{a:concludingstep} k}$-almost invariant under $G_\gp[k]$.
Then for any $f \in C^1(X)$
\begin{align}\label{eq:finalstepequi}
\Big| \int_X f - \int_{Y_\data}f\de \mu \Big| \leq \gp^{- k} \lev(f)\norm{f}_{C^1(X)}.
\end{align}
\end{lemma}

\begin{proof}
Fix $A=\ref{a:concludingstep}>0$ to be determined in the proof and suppose that $\mu$ is $\gp^{-A k}$-almost invariant under $G_\gp[k]$ and $\gp^k \geq \gp^{A}\cpl(X)^{A}$. 
We fix $f\in C^1(X)$ and wish to show \eqref{eq:finalstepequi}.
We may assume $\lev(f)\leq 2\gp^{k}$ as the statement is trivial otherwise. 

Fix $s \in \Q_\gp$ to be determined (of size $|s|$ some small power of $\gp^{ A k}$).
As $\mu$ is $\gp^{-A k}$-almost invariant under $G_\gp[k]$, we have by \ref{item:alminv-conjugation}
\begin{align}\label{eq:finalstep-1}
\int_X \int_{G_\gp[k]} f(x g_\gp u(s)) \de g_{\gp} \de \mu(x)
= \int_X f \de \mu + O\big(|s|^\star \gp^{-Ak} \lev(f) \norm{f}_{C^1(X)}\big),
\end{align}
where $\de g_{\gp}$ denotes the normalized Haar measure on $G_\gp[k]$.

We estimate the inner integral on the left-hand side of \eqref{eq:finalstep-1} using effective decay of matrix coefficients for $L^2_0(\mu_X)$.
Let $\Omega_f = \prod_{\gep} \Omega_\gep$ be the compact open subgroup of $\rho(\G(\A_f))$ with $\Omega_\gep = G_{\gep}[\ord_\gep(\lev(f))]$ when $\gep \neq \gp$ and $\Omega_\gp =G_\gp[k]$.
Let $\psi$ be a smooth non-negative function on $X$ with the following properties:
\begin{itemize}
    \item $\psi$ is invariant under $\Omega_f$.
    \item The support of $\psi$ is contained in
\begin{align*}
x \big(\{g \in \rho(\G(\R)): \norm{g-\id} < \gp^{-2k}\} \times \Omega_f\big).
\end{align*}
\item $\norm{\psi}_\infty \ll \gp^{\star k}$. (The volume of the real manifold $X/\Omega_f$ is $\ll \gp^{\star k}$.)
\item $\int_X \psi =1$.
\end{itemize}
See also the discussion in the proof of Lemma~\ref{lem:partitionofunity}; in view of the assumption on $k$ and \eqref{eq:injradius}, the above neighborhood of $x$ is injective for $A>0$ sufficiently large.
By Lipschitz continuity we conclude that
\begin{align}
\int_{G_\gp[k]} f(x g_\gp u(s)) \de g_{\gp} 
&= \int_X f(yu(s))\psi(y) \de \mu_X(y) +O\big(\gp^{-2k}\norm{f}_{C^1(X)}\big)\nonumber\\
&= \langle u(s).f,\psi\rangle_{L^2(\mu_X)} +O\big(\gp^{-2k}\norm{f}_{C^1(X)}\big).\label{eq:finalstep-2}
\end{align}
As the representation of $\theta_\gp(\SL_2(\Q_\gp))$ on $L_0^2(X)$ is $\tempered$-tempered, we have as in Lemma~\ref{lem:generic for one fct}
\begin{align}
 \langle u(s).f,\psi\rangle_{L^2(\mu_X)} &= \int_X f + O\big((1+|s|)^{-\star} \lev(f)^{\frac{3}{2}}\lev(\psi)^{\frac{3}{2}}\norm{f}_\infty \norm{\psi}_\infty\big)\nonumber\\
 &= \int_X f + O\big((1+|s|)^{-\star} \gp^{\star k}\norm{f}_\infty \big) \label{eq:finalstep-3}
\end{align}
using the above properties of $\psi$ and the assumption that $\lev(f) \leq 2 \gp^{k}$.
Combining \eqref{eq:finalstep-1}, \eqref{eq:finalstep-2}, and \eqref{eq:finalstep-3} we obtain
\begin{align*}
\Big| \int_X f - \int_{Y_\data}f\de \mu \Big|
\ll \big(|s|^\star \gp^{-Ak} + \gp^{-2 k} + (1+|s|)^{-\star} \gp^{\star k}\big) \lev(f) \norm{f}_{C^1(X)}.
\end{align*}
For a suitable choice of $s \in \Q_{\gp}$ and $A>0$ the middle term in the above right-hand side dominates. For $\gp$ sufficiently large, this proves the lemma.
\end{proof}

\begin{proof}[Proof of Proposition~\ref{prop:notweakerprop}]
The reader is advised to first recall the inductive scheme outlined in \S\ref{sec:strategy} that we will follow here.
Set $A = \max\{\ref{a:addinv-intro},3\ref{a:concludingstep}\}$ and $\delta = 1/(3A)$.
The following argument alternates between applying Proposition~\ref{prop:addinv-intro} and Theorem~\ref{thm:effgen-intro}.

\textsc{Initial Step:}
Starting from the obvious invariance under $H_\gp$ we find additional almost invariance.
Let $k>0$ be maximal with $\gp^k \leq \mcpl(Y_\data)^{1/\ref{a:addinv-intro}}$.
If $\gp^{k} < \gp^{\ref{a:addinv-intro}}\cpl(X)^{\ref{a:addinv-intro}}$, the proposition is trivially true. So we assume otherwise.
We take $M = H_\gp$ and note that $\mu$ is invariant, and in particular $\gp^{-\ref{a:addinv-intro} k}$-almost invariant, under $M[3k]$.
By Lemma~\ref{lem:Heffgen}, $M = H_\gp$ is $0$-generated (and, in particular, $k$-generated).
Overall, we are in position to apply Proposition~\ref{prop:addinv-intro}.
Thus, there exists a highest weight vector $\vpz \in \gfrak_\gp[0]$ of non-zero weight with $\vpz \mod \gp \not \in \hfrak_{\gp}[0] \mod \gp$ so that $\mu$ is $\gp^{-k/\ref{a:addinv-intro}}$-almost invariant under $\vpz$.

Let $\mathcal{N}$ be the list consisting of the nilpotents $\wpz_1,\ldots,\wpz_{\dim(\hfrak)}$ of pure non-zero weight $0$-generating $H_\gp$ (provided by Lemma~\ref{lem:Heffgen}) and the additional direction $\vpz$.
Now apply Theorem~\ref{thm:effgen-intro} with $\mathcal{N}$, with $k/\ref{a:addinv-intro}$ (instead of $k$) and with the above $\delta$.
As before, if \eqref{eq:lowerboundeffgen} is not satisfied, we conclude.
Else, we obtain a perturbation $\vpz_1,\ldots,\vpz_{\dim(\hfrak)}$ of $\wpz_1,\ldots,\wpz_{\dim(\hfrak)}$ and $\vpz_{\dim(\hfrak)+1}$ of $\vpz$ as well as a constant $\alpha_1 \gg 1$ so that the following hold:
\begin{itemize}
\item $\norm{\vpz_{i} - \wpz_i} < \gp^{-\alpha_1 k/\ref{a:addinv-intro} + \dim(\G)}$ for $1 \leq i \leq \dim(\hfrak)$.
\item $\norm{\vpz_{\dim(\hfrak)+1} - \vpz} < \gp^{-\alpha_1 k/\ref{a:addinv-intro} + \dim(\G)}$.
\item The vector $\vpz_{\dim(\hfrak)+1}$ is a highest weight vector of pure non-zero weight.
\item The vectors $\vpz_i$, $i\leq \dim(\hfrak)$, are of pure non-zero (possibly negative) weight.
\item If $\Mbf_1$ is the Zariski closure of the group generated by $\theta_{\gp}(\SL_2(\Q_\gp))$ and the one-parameter unipotent subgroups $\{\exp(t\vpz_i)\colon t \in \Q_\gp\}_{1 \leq i \leq \dim(\hfrak)+1}$, then $M_1 = \Mbf_1(\Q_\gp)$ is $\delta\alpha_1 k/\ref{a:addinv-intro}$-generated by nilpotents in $\gfrak_\gp[0]$ of pure non-zero weight.
Specifically, this list of nilpotents consists of vectors from $\{\vpz_1,\ldots,\vpz_{\dim(\hfrak)+1}\}$ and the upper and lower nilpotent $\zpz^+,\zpz^-$ from $\mathrm{D}\theta_\gp(\mathfrak{sl}_2(\Q_\gp))$.
\end{itemize}
For convenience, we set $k_1 = \lceil \delta \alpha_1 k/\ref{a:addinv-intro}\rceil$ so that $k_1 \geq c_1 k$ for some constant $c_1>0$ depending only on $N$ and so that $M_1$ is $k_1$-generated in the above manner.
By the first bullet point, $\mu$ is $\gp^{-2Ak_1}$-almost invariant under the new directions $\vpz_i$ (assuming that $k_1$ is sufficiently large).
By Lemma~\ref{lem:implicitfctthm}, this implies that $\mu$ is $\gp^{-Ak_1}$-almost invariant under the group $M[3k_1]$.
Note that we have used our choice of $\delta>0$ (and in particular our freedom to do so) here.

Also, notice that the Lie algebra of the group generated by the one-parameter groups $\{\exp(t\vpz_i)\colon t \in \Q_\gp\}_{1 \leq i \leq \dim(\hfrak)}$ agrees with $\hfrak_\gp$ modulo $\gp$.
Since $\vpz \mod \gp \not\in \hfrak_\gp[0] \mod \gp$, this shows that $\dim(\Mbf_1) \geq \dim(\Hbf)+1$ (or, more precisely, the dimension of $\Mbf_1$ is at least the dimension of $\Hbf$ plus the dimension of the non-trivial representation generated by the highest weight vector $\vpz_{\dim(\hfrak)+1}$).

\textsc{Inductive Step:}
Suppose by induction that for $i \in \{\dim(\hfrak)+1,\ldots,\dim(\gfrak)\}$ we are given $k_i \geq c_i k$ for some constant $c_i$ depending only on $N$ and nilpotents $\vpz_1,\ldots,\vpz_i\in \gfrak_\gp[0]$ of pure non-zero weight with the following properties:
\begin{itemize}
\item The Lie algebra of the group generated by $\{\exp(t\vpz_i)\colon t \in \Q_\gp\}_{1 \leq i \leq \dim(\hfrak)}$ agrees with the Lie algebra $\hfrak_\gp$ modulo $\gp$.
\item The vectors $\vpz_{\dim(\hfrak)},\ldots,\vpz_i$ are highest weight vectors that are linearly independent modulo $\gp$ and that are not contained in $\hfrak_\gp[0] \mod \gp$.
\item Let $\Mbf_i$ the Zariski closure of the group generated by the principal $\SL_2$ and the one-parameter unipotent subgroups  $\{\exp(t\vpz_i)\colon t \in \Q_\gp\}$.
Then $M_i= \Mbf_i(\Q_\gp)$ is $k_i$-generated by a list of nilpotents so that each element is contained in $\{\vpz_1,\ldots,\vpz_{i}\}$ or is equal to $\zpz^+$ or $\zpz^-$.
\item The measure $\mu$ is $\gp^{-Ak_i}$-almost invariant under $M_i[3k_i]$.
\end{itemize}
Note that the assumed list $(\vpz_1,\ldots,\vpz_i)$ might not overlap at all with the list attained at the initial step or the previous step in the induction.
Also, note that $M_i$ is not guaranteed to contain the invariance group $H_{\gp}$.

If $\Mbf_i$ is not a proper subgroup, Lemma~\ref{lem:concludingstep} implies the proposition (and the induction is aborted).
So assume that $\Mbf_i$ is a proper subgroup.
We apply Proposition~\ref{prop:addinv-intro} for the group $M_i$ and for $k_i$.
If $\gp^{k_i}$ is too small, we conclude.
Otherwise, we obtain a highest pure non-zero weight vector $\vpz \in \gfrak_\gp[0]$ with $\vpz \mod \gp \not\in \mfrak_i[0] \mod \gp$ under which $\mu$ is $\gp^{-k_i/\ref{a:addinv-intro}}$-almost invariant.

Now apply Theorem~\ref{thm:effgen-intro} to the list $(\vpz_1,\ldots,\vpz_i,\vpz)$, $\delta = \frac{1}{3 \ref{a:addinv-intro}}$ (as before), and $k = k_i/\ref{a:addinv-intro}$.
Again, if $k$ is too small, we conclude.
Otherwise, we find a new list $(\vpz_1',\ldots,\vpz_{i+1}')\in \gfrak_\gp[0]^{i+1}$ of nilpotent elements and $\alpha_{i} \in (\ref{k:effgen-intro}\delta^{\dim(\gfrak)},1]$ such that the inductive assumption is satisfied for the nilpotents $\vpz_1',\ldots,\vpz_{i+1}'$ and for $k_{i+1} = \lceil \delta \alpha_i k_i/\ref{a:addinv-intro}\rceil$.

Finally, we remark that the above induction automatically stops after $\dim(\gfrak)-\dim(\hfrak)$ many steps. This proves the proposition.
\end{proof}

\begin{proof}[Proof of Theorem~\ref{thm:main equi} for $F=\Q$]
Recall the setting:
\begin{itemize}
    \item $\G$ is a $\Q$-anisotropic simply connected semisimple group and $\rho: \G \to \SL_N$ is a homomorphism with central kernel defined over $\Q$.
    \item $\data = (\H,\iota,g_\data)$ is semisimple simply connected data over $\Q$ consistent with the pair $(\G,\rho)$.
    \item $X = [\rho(\G(\A))]$ and $Y_\data = [\iota(\H(\A))g_\data] \subset X$.
\end{itemize}

To simplify notation, we will first reduce to the case $g_\data = \id$.
By Theorem~\ref{thm:diameter} (based on \cite{MSGT-diameter} and Proposition~\ref{prop:smallvectors}) using a good prime $\gp_0$ for $X$ with $\gp_0 \ll \log(\cpl(X))^2$ there exists $\gamma \in \G(\Q)$  so that
\begin{align*}
\norm{\rho(\gamma)g_\data}_\infty \ll 1,\
\norm{\rho(\gamma)g_\data}_{\gep} = 1 \text{ for all } \gep \neq \gp_0,\
\text{ and }\norm{\rho(\gamma)g_\data}_{\gp_0} \ll \cpl(X)^\star.
\end{align*}
Let $g = \rho(\gamma) g_\data$. Then $Y_\data$ agrees with $Y_{(\H,\iota^\gamma,g)}$ where $\iota^\gamma(\cdot) = \rho(\gamma) \iota(\cdot) \rho(\gamma)^{-1}$.
By the above bounds on $g$, the theorem follows if we establish it for $(\H,\iota^\gamma,\id)$.
Hence, we assume without loss of generality that $g_\data = \id$.
In the absence of the translate, we may use the complexity of an orbit and the height of the associated group interchangably.

Recall e.g.~from \cite[Lemma 8.6]{AW-realsemisimple} that there exists $\consta\label{a:heightnormal}>0$ depending only on $N$ so that $\height(\Lbf) \ll \height(\G)^{\ref{a:heightnormal}}$ for any normal subgroup $\Lbf \lhd \G$.
Set $A = \max\{2\ref{a:notweakerprop},\ref{a:heightnormal}\}$.
We prove the following claim by induction on the dimension of $\G$, which in turn clearly implies the theorem.

\begin{claim*}
Assuming $g_\data = \id$ we have for all $C^1$-functions $f$
\begin{align}\label{eq:proofmain-inductiveclaim}
\Big| \int_{X} f - \int_{Y_\data} f\Big|
\ll \frac{\cpl(X)^{A}}{\mcpl(Y_\data)^{1/A}} \lev(f)\norm{f}_{C^1(X)}.
\end{align}
\end{claim*}

If $\iota(\H)$ is contained in a proper normal subgroup $\Lbf\lhd \rho(\G)$, then
\begin{align*}
\mcpl(Y_\data) \leq \cpl([\Lbf(\A)]) = \height(\Lbf) \ll \cpl(X)^{\ref{a:heightnormal}}.
\end{align*}
Thus, the claim in \eqref{eq:proofmain-inductiveclaim} is trivial in this case.
Assume from now that $\iota(\H)$ is not contained in a proper normal subgroup $\Lbf\lhd \rho(\G)$.

We prove \eqref{eq:proofmain-inductiveclaim} by induction on the dimension of $\G$ and so may assume that \eqref{eq:proofmain-inductiveclaim} holds for $Y_{\data}$ considered inside intermediate semisimple orbits.
Let $\mathcal{B}$ be the collection of connected semisimple proper $\Q$-subgroups $\Lbf < \rho(\G)$ with $\iota(\H) < \Lbf$ and with
\begin{align*}
\height(\Lbf)^{3A^2} <  \height(\Lbf')
\end{align*}
for all semisimple $\Q$-subgroups $\Lbf'\lneq \Lbf$ containing $\iota(\H)$.
For convenience, we define the relative minimal complexity
\begin{align}\label{eq:relative mincpl}
\mcpl_\Lbf(Y_\data) = \min\big\{\cpl([\Lbf'(\A)]): 
\iota(\H)\subseteq \Lbf' \subsetneq \Lbf \text{ semisimple}\big\}.
\end{align}
Then $\Lbf \in \mathcal{B}$ if and only if $\cpl([\Lbf(\A)])^{3A} < \mcpl_\Lbf(Y_\data)^{1/A}$.

Let $\Lbf_0 > \iota(\H)$ be a connected semisimple $\Q$-subgroup of $\rho(\G)$ so that 
\begin{align*}
\cpl([\Lbf_0(\A)]) = \mcpl(Y_\data).
\end{align*}
We construct a subgroup of $\Lbf_0$ in the collection $\mathcal{B}$ as follows:
If $\Lbf_0 \in \mathcal{B}$, we are done.
Otherwise, there exists $\Lbf_1 \lneq \Lbf_0$ with $\height(\Lbf_1)\leq 
\height(\Lbf_0)^{3A^2}$.
Continuing inductively, we find a sequence of subgroups $\Lbf_0 \gneq \Lbf_1 \gneq \ldots$ which has to terminate after at most $\dim(\G)$ steps.
Thus, there exists $j \leq \dim(\G)$ such that $\Lbf_j \in \mathcal{B}$ and
\begin{align*}
\height(\Lbf_j) \leq \height(\Lbf_0)^{(3A^2)^j}.
\end{align*}
In particular,
\begin{align}\label{eq:goodintermediategroup}
\mcpl(Y_\data) \leq \cpl([\Lbf_j(\A)]) \leq \mcpl(Y_\data)^{(3A^2)^{\dim(\G)}}
\end{align}
For simplicity, we set $\Lbf := \Lbf_j$.
Note that $\Lbf = \iota(\H)$ is entirely possible and not ruled out here.
Write $\tilde{\Lbf}$ for the simply connected cover of $\Lbf$ and $\iota_\Lbf\colon \tilde{\Lbf} \to \Lbf$ for the covering map.
Set $\data' = (\tilde{\Lbf},\iota_\Lbf,\id)$.

We wish to show equidistribution with a rate for $Y_{\data'}$ in $X$ and for $Y_\data$ in $Y_{\data'}$.
For the former, note that there is by Proposition~\ref{prop:splitting place} a good prime $\gp$ for the homogeneous set $Y_{\data'}$ and for $X$ with
\begin{align*}
\gp \ll \max\{\log(\cpl(X)),\log(\cpl(Y_{\data'}))\}^2 \ll \cpl(X) \log(\mcpl(Y_\data))^2
\end{align*}
where \eqref{eq:goodintermediategroup} was used for the second inequality.
Thus, by Proposition~\ref{prop:notweakerprop} (noting also that $\Lbf$ is not contained in a proper normal subgroup of $\G$ because $\iota(\H)$ is not) we have for all $C^1$-functions $f$ on $X$
\begin{align}\nonumber
\Big| \int_{X}f -\int_{Y_{\data'}}f \Big| 
&\ll \frac{\cpl(X)^{2\ref{a:notweakerprop}}\log(\mcpl(Y_\data))^{2\ref{a:notweakerprop}}}{\mcpl(Y_{\data'})^{1/\ref{a:notweakerprop}}} \lev(f) \norm{f}_{C^1(X)}\\
&\ll \frac{\cpl(X)^{2\ref{a:notweakerprop}}}{\mcpl(Y_{\data})^{1/(2\ref{a:notweakerprop})}} \lev(f) \norm{f}_{C^1(X)}\label{eq:equiofintermediateinbig}
\end{align}
using also $\mcpl(Y_{\data'}) \geq \mcpl(Y_{\data})$.

By the inductive hypothesis, we have for all $f \in C^1(X)$
\begin{align*}
\Big| \int_{Y_{\data'}} f - \int_{Y_{\data}} f\Big| 
\ll \frac{\cpl(Y_{\data'})^{A}}{\mcpl_\Lbf(Y_{\data})^{1/A}} \lev(f|_{Y_{\data'}}) \norm{f|_{Y_{\data'}}}_{C^1(Y_{\data'})}
\end{align*}
Note that since $\Lbf \in \mathcal{B}$
\begin{align*}
\mcpl_\Lbf(Y_{\data})^{1/A} \geq \cpl(Y_{\data'})^{3A}.
\end{align*}
Also, any $C^1$-function $f$ on $X$ is invariant under $\prod_\gep G_\gep[\ord_\gep(\lev(f))]$ and hence the restriction $f|_{Y_{\data'}}$ is also invariant under $\prod_\gep L_\gep[\ord_\gep(\lev(f))]$.
Moreover, we have $\norm{f|_{Y_{\data'}}}_{C^1(Y_{\data'})} \ll \norm{f}_{C^1(X)}$ where the implicit constant accounts for possibly incompatible choices of orthonormal bases of the Lie algebras $\lfrak_\infty < \gfrak_\infty$.
Overall, we obtain
\begin{align}
\Big| \int_{Y_{\data'}} f - \int_{Y_{\data}} f\Big| 
\ll \frac{1}{\mcpl(Y_\data)^{2A}} \lev(f) \norm{f}_{C^1(X)}.
\label{eq:equiofsmallinintermediate}
\end{align}
Combining \eqref{eq:equiofintermediateinbig} and \eqref{eq:equiofsmallinintermediate} proves the claim in \eqref{eq:proofmain-inductiveclaim}. 
As discussed earlier, the theorem follows.
\end{proof}

\begin{proof}[Proof of Theorem~\ref{thm:asinEMV} for $F=\Q$]
As in the above proof of Theorem~\ref{thm:main equi} for $F=\Q$, we may assume $g_\data = \id$.
We set $A= 2\ref{A:main}^2$ (where $\ref{A:main}>1$ is as in Theorem~\ref{thm:main equi} for $F=\Q$) and $\delta= A^{-\dim(\G)}$.
In view of the assumptions of the theorem, we may assume $B^{\delta/(2\ref{A:main})} \geq \cpl(X)^{\ref{A:main}}$ or equivalently $B \geq \cpl(X)^{A/\delta}$.

For any $\delta'>0$ let $\mathcal{C}_{\delta'}$ be the collection of connected semisimple $\Q$-subgroups $\Lbf <\rho(\G)$ containing $\iota(\Hbf)$ and 
\begin{align*}
B^{\delta'} < \mcpl_{\Lbf}(Y_{\data}),
\end{align*}
where we use the relative minimal complexity $\mcpl_{\Lbf}(\cdot)$ introduced in \eqref{eq:relative mincpl} as a shorthand.

If $\rho(\G) \in \mathcal{C}_\delta$, we have $\mcpl(Y_\data) \geq B^\delta$ (by definition of $\mathcal{C}_\delta$) and so
\begin{align*}
\frac{\cpl(X)^{\ref{A:main}}}{\mcpl(Y_\data)^{1/\ref{A:main}}}
\leq \frac{B^{\delta/(2\ref{A:main})}}{B^{\delta/\ref{A:main}}} \leq B^{-\delta/(2\ref{A:main})}
\end{align*}
so that the theorem follows from Theorem~\ref{thm:main equi} for $F=\Q$ with $\ref{a:main2} \geq 3 \ref{A:main}/\delta$ (accounting also for implicit constants and using again that $B$ is sufficiently large).

So assume now that $\rho(\G)\not \in \mathcal{C}_\delta$.
Thus, there exists a proper semisimple subgroup $\Lbf_1<\rho(\G)$ with $\Lbf_1 > \iota(\Hbf)$ and $\cpl([\Lbf_1(\A)]) \leq B^{\delta}\leq B$ (where the exponent has worsened by a factor of $A$).
Let $\tilde{\Lbf}_1$ be the simply connected cover of $\Lbf_1$ and let $\iota_{\Lbf_1}\colon \tilde{\Lbf}_1 \to \Lbf_1$ be the covering map.

If $\Lbf_1 \in \mathcal{C}_{A \delta}$, we apply Theorem~\ref{thm:main equi} for $F=\Q$ with the pair $(\tilde{\Lbf}_1,\iota_{\Lbf_1})$ instead of $(\G,\rho)$.
Thus, we have
\begin{align*}
\frac{\cpl([\Lbf_1(\A)])^{\ref{A:main}}}{\mcpl_{\Lbf_1}(Y_\data)^{1/\ref{A:main}}}
\leq \frac{B^{\ref{A:main}\delta}}{B^{A\delta/\ref{A:main}}}
\leq B^{-\ref{A:main}\delta}
\end{align*}
and the theorem follows from Theorem~\ref{thm:main equi} for $F=\Q$ with $\ref{a:main2}\geq 2/(\ref{A:main}\delta)$.

If $\Lbf_1 \not \in \mathcal{C}_{A\delta}$, there exists a proper semisimple subgroup $\Lbf_2 < \Lbf_1$ with
\begin{align*}
 \cpl([\Lbf_2(\A)]) \leq B^{A\delta} \leq B
\end{align*}
containing $\iota(\Hbf)$.
We continue like this by induction constructing a sequence of subgroups $\Lbf_1 \gneq \Lbf_2 \gneq\Lbf_3 \gneq \ldots$ containing $\iota(\Hbf)$ with $\cpl([\Lbf_j(\A)]) \leq B^{A^{j-1}\delta}\leq B$ until we find a subgroup $\Lbf_j \in \mathcal{C}_{A^j \delta}$ with $j < \dim(\G)$.
Let $\tilde{\Lbf}_j$ be the simply connected cover of $\Lbf_j$ and let $\iota_{\Lbf_j}\colon \tilde{\Lbf}_j \to \Lbf_j$ be the covering map.
Applying Theorem~\ref{thm:main equi} for $F=\Q$ with the pair $(\tilde{\Lbf}_j,\iota_{\Lbf_j})$ we obtain the rate
\begin{align*}
\frac{\cpl([\Lbf_j(\A)])^{\ref{A:main}}}{\mcpl_{\Lbf_j}(Y_\data)^{1/\ref{A:main}}}
\leq \frac{B^{A^{j-1}\delta\ref{A:main}}}{B^{A^j \delta/\ref{A:main}}}
= \frac{B^{A^{j-1}\delta\ref{A:main}}}{B^{2A^{j-1} \delta\ref{A:main}}} = B^{-A^{j-1}\delta\ref{A:main}}.
\end{align*}
This implies the theorem.
\end{proof}

 \subsection{The number field case}\label{sec:numberfield case}

We finally prove Theorems~\ref{thm:main equi} and \ref{thm:asinEMV} as stated by reducing them to the already proven case $F=\Q$ using restriction of scalars.
So let $F$ be a number field and write $D_F$ for the absolute value of its discriminant $\disc(F)$ and $d = [F:\Q]$ for the degree.

Recall first the setup:
\begin{itemize}
    \item $\G$ is an $F$-anisotropic simply connected semisimple group and $\rho: \G \to \SL_N$ is a homomorphism with finite central kernel defined over $F$.
    \item $\data = (\H,\iota,g_\data)$ is semisimple simply connected data over $F$ that is consistent with $(\G,\rho)$, i.e.~$\iota\colon \H \to \rho(\G)$ is a homomorphism with central kernel and $g_\data \in \rho(\G(\A_F))$.
\end{itemize}

We introduce some notation pertaining to restriction of scalars.
Set for simplicity $N' = N \, [F:\Q] = N \, d$.
By Minkowski's second theorem, we may fix linearly independent vectors $\alpha_1,\ldots,\alpha_{d} \in \mathcal{O}_F$ which span a sublattice of index $O_{d}(1)$ and which have norm $\ll_{d} D_F^{\frac{1}{2}}$ (under the complete embedding $F \to \R^{d}$). Representing multiplication by $F$ in this basis we obtain homomorphisms
$\vartheta\colon F \to \Mat_{d}(\Q)$ and $\vartheta\colon\Mat_N(F) \to \Mat_{N'}(\Q)$.
These will serve as our explicit realization of restriction of scalars.

For an $F$-subspace $V \subset \mathfrak{sl}_N$ write $V' = \vartheta(V)$ for its restriction of scalars.
We write $\height_F(V)$ for the height of $V$ and $\height_\Q(V')$ for the height of $V'$ where the indices aim to emphasize the field of definition.

\begin{lemma}\label{lem:heightafterrestrofscalars}
For any $F$-subspace $V \subset \mathfrak{sl}_N$ we have
\begin{align}\label{eq:ht restrofscalars}
D_F^{-\star} \height(V')
\ll \height(V)
\ll D_F^{\star} \height(V').
\end{align}
\end{lemma}

\begin{proof}
This is certainly well-known; we include a proof for lack of explicit reference.
We record first a few elementary properties of $\vartheta$.
For any rational prime $\gep$ and all $\vpz \in \Mat_N(F)$ we have 
\begin{align}\label{eq:finitenormsafterrestrofscalars}
\norm{\vartheta(\vpz)}_\gep \ll \max_{w\mid \ell}\norm{\vpz}_w \ll \norm{\vartheta(\vpz)}_\gep
\end{align}
where the implicit constant is $1$ at all but $O_{d}(1)$-many places.
Moreover,
\begin{align}\label{eq:Euclnormafterrestrofscalars}
D_F^{-\star}\norm{\vartheta(\vpz)}_\infty^2 
\ll \sum_{w\mid \infty}\norm{\vpz}_w^2 \ll 
D_F^{\star}\norm{\vartheta(\vpz)}_\infty^2.
\end{align}

We turn to proving \eqref{eq:ht restrofscalars} beginning with the upper bound.
Apply Siegel's lemma to obtain a basis $(\vpz_i)_{i \leq d\dim_F(V)}$ of $V'$ consisting of integer vectors such that $\norm{\vpz_i}_\infty \ll \height(V')$.
Note that $V'$ is invariant under the $F$-module structure, i.e.~under $\vartheta(F)$.
Pick a subset $\mathcal{I} = \{i_1,\ldots, i_{\dim_F(V)}\}$ of indices $i$ such that $(\vpz_i)_{i \in \mathcal{I}}$ are linearly independent over $F$ and let $\upz_i \in V$ with $\vartheta(\upz_i)= \vpz_i$.
The vectors $\upz_i$, $i\in \mathcal{I}$, span $V$ and hence by \eqref{eq:finitenormsafterrestrofscalars} and \eqref{eq:Euclnormafterrestrofscalars}
\begin{align*}
\height(V) = \prod_{w} \norm{\upz_{i_1}\wedge \ldots \wedge \upz_{i_{\dim_F(V)}}}_w \ll D_F^\star \height(V')^\star.
\end{align*}

To prove the lower bound in \eqref{eq:ht restrofscalars}, we first assume $\dim(V)=1$ and let $\vpz \in V$ be an integral vector with $\prod_{w \nmid \infty} \norm{\vpz}_w \ll D_F$.
We compute the covolume of the lattice $\mathcal{O}_F.\vpz$ with respect to the Euclidean norm $\sum_{w\mid \infty}\norm{\cdot}_w^2$. Explicitly, it is given by $|\det(\langle \alpha_i\vpz,\alpha_j \vpz\rangle)|$ where $\langle \upz,\upz'\rangle = \sum_{w}\langle \upz,\upz'\rangle_w$.
Expanding this expression
\begin{align*}
\det(\langle \alpha_i\vpz,\alpha_j \vpz\rangle)
&= \sum_{\tau} (-1)^{\mathrm{sgn}(\tau)} \prod_{i}\langle \alpha_i\vpz,\alpha_{\tau(i)} \vpz\rangle \\
&= \sum_{\tau} (-1)^{\mathrm{sgn}(\tau)} \prod_{i}
\sum_{\sigma} \prescript{\sigma}{}{\alpha_i} \prescript{\sigma}{}{\alpha_{\tau(i)}}
\langle \prescript{\sigma}{}{\vpz},\prescript{\sigma}{}{\vpz}\rangle
\end{align*}
where $\tau$ runs over all permutations of $\{1,\ldots,d\}$ and $\sigma$ runs over all embeddings $\sigma: F\to \C$. We enumerate these embeddings by $\sigma_1,\ldots$ and obtain 
\begin{align*}
\det(\langle \alpha_i\vpz,\alpha_j \vpz\rangle)
= \sum_{\tau} (-1)^{\mathrm{sgn}(\tau)} \sum_{\tau'}
\prod_{i}  \prescript{\sigma_{\tau'(i)}}{}{\alpha_i} \prescript{\sigma_{\tau'(i)}}{}{\alpha_{\tau(i)}}
\langle \prescript{\sigma_{\tau'(i)}}{}{\vpz},\prescript{\sigma_{\tau'(i)}}{}{\vpz}\rangle
\end{align*}
where $\tau'$ runs over all self-maps of $\{1,\ldots,d\}$. The expression 
\begin{align*}
\sum_{\tau} (-1)^{\mathrm{sgn}(\tau)} \prod_i\prescript{\sigma_{\tau'(i)}}{}{\alpha_{\tau(i)}}
\end{align*}
vanishes unless $\tau'$ is bijective in which case its absolute value is $\ll D_F$.
So we may restrict the summation to permutations $\tau'$ to obtain 
\begin{align*}
|\det(\langle \alpha_i\vpz,\alpha_j \vpz\rangle)|
&= \Big|\sum_{\tau} (-1)^{\mathrm{sgn}(\tau)} \sum_{\tau'}
\prod_{i}  \prescript{\sigma_{i}}{}{\alpha_{\tau'(i)}} \prescript{\sigma_{i}}{}{\alpha_{\tau\circ \tau'(i)}}
\langle \prescript{\sigma_{i}}{}{\vpz},\prescript{\sigma_{i}}{}{\vpz}\rangle\Big| \\
&\ll D_F^3 \height_F(V).
\end{align*}
By \eqref{eq:Euclnormafterrestrofscalars}, this implies \eqref{eq:ht restrofscalars} when $\dim(V) =1$.

For $\dim(V)>1$, apply a version of Siegel's lemma over $F$ --- see e.g.~\cite{BombieriVaaler} --- to find a basis $\upz_i$ of $V$ of integral vectors with $\prod_i \height_F(\upz_i) \ll D_F^\star\height_F(V)$. 
Here, $\height_F(\upz_i) = \prod_{w} \norm{\upz_i}_w$.
Applying the above calculation for dimension $1$, we find that the covolume of $\mathcal{O}_F.\upz_i$ with respect to either Euclidean norm in \eqref{eq:Euclnormafterrestrofscalars} is $\ll D_F^\star \height(\upz_i)$. Taking the product, this yields the lower bound in \eqref{eq:ht restrofscalars} and hence the lemma.
\end{proof}

Denote by $\G'= \Res_{F/\Q}(\G)$ the restriction of scalars of $\G$ and by 
\begin{align*}
\rho': \G' \to \SL_{N'}
\end{align*}
the homomorphism defined over $\Q$ which is induced by $\rho$.
Note that $\G'$ is a simply connected semisimple $\Q$-group that is $\Q$-anisotropic. Also, $\rho'$ has finite kernel.
The analogous notation is used for $(\Hbf,\iota)$ and we write $\Lbf'= \Res_{F/\Q}(\Lbf)<\mathfrak{sl}_{N'}$ for any $F$-subgroup $\Lbf < \rho(\G)$ as well.
To clarify the proof, we denote by $g' \in \SL_{N'}(\A_\Q)$ the element corresponding to $g \in \SL_N(\A_F)$.
Lastly, we use the notation $[B]$ for the image in $\SL_N(F)\backslash \SL_N(\A_F)$ resp.~$\SL_{N'}(\Q)\backslash \SL_{N'}(\A_\Q)$ of a subset $B$ of $\SL_N(\A_F)$ resp.~$\SL_{N'}(\A_\Q)$.
We set $X' = [\rho'(\G'(\A_\Q))]$ and $Y_{\data}' = [\iota'(\Hbf'(\A_\Q))]$.

\begin{lemma}\label{lem:cplafterres}
We have
\begin{align}\label{eq:Xrestrscalars}
D_F^{-\star} \cpl(X')^\star 
\ll \cpl(X) 
\ll D_F^\star \cpl(X')^\star.
\end{align}
Moreover, for any semisimple $F$-subgroup $\Lbf < \rho(\G)$ and let $g \in \rho(\G(\A_F))$
\begin{align}\label{eq:intermrestrscalars}
\cpl([\Lbf(\A_F)g]) \ll D_F^\star\, \cpl(X')^\star \cpl([\Lbf'(\A_F)g']).
\end{align}
\end{lemma}

\begin{proof}
The estimates in \eqref{eq:Xrestrscalars} are a direct consequence of Lemma~\ref{lem:heightafterrestrofscalars}.

For \eqref{eq:intermrestrscalars}, notice first that, in view of Theorem~\ref{thm:diameter}, there exists $\gamma' \in \G'(\Q)$ with $(\rho'(\gamma')g')_\gep \in \SL_{N'}(\Z_\gep)$ for $\gep \neq \gp$, $\norm{(\rho'(\gamma')g')_\infty}\ll 1$, and $\norm{(\rho'(\gamma')g')_\gp} \ll \gp^\star \cpl(X')^\star$.
Here, we may take $\gp$ to be a good prime in the sense of Proposition~\ref{prop:splitting place} for $X'$ only and so $\gp \ll (\log\vol(X'))^2 \ll (\log\cpl(X'))^2$ (see Proposition~\ref{prop: vol complexity'}).
The corresponding bounds (with an additional polynomial dependence on $D_F$) also hold for $g$ and the element $\gamma \in \G(F)$ corresponding to $\gamma'$ (see \eqref{eq:finitenormsafterrestrofscalars} and \eqref{eq:Euclnormafterrestrofscalars}).
In particular, $\tilde{\Lbf} = \gamma \Lbf \gamma^{-1}$ satisfies
\begin{align*}
\cpl([\Lbf(\A_F)g]) \ll \cpl([\tilde{\Lbf}(\A_F)])\cpl(X')^\star = \height_{\Q}(\tilde{\Lbf})\cpl(X')^\star D_F^\star
\end{align*}
and
\begin{align*}
\cpl([\Lbf'(\A_F)g']) \gg \height_F(\tilde{\Lbf}') \cpl(X')^{-\star}.
\end{align*}
This together with Lemma~\ref{lem:heightafterrestrofscalars} applied to the Lie algebra of $\tilde{\Lbf}$ yields the lemma.
\end{proof}

\begin{lemma}\label{lem: Q subgroup of res containg H}
Let $\Mbf>\iota(\H)$ be a connected semisimple $F$-subgroup of $\rho(\G)$. 
Let $\Lbf< \Mbf'$ be a $\Q$-subgroup which contains $\iota'(\H')$. Assume further that 
\begin{enumerate}[label=\textnormal{(\theenumi)}]
\item $\iota(\H)$ projects non-trivially onto all $F$-almost simple factors of $\bf M$, and   
\item $\Lbf(\Q)\subset \Mbf'(\Q) = \Mbf(F)$ is Zariski dense in the $F$-group $\Mbf$. 
\end{enumerate}
Then $\Lbf=\Mbf'$. 
\end{lemma}

\begin{proof}
Let $\lfrak$ resp.~$\mfrak'$ be the Lie algebra of $\Lbf$ resp.~$\Mbf'$ over $\Q$.
We will show $\mathfrak l=\mfrak'$; this establishes the claim since $\bf M$ is connected. 

For this, we let $\sfrak'=\bigcap_{\alpha \in F} \alpha. \lfrak$ where we use that $\gfrak'$ is naturally an $F$-module.
The Lie algebra $\sfrak'$ is invariant under scalars in $F$ and so $\sfrak' = \Res_{F/\Q}(\sfrak)$ for some $F$-subalgebra $\sfrak$ of $\mfrak$.
By definition, we have $[\mathfrak l,\sfrak']\subset \sfrak'$ and in particular 
\begin{align*}
\Ad(\Lbf(\Q))\sfrak(F) = \Ad(\Lbf(\Q))\sfrak'(\Q) \subset \sfrak'(\Q) =\sfrak(F)   
\end{align*}
when viewing ${\bf L}(\Q)$ as a subgroup of $\Mbf(F)$.
In view of our assumption~(2), $\sfrak$ is a Lie ideal of $\mfrak$.
Notice that $\sfrak$ contains the Lie algebra of $\iota(\Hbf)$ and hence $\sfrak = \mfrak$ by assumption (1).
We conclude that $\mfrak'= \sfrak' \subset \lfrak \subset \mfrak'$ and the lemma follows.
\end{proof}

\begin{proof}[Proof of Theorem~\ref{thm:main equi} over general number fields]
The theorem follows from its version over $\Q$ (proven earlier) once we can show that
\begin{align}\label{eq:mcplvsmcpl'}
\mcpl(Y_\data) \ll D_F^{\star} \cpl(X)^{\star} \mcpl(Y_\data')^\star.
\end{align}
In view of \eqref{eq:intermrestrscalars}, problems may arise precisely from subgroups $\Lbf \lneq \rho'(\G')$ with $\iota'(\Hbf') < \Lbf$ that do not arise through restriction of scalars; here, we wish to use Lemma~\ref{lem: Q subgroup of res containg H}.

Let $\mcpl_{\mathrm{erg}}(Y_\data')$ be the minimum over the complexities $\cpl([\Lbf(\A)g'])$ where $\Lbf$ runs over all connected semisimple $\Q$-subgroups $\Lbf \lneq \rho'(\G')$ with $\iota'(\Hbf') < \Lbf$ \emph{and} for which $\iota'(\Hbf')$ is not contained in a proper factor of $\Lbf$.
Clearly, $\mcpl_{\mathrm{erg}}(Y_\data') \geq \mcpl(Y_\data')$. We claim that conversely
\begin{align}\label{eq:ergmcpl}
\mcpl_{\mathrm{erg}}(Y_\data') \ll \cpl(X')^\star \mcpl(Y_\data')^\star.
\end{align}
This implies \eqref{eq:mcplvsmcpl'} (and hence the theorem) since, in view of Lemmas~\ref{lem: Q subgroup of res containg H} and \ref{lem:cplafterres}, we have
\begin{align*}
\mcpl(Y_\data) \ll |\disc(F)|^{\star} \cpl(X)^{\star} \mcpl_{\mathrm{erg}}(Y_\data')^\star.
\end{align*}

It remains to show \eqref{eq:ergmcpl}. 
By the same argument as in Lemma~\ref{lem:cplafterres} (relying on Theorem~\ref{thm:diameter}) we may assume that $g'_\gep \in \SL_{N'}(\Z_\gep)$ for $\gep \neq \gp$, $\norm{g'_\infty}\ll 1$, and $\norm{g'_\gp} \ll \cpl(X')^\star$ where $\gp \ll (\log \cpl(X))^2$.
In particular, for any connected semisimple $\Q$-subgroup $\Lbf < \rho'(\G')$
\begin{align*}
\cpl(X')^{-\star} \height_{\Q}(\Lbf) \ll \cpl([\Lbf(\A)g']) \ll \cpl(X')^\star \height_{\Q}(\Lbf).
\end{align*}
Factors of a semisimple $\Q$-subgroup $\Lbf< \rho'(\G')$ have height controlled polynomially by the height of $\Lbf$ (see e.g.~\cite[Lemma 8.6]{AW-realsemisimple}). Together, these estimates establish \eqref{eq:ergmcpl} and hence the theorem.
\end{proof}

\begin{proof}[Proof of Theorem~\ref{thm:asinEMV}]
As in the above proof of Theorem~\ref{thm:main equi}, the theorem can be reduced to the version over $\Q$ that we have already established.
Alternatively, it may be deduced from Theorem~\ref{thm:main equi} directly as was done earlier for $F=\Q$.
\end{proof}

\newpage

\begin{appendix}
\section{Volume and arithmetic complexity}\label{sec:volvscpl}

In this appendix we establish, in particular, the volume and complexity comparison of Proposition~\ref{prop: vol complexity'}.
In fact, we will establish a slightly finer result not assuming that the ambient homogeneous space is compact.

For the following we fix
\begin{itemize}
    \item a semisimple simply connected group $\Hbf$ defined over $\Q$,
    \item a homomorphism $\iota: \Hbf \to \SL_N$ with central kernel, and
    \item an element $g \in \SL_N(\A)$.
\end{itemize}
Set $Y = [\iota(\Hbf(\A))g]$ and $H = g^{-1}\iota(\Hbf(\A))g$.

\subsubsection*{The notion of volume}
Let us first recall the definition of a volume from \cite{EMMV} which was already introduced in \S\ref{sec:volumeoverQ}.
Let $\Omega \subset \SL_N(\A)$ be an open neighborhood of the identity with compact closure
and define the volume (with respect to $\Omega$) to be
\begin{align}\label{eq:defvolume}
\vol(Y_\data) = \vol_{\Omega}(Y) = m_H(\Omega \cap H)^{-1}
\end{align}
where $m_H$ is the Haar measure on $H$ that descends to the invariant probability measure on $Y$.
A different choice $\Omega$ of an open neighborhood of the identity yields a comparable notion of volume --- see the discussion in \cite[\S2.3]{EMMV}.

\begin{remark}
Instead of using the group $H$ one might also use the full stabilizer group of the orbit $Y$. The full stabilizer group is $H' = g^{-1}\iota(\Hbf(\A))\Nbf(F)g$ \cite[Lemma 2.2]{EMMV} where $\Nbf$ is the normalizer of $\Hbf$.
In particular, it is often contains $H$ as an infinite index subgroup.
The volume defined using $H'$ turns out to be comparable to the volume defined using $H$ --- see \cite[\S5.12]{EMMV} which relies on a deep result of Borel and Prasad \cite{BorPr-Finitness}.
\end{remark}

In the following, we assume that $\Omega$ is of the form 
\begin{align*}
\Omega = \Omega_\infty \times \textstyle{\prod}_p \SL_N(\Z_p)
\end{align*}
where $\Omega_\infty = \exp(\Xi_\infty) \subset \SL_N(\R)$ for a bounded symmetric open neighborhood $\Xi_\infty \subset\mathfrak{sl}_N(\R)$ of $0$ on which $\exp$ is a diffeomorphism.
In practice, one needs to assume that $\Xi_\infty$ is sufficiently small (depending only on $N$) --- see the discussion around \cite[(5.3),(5.5)]{EMMV}.

\subsubsection*{The notion of complexity}

For a $\Q$-subgroup $\Lbf< \SL_N$ with Lie algebra $\lfrak$ we let
\begin{align*}
\vpz_\Lbf \in \bigwedge^{\dim(\lfrak)} \lfrak(\Q) 
\subset \bigwedge^{\dim(\lfrak)} \mathfrak{sl}_N(\Q)
\end{align*}
be a non-zero vector. The height of $\Lbf$ is then defined to be $\height(\Lbf) = \prod_{\gplace \in \Sigma} \norm{\vpz_{\Lbf}}_{\gplace}$ where $\Sigma$ is the set of places of $\Q$ (see also \S\ref{sec:setup}).
As in the introduction, the complexity of the orbit $Y$ is given by
\begin{align*}
\cpl(Y) = \prod_{\gplace\in\Sigma} \norm{g_\gplace^{-1}.\vpz_{\iota(\Hbf)}}_\gplace.
\end{align*}

\subsubsection*{The notion of minimal height}
Define
\begin{align*}
\minht(Y) = \min_{[g]\in Y} \max_{\vpz \in \Q^n \setminus \{0\}} \content(g^{-1}.\vpz)^{-1}
\end{align*}
where $\content(g^{-1}.\vpz) = \prod_{w \in \Sigma} \norm{g^{-1}.\vpz}_w$.
The quantity $\minht(Y)$ measures how far up the cusp the homogeneous set $Y$ is.

\subsubsection*{Comparing volume and complexity}

The rest of this section is dedicated to the proof of the following proposition.

\begin{proposition}\label{prop: vol complexity}
There exist a constant $\consta\label{a:appmain}>1$ depending only on $N$ such that
\begin{align}\label{eq:volvscpl}
\cpl(Y)^{1/\ref{a:appmain}}
\ll \vol(Y) 
\ll\minht(Y)^{\ref{a:appmain}}\cpl(Y)^{\ref{a:appmain}}
\end{align}
where the implicit constants depend only on $N$.
\end{proposition}

Note that Proposition~\ref{prop: vol complexity} implies Proposition~\ref{prop: vol complexity'} in light of the bound on $\minht(Y)$ obtained from Proposition~\ref{prop:smallvectors}.
Since the height in the cusp of $Y$ should, in principle, be dictated by the amount the translation occurs within the centralizer of the invariance group, it is conceivable that Proposition~\ref{prop: vol complexity} holds without the polynomial factor in $\minht(Y)$.

For the proof of Proposition~\ref{prop: vol complexity}, we shall need bounds on the minimal discriminant of a splitting field of $\Hbf$ provided by the following lemma.

\begin{lemma}\label{lem:heightofmaxtorus}
There exists $\consta\label{a:heightofmaxtorus}>0$ depending only on $N$ with the following property.
Let $\Lbf < \SL_N$ be a connected reductive $\Q$-subgroup.
Then there exists a maximal $\Q$-torus $\Tbf< \Lbf$ with
\begin{align*}
\height(\Tbf)\ll \height(\Lbf)^{\ref{a:heightofmaxtorus}}.
\end{align*}
Moreover, if $F$ is the (Galois) splitting field of $\Tbf$ then
\begin{align*}
|\disc(F)| \ll \height(\Tbf)^\star.
\end{align*}
\end{lemma}

We will use the following simple and well-known lemma.

\begin{lemma}\label{lem:polynotzero}
Let $f \in \C[x_1,\ldots,x_n]$ be a non-zero polynomial of degree $d$.
Then there exists $x \in \Z^n$ with $\sup_i |x_i| \leq d^n$ and $f(x) \neq 0$. 
\end{lemma}

\begin{proof}
We prove the claim by induction on the number of variables $n$.
If $n=1$, then $f$ has at most $d$ zeroes and the lemma is clear.

Suppose the lemma is known for all polynomials in $n-1$ variables and let $f$ be given.
Consider the polynomial
\begin{align*}
g(x_1,\ldots,x_n) = f(x_1,x_2+\alpha_2x_1,\ldots,x_n+\alpha_nx_1).
\end{align*}
for some $\alpha_2,\ldots,\alpha_n\in \Z$.
Then $g$ is of the form $g = \phi(\alpha_2,\ldots,\alpha_n)x_1^d + g_1$ for some non-zero polynomial $\phi$ of degree at most $d$ and $g_1$ of degree smaller than $d$ in $x_1$.
By assumption we may choose $\alpha_2, \ldots, \alpha_n \in \Z$ with $|\alpha_i| \leq d^{n-1}$ such that $\phi(\alpha_2,\ldots,\alpha_n)\neq 0$.
In particular, $g(x_1,0,\ldots,0)$ is a non-constant polynomial in $x_1$ and so we may choose $\alpha_1\in\Z$ with $|\alpha_1|\leq d$ and $g(x_1,0,\ldots,0) \neq 0$.
This proves the lemma with the point $(\alpha_1,\alpha_1\alpha_2,\ldots,\alpha_1\alpha_n)$.
\end{proof}

\begin{proof}[Proof of Lemma~\ref{lem:heightofmaxtorus}]
We may choose the vector $\vpz_\Lbf$ to be integral (i.e.~contained in $\wedge^{\dim(\Lbf)}\mathfrak{sl}_N(\Z)$) and primitive (i.e.~$\norm{\vpz_{\Lbf}}_\gep = 1$ for every prime $\gep$).
In particular, $\height(\Lbf) = \norm{\vpz_{\Lbf}}_\infty$.

We first prove the following claim.

\begin{claim*}
There exists a non-trivial $\ad|_{\lfrak}$-semisimple integral element $\wpz \in \lfrak = \Lie(\Lbf)$ with $\norm{\wpz}_\infty \ll \norm{\vpz_{\Lbf}}_\infty^\star$.
\end{claim*}

Note that $\norm{\vpz_{\Lbf}}_\infty$ is precisely the covolume of the lattice $\lfrak(\Z)$ in $\lfrak(\R)$.
By Minkowski's second theorem, there exist vectors $\vpz_1,\ldots,\vpz_{\dim(\lfrak)} \in \lfrak(\Z)$ which are linearly independent and which satisfy $\norm{\vpz_i} \ll \norm{\vpz_{\Lbf}}_\infty$ for every $i =1,\ldots,\dim(\lfrak)$.
Consider the proper subvariety of $\Abf^{\dim(\lfrak)}$ 
\begin{align*}
\Big\{(x_1,\ldots,x_{\dim(\lfrak)}): \sum_i x_i \vpz_i \text{ is } \ad|_{\lfrak}\text{-nilpotent}\Big\}
\end{align*}
which is defined by polynomials of degree $O_N(1)$.
By Lemma~\ref{lem:polynotzero} applied to any of these defining polynomials,
there exist $x_i \in \Z$, $|x_i| \ll 1$, such that $\wpz' = \sum_i x_i \vpz_i$ is not $\ad|_{\lfrak}$-nilpotent.
In particular, $\norm{\wpz'}_\infty \ll \norm{\vpz_{\Lbf}}_\infty^\star$.
Note that the center of $\lfrak$ trivially consists of $\ad|_{\lfrak}$-nilpotent (but not nilpotent) elements and so $\wpz'$ is not contained in the center.
The coefficients of the characteristic polynomial of $\wpz'$ are of size $\ll \norm{\wpz'}_\infty^\star$.
The Jordan decomposition then implies that the semisimple part $\wpz$ of $\wpz'$ is non-zero and satisfies $\norm{\wpz}_\infty\ll \norm{\wpz'}_\infty^\star$ which concludes the claim.

As any semisimple element is contained in the Lie algebra of a maximal torus, the centralizer $\lfrak'$ of the above constructed element $\wpz$ has the same absolute rank as $\lfrak$.
By the above bound on the size of $\wpz$, we have $\height(\lfrak') \ll \height(\lfrak)^\star = \height(\Lbf)^\star$. 
Note that $\lfrak'$ is a proper reductive subalgebra of $\lfrak$.
Overall, we may as well find a maximal subalgebra of $\lfrak'$ consisting of semisimple elements.
Thus, the first part of the lemma follows by induction on the dimension.

For the second part, recall that for any element $\wpz \in \mathfrak{sl}_N(\R)$ the coefficients of the characteristic polynomial are integral polynomials in the entries of $\wpz$ and have size $\ll \norm{\wpz}_\infty^\star$.
We apply this observation to an integral basis $(\wpz_i)_i$ of $\tfrak$ with $\norm{\wpz_i}_\infty \ll \height(\Tbf)^\star$. We obtain that for each $\wpz_i$ all eigenvalues are of size $\ll \height(\Tbf)^\star$ (e.g.~by Cauchy's bound on zeros of real polynomials).
Since the discriminant of a compositum of extensions is bounded in terms of the individual extensions, the splitting field of a $\Q$-torus $\Tbf<\SL_N$ has the required bound. 
\end{proof}

\begin{proof}[Proof of Proposition~\ref{prop: vol complexity}]
A version of \eqref{eq:volvscpl} was established in \cite[(B.2)]{EMMV}, namely
\begin{align}\label{eq:discvsvol}
\disc(Y)^\star \ll \vol(Y) \ll \disc(Y)^\star
\end{align}
where the notion of discriminant $\disc(Y)$ from \cite{EMMV} is equal to $ \cpl(Y)$ up to a factor depending only on $\Hbf$ (or rather the quasi-split inner forms of the factors of $\Hbf$).
More specifically, 
\begin{align}\label{eq:discvscpl}
\disc(Y) = \frac{\mathcal{D}(\Hbf)}{\mathcal{E}(\Hbf)}\cpl(Y)
\end{align}
for some numbers $\mathcal{D}(\Hbf) \geq 1$ and $0 < \mathcal{E}(\Hbf) \leq 1$.
In particular, $\disc(Y) \geq \cpl(Y)$ and thus $\cpl(Y)^\star \ll \vol(Y)$ follows.
Here, the notation is taken from \cite{EMMV} for the readers' convenience and, in particular, $\mathcal{D}(\Hbf)$ does not refer to data of a homogeneous set.
In the remainder of the proof we will show that
\begin{align*}
\mathcal{D}(\Hbf) \ll \minht(Y)^\star \cpl(Y)^\star
\quad \text{and} \quad
\mathcal{E}(\Hbf) \gg 1.
\end{align*}
We note that both of the quantities $\mathcal{D}(\Hbf),\mathcal{E}(\Hbf)$ depend only on the $\SL_N(\Q)$-conjugacy class of $\Hbf$ and not on $\Hbf$ itself.

To recall the definition of $\mathcal{D}(\Hbf)$ from \cite[(B.13)]{EMMV}, we repeat parts of the discussion at the beginning of \S\ref{sec:good place}.
As $\Hbf$ is simply connected, we may write $\Hbf = \Hbf_1\cdots\Hbf_k$ where the groups $\Hbf_i$ are simply connected $\Q$-almost simple $\Q$-subgroups of $\Hbf$.
Write $\Hbf_i = \Res_{F_i/\Q}\Hbf_i'$ where $F_i/\Q$ is a finite extension and $\Hbf_i'$ is an absolutely almost simple group over $F_i$.
Associated to $\Hbf_i'$ is a `least distorted' form of it:
let $\mathcal{H}_i'$ be the unique quasi-split inner form of $\Hbf_i'$ over $\Q$.
Let $L_i/F_i$ be the corresponding number field defined as in~\cite[\S0.2]{Pr-Volume}. 
That is, $L_i$ is the splitting field of~$\Hcal_i'$ except in the case where~$\Hcal_i'$ is a triality form of type~${}^6\mathsf{D}_4$ where it is a degree~$3$ subfield
of the degree~$6$ Galois splitting field with Galois group~$S_3$. 
Let $s(\mathcal{H}_i')$ be as defined in \cite[\S0.4]{Pr-Volume} which is an integer only depending on the root system.
Overall, we set
\begin{align*}
\mathcal{D}(\Hbf) = \prod_i \mathcal{D}(\Hbf_i),\
\mathcal{D}(\Hbf_i) = \big( D_{L_i/F_i}^{s(\mathcal{H}_i')} D_{F_i}^{\dim(\Hbf_i)}\Big)^{\frac{1}{2}}.
\end{align*}
Here, we write $D_F$ for the absolute value of the discriminant of a number field $F$ and, similarly, we set $D_{L/F}$ to be absolute value of the norm of the relative discriminant ideal for extensions $L/F/\Q$.

We claim that for every $\Hbf_i$
\begin{align}\label{eq:estimateD(H_i)}
\mathcal{D}(\Hbf_i) \ll \minht(Y)^\star\cpl(Y)^\star
\end{align}
and in particular
\begin{align}\label{eq:estimateD(H)}
\mathcal{D}(\Hbf) \ll \minht(Y)^\star\cpl(Y)^\star.
\end{align}
To see \eqref{eq:estimateD(H_i)}, note that $\mathcal{D}(\Hbf_i) \ll D_{L_i}^\star\ll D_{L_i'}^\star$ where $L_i'/\Q$ is the Galois closure of $L_i/\Q$.
Here, the extension $L_i'/\Q$ is the (Galois) splitting field of the quasi-split group $\Res_{F_i/\Q}\mathcal{H}_i'$.
As $\Res_{F_i/\Q}\mathcal{H}_i'$ is an inner form of $\Hbf_i$, any splitting field of $\Hbf_i$ contains $L_i'$ and the same is true for any $\SL_N(\Q)$-conjugate of $\Hbf_i$.
Note that the height of any factor of a semisimple $\Q$-group is controlled polynomially by the height of the $\Q$-group --- see e.g.~\cite[Lemma 8.6]{AW-realsemisimple}.
Combining this with Lemma~\ref{lem:heightofmaxtorus} we have 
\begin{align*}
D_{L_i'} \ll \min_{\gamma \in \SL_N(\Q)}\height(\gamma \iota(\Hbf_i)\gamma^{-1})^\star
\ll \min_{\gamma \in \SL_N(\Q)}\height(\gamma \iota(\Hbf)\gamma^{-1})^\star.
\end{align*}
Let $[g]\in Y$ be a coset achieving the minimal height of $Y$ in the cusp.
We may assume $g\in \SL_N(\R)$ by strong approximation.
By reduction theory, there exists $\gamma \in \SL_N(\Z)$ such that 
\begin{align*}
\norm{\gamma g}_{\infty},\norm{(\gamma g)^{-1}}_{\infty} \ll \big(\min_{\vpz \in \Z^n \setminus \{0\}} \norm{g^{-1}.\vpz}\big)^{-\star} = \minht(Y)^{\star}
\end{align*}
Thus,
\begin{align*}
D_{L_i'} \ll \height(\gamma^{-1} \iota(\Hbf)\gamma)^\star 
\ll \minht(Y)^{\star} \cpl(Y)^\star
\end{align*}
proving \eqref{eq:estimateD(H_i)} and hence \eqref{eq:estimateD(H)}.

The quantity $\mathcal{E}(\Hbf)$ is a local product of normalized cardinalities of certain finite groups associated to the quasi-split inner forms $\mathcal{H}_i'$. We refer to \cite[(B.15)]{EMMV} for an exact definition, but note that it may be expressed as reciprocal of the Dedekind zeta functions of the fields $L_i$ and of certain Dirichlet $L$-functions at integer values $\geq 2$ --- see \cite[Rem.~3.11]{Pr-Volume} or \cite[\S5.2]{MohammadiSalehiGolsefidy-vertices}.
In particular, $\mathcal{E}(\Hbf) \gg \zeta(2)^{-\star} \gg 1$.
Combining this with \eqref{eq:discvscpl} and \eqref{eq:estimateD(H)} yields
\begin{align*}
\cpl(Y) \leq \disc(Y) \ll \minht(Y)^\star \cpl(Y)^\star.
\end{align*}
In view of \eqref{eq:discvsvol}, the proposition follows.
\end{proof}
\end{appendix}

\bibliographystyle{plain}
\bibliography{papers}

\end{document}